\documentclass[11pt]{amsart}

\usepackage[
backend=biber,
style=alphabetic,
sorting=nyt,
isbn=false,
]{biblatex}
\renewbibmacro{in:}{}

\usepackage{graphicx,amssymb,epstopdf,subcaption,mathtools,tikz-cd,pb-diagram,thmtools, thm-restate,chngcntr,pinlabel,pifont, parskip}
\usepackage{amsfonts,amsthm}
\usepackage{mathrsfs}
  \usepackage{amsthm} 
  \usepackage{amsmath} 
  \usepackage{amssymb} 
  \usepackage{mathtools} 
   \usepackage{amsfonts}
  \usepackage{morefloats}
  \usepackage{bbm}
  \usepackage{tikz-cd}

  \usepackage{wrapfig}
   
  \usepackage[scr=boondox]{mathalfa}

 \usepackage{graphicx}
  \usepackage{xspace}
  \usepackage[all]{xy}
  \usepackage{pinlabel}
  \usepackage{enumerate}
  \usepackage{verbatim}
\usepackage[margin=1.3in]{geometry}
\usepackage{xargs} 
\usepackage{verbatim}

  \usepackage{hyperref}

  \newcommand{\p}[1]{\medskip \noindent \emph{#1}.}

  \newcommand{\calC}{\mathcal{C}}
  
  \newcommand{\calE}{\mathcal{E}}

  \newcommand{\calH}{\mathcal{H}}

  \newcommand{\calM}{\mathcal{M}}
  \newcommand{\calN}{\mathcal{N}}

  \newcommand{\calS}{\mathcal{S}}
  \newcommand{\calT}{\mathcal{T}}
  
  \newcommand{\calW}{\mathcal{W}}

  \newcommand{\ZZ}{\mathbb{Z}}
  \newcommand{\Z}{\mathbb{Z}}

  \newcommand{\cut}{\searrow}
  \newcommand{\calV}{\mathcal{V}}

  \newtheorem{theorem}{Theorem}[section]
  \newtheorem{proposition}[theorem]{Proposition}
  
  \newtheorem{lemma}[theorem]{Lemma}

  \newtheorem*{question*}{Question}

  \newtheorem{mainthm}{Theorem}
  
  \newtheorem{mainquestion}[mainthm]{Question}

  \newtheorem{altthm}{Theorem}

  \theoremstyle{definition}

  \newtheorem*{claim*}{Claim}

  \newtheorem*{answer*}{Answer}
  \newtheorem*{application*}{Application}
  \newtheorem*{notation*}{Notation}

  \theoremstyle{remark}
  
  \newtheorem*{remark*}{Remark}

  \newtheoremstyle{noparens}
    {}{}{\itshape}{}%
    {\bfseries}{.}{ }%
    {\thmname{#1}\thmnote{ {\mdseries #3}}}
\theoremstyle{noparens} 
\newtheorem{mainthm*}[mainthm]{Theorem} 

  \newcommand{\altref}[1]{Theorem~\ref{#1}}

  \DeclareMathOperator{\Mod}{Mod}
  \DeclareMathOperator{\PMod}{PMod}
  \DeclareMathOperator{\Ext}{Ext}

  \DeclareMathOperator{\Image}{im}

  \DeclareMathOperator{\Stab}{Stab}

  \DeclareMathOperator{\Fix}{Fix}
  
\DeclareMathOperator{\Homeo}{Homeo}
  
  \newcommand{\Aut}{\ensuremath{\operatorname{Aut}}\xspace} 
   
  \newcommand{\Inn}{\ensuremath{\operatorname{Inn}}\xspace} 

  \newcommand{\SL}{\ensuremath{\operatorname{SL}}\xspace} 
   
  \newcommand{\PSL}{\ensuremath{\operatorname{PSL}}\xspace} 
   
  \newcommand{\param}{{\mathchoice{\mkern1mu\mbox{\raise2.2pt\hbox{$
  \centerdot$}}
  \mkern1mu}{\mkern1mu\mbox{\raise2.2pt\hbox{$\centerdot$}}\mkern1mu}{
  \mkern1.5mu\centerdot\mkern1.5mu}{\mkern1.5mu\centerdot\mkern1.5mu}}}

  \renewcommand{\setminus}{{\smallsetminus}}

  \DeclarePairedDelimiter\abs{\lvert}{\rvert}

  \newcommand{\xdownarrow}[1]{%
  {\left\downarrow\vbox to #1{}\right.\kern-\nulldelimiterspace}}

\definecolor{cyan1}{RGB}{0,255,255}
\definecolor{dustyblue}{RGB}{55,171,200}
\definecolor{dustypink}{RGB}{224, 157, 205}
\definecolor{dustygreen}{RGB}{98, 196, 144}
\definecolor{dustyorange}{RGB}{200,150,25}

\DeclareMathOperator{\ab}{ab}

\DeclareMathOperator{\Int}{Int}
\newcommand{\defn}[1]{\textbf{#1}}
\newcommand{\extremal}{\calE_\Phi}
\newcommand{\extremalarg}[1]{\calE_\Phi^{#1}}

\newcommand{\cent}[2]{C_{#2}\left({#1}\right)}

\DeclareMathOperator{\B}{B}
\newcommand{\braid}[1]{\B_{#1}}

\DeclareMathOperator{\UConf}{UConf}
\newcommand{\conf}[1]{\UConf_{#1}(\mathbb{C})}
\newcommand{\disk}[1]{\mathbb{D}_{#1}}

\DeclareMathOperator{\CRS}{CRS}
\newcommand{\crs}[1]{\CRS\left(#1\right)}

\DeclareMathOperator{\Aff}{Aff}

\newcommand{\extgen}[1]{{\mathcal{SG}}_{#1}}

\newcommand{\sphere}[1]{\mathbb{S}_{#1,1}}
\newcommand{\outmulti}{\circ}

\DeclareMathOperator{\Comm}{Comm}
\DeclareMathOperator{\orb}{orb}

\newcommand{\diskinner}[1]{\mathbb{I}_{#1}}

\usepackage{textcomp,gensymb}
\title{Rigidity of maps between configuration spaces}

  \author{Rodrigo De Pool}
\address{Department of Mathematics, University of Notre Dame, South Bend, Indiana, United States of America}
 \email{rdepoola@nd.edu}

 \author{Peter Huxford}
\address{Department of Mathematics, Rice University, Houston, Texas, United States}
 \email{hux@rice.edu}
 
\author{Daniel Minahan}\thanks{DM is partially supported by NSF Grant 2402060}
\address{Department of Mathematics, University of Chicago, Chicago, Illinois, United States of America}
 \email{dminahan@uchicago.edu}

 \author{Jeroen Schillewaert}\thanks{JS is supported by the New Zealand Marsden Fund through grant UOA-2122}
 \address{Department of Mathematics, University of Auckland, Auckland, New Zealand}
 \email{j.schillewaert@auckland.ac.nz}
\begin{document}

\begin{abstract}
Let $n\geq5$ and $m\geq3$. Let $\Phi\colon\braid{n}\to\braid{m}$ be a homomorphism of braid groups. We prove that if the image of $\Phi$ is irreducible and not cyclic, then $m=n$ and $\Phi$ agrees with an automorphism modulo the center $Z(\braid{m})$. This resolves in the affirmative a conjecture of Chen, Kordek, and Margalit. It also provides a partial resolution to a problem on the K3 problem list.  As a consequence, we prove that every holomorphic map $\conf{n} \to \conf{m}$ for $n \geq 5$ and $m \geq 3$ is affine equivalent to either a constant map or the identity map. This resolves a conjecture of Farb for $n\neq4$.
\end{abstract}

\maketitle

\section{Introduction}\label{section:intro}

Let $n\geq3$ and let $\mathfrak{S}_n$ be the symmetric group on $n$ letters.  Let
\[
  \conf{n} \coloneqq \{(z_1,\ldots,z_n) \in \mathbb{C}^n : z_i \neq z_j \text{ when } i\neq j\}\, \big/\, \mathfrak{S}_n.
\]
This is the \defn{configuration space} of $n$ unordered points in the complex plane $\mathbb{C}$. Note that $\conf{n}$ is a complex manifold. Let $\Aff\cong\mathbb{C}\rtimes\mathbb{C}^*$ be the group of affine transformations of $\mathbb{C}$.  The group $\Aff$ acts on $\conf{m}$ in a natural way. Two holomorphic maps $\conf{n}\to\conf{m}$ are \defn{affine equivalent} if they become equal after post-composing with the topological quotient map $\conf{m}\to\conf{m}/\Aff$.  We prove the following rigidity theorem.

\begin{mainthm}\label{mainthm:holo}
  Let $n \geq 5$ and $m\geq 3$.  Let  $\Psi:\conf{n} \rightarrow \conf{m}$ be a holomorphic map that is not affine equivalent to a constant map. Then $m=n$ and $\Psi$ is affine equivalent to the identity.
\end{mainthm}

\p{Ferrari map} The lower bound on $n$ is sharp, due to the existence of a map $\conf{4}\to\conf{3}$ dating back to Ferrari in 1540. His method that solves for the roots of a quartic polynomial begins with a construction that can be regarded as a holomorphic map from the space of quartic polynomials to the space of cubic polynomials.  Remarkably, if the initial quartic polynomial has no repeated roots, i.e., is square-free, then the resulting cubic polynomial will also be square free.  Thus, Ferrari's construction defines a holomorphic map $R\colon\conf{4}\to\conf{3}$.

Farb conjectured~\cite[Conjecture 2.1]{Far23} that for $n\geq4$ and $m\geq3$, any holomorphic map $\conf{n}\to\conf{m}$ must, up to affine equivalence, be constant, the identity, or $n=4$ and the map factors through $R$. Theorem~\ref{mainthm:holo} resolves this conjecture for $n\neq 4 $.

\p{Braid groups} A crucial tool in the study of $\conf{n}$ is the fundamental group
\[
  \braid{n}=\pi_1(\conf{n}),
\]
called the \defn{braid group} on $n$ strands. The group $\braid{n}$ is isomorphic to the mapping class group $\Mod(\disk{n})$ of the $n$-times punctured disk $\disk{n}$~\cite[Theorem 1.1]{BirmanBrendle}. We say that a homomorphism $\Phi:\braid{n}\to\braid{m}$ is \defn{irreducible} if $\Phi(\braid{n})$ does not preserve any non-empty multicurve in $\disk{m}$.  If $\Psi:\conf{n} \rightarrow \conf{m}$ is a holomorphic map that is not affine equivalent to a constant map, then the pushforward $\Psi_*:\braid{n} \rightarrow \braid{m}$ is irreducible as a consequence of the work of the first author and Souto~\cite[Theorem 1.2]{DPS24} which extends work of Daskalopalous--Wentworth~\cite[Theorem 5.7]{DW07}.

\p{Central equivalence} The analog of affine equivalence in the setting of braid groups is the notion of central equivalence. Let $G$ be a group and let $Z(G)$ denote its center.  If $\Phi:\braid{n} \rightarrow G$ is a homomorphism, then let $\overline{\Phi}\colon\braid{n}\to G/Z(G)$ denote the composition of $\Phi$ with the quotient map $G \rightarrow G/Z(G)$.  Two homomorphisms $\Phi_1,\Phi_2\colon\braid{n}\to G$ are \defn{centrally equivalent} if there is some $\psi\in\Aut(G)$ for which $\overline{\Phi_1} = \overline{\psi\circ\Phi_2}$.

The bulk of this paper is devoted to proving the following result.    

\begin{mainthm}\label{mainthm:natleast5}
  Let $n \geq 5$ and $m\geq 3$.  Let  $\Phi\colon\braid{n}\to\braid{m}$ be an irreducible homomorphism with non-cyclic image. Then $m=n$ and $\Phi$ is centrally equivalent to the identity.
\end{mainthm}

Chen--Salter, alongside many other results about holomorphic maps between other configuration spaces, show that Theorem~\ref{mainthm:natleast5} implies Theorem~\ref{mainthm:holo}~\cite[Theorem 1.5]{ChenSalter2026}. Theorem~\ref{mainthm:natleast5} was formulated as a question by Chen--Kordek--Margalit~\cite[Question 1.5]{CKM} and then later as a conjecture by the same authors~\cite[pg. 70]{CKMIcerm}.  Theorem~\ref{mainthm:natleast5} was also later conjectured by Chen--Salter~\cite[Conjecture 1.4]{ChenSalter2026}, and this conjecture was included in the K3 problems list \cite[Problem 2.7, Remark (2)]{K3problems}.  Problem 2.7 on the K3 problem list asks to classify all homomorphisms $\Phi:\braid{n} \rightarrow \braid{m}$.  Our Theorem \ref{mainthm:natleast5} resolves Problem 2.7 for all $n \geq 5$ and for $\Phi$ irreducible.  The pushforward $R_*:\braid{4} \rightarrow \braid{3}$ of Ferrari's map $R$ is irreducible and has non-cyclic image, so the bound $n \geq 5$ is also sharp in Theorem~\ref{mainthm:natleast5}.

\p{Prior results on homomorphisms of braid groups} Dyer--Grossman proved that $\Aut(\braid{n}) \cong \Inn(\braid{n}) \rtimes \ZZ/2\ZZ$~\cite{DyerGrossman}. In unpublished work, Lin showed that Theorem~\ref{mainthm:natleast5} holds when $m < n$~\cite[Theorem 3.1]{Lin2004braid}.  Bell--Margalit classify injective endomorphisms $\braid{n} \hookrightarrow \braid{n}$ for $n \geq 4$~\cite[Main Theorem 1]{BellMargalit}.  They further classify injective homomorphisms $\braid{n} \hookrightarrow \braid{n+1}$ for $n \geq 7$~\cite[Main Theorem 2]{BellMargalit}.  Castel generalized Bell--Margalit's work to the case that $n \geq 6$ and $\Phi:\braid{n} \rightarrow \braid{n+1}$ is arbitrary~\cite[Theorem 1.1.1]{Cas16}. Chen--Kordek--Margalit~\cite[Theorem 1.1]{CKM} classify all homomorphisms $\Phi:\braid{n} \rightarrow\braid{m}$ with $n \geq 5$ and $m \leq 2n$.

\p{Prior work on holomorphic maps between configuration spaces} In unpublished work, Lin proved Theorem~\ref{mainthm:holo}~\cite[Theorems 1.4 and 8.8]{Lin2004holo} for $m \leq n$ and $n \geq 5$.  Lin also showed that if $n\neq4$ and $n(n-1)\nmid m(m-1)$, then any holomorphic map $\Psi:\conf{n} \rightarrow \conf{m}$ is affine equivalent to a constant map~\cite[Corollary 9.9]{Lin2004holo}.  Chen--Kordek--Margalit upgrade this result to the case that $n \geq 5$ and either $n\nmid m$ and $n\nmid m-1$, or $n-1\nmid m$ and $n-1\nmid m-1$~\cite[Theorem 1.4]{CKM}.  Chen--Salter~\cite[Theorem 3.1]{ChenSalter2026} prove Theorem~\ref{mainthm:holo} in the case that $n \geq 5$ and $n \leq 2m$.  The second and fourth authors have classified, for $m,n\in\{3,4\}$, the holomorphic maps $\conf{n}\to\conf{m}$ up to affine equivalence~\cite[Theorem 1.1]{HuxfordSchillewaert2025},  This is done by classifying the corresponding homomorphisms $\braid{n}\to\braid{m}$.

\p{The assumptions of holomorphicity and irreducibility} The hypotheses of holomorphicity in Theorem~\ref{mainthm:holo}, and irreducibility in Theorem~\ref{mainthm:natleast5}, are necessary. Indeed, there is a continuous map $\conf{n}\to\conf{n+1}$ given by appending $1+\sum_{i=1}^n\abs{z_i}$ to the configuration $\{z_1,\ldots,z_n\}\in\conf{n}$. This induces the standard inclusion homomorphism $\braid{n}\to\braid{n+1}$, which is reducible: therefore this continuous map is not homotopic to a holomorphic map. As part of their classification of homomorphisms $\braid{n}\to\braid{m}$ for $n\geq5$ and $m\leq2n$, Chen--Kordek--Margalit construct more complicated examples of reducible homomorphisms $\braid{n}\to\braid{2n}$ called $k$-twist cablings \cite{CKM}. We believe that classifying all homomorphisms $\braid{n}\to\braid{m}$, which is equivalent to classifying the homotopy classes of continuous maps $\conf{n}\to\conf{m}$, for $n\geq5$ is an interesting problem that our Theorem~\ref{mainthm:natleast5} may be a stepping stone towards solving.

\p{The assumption $m\geq3$} It is possible to define $\conf{m}$ and $\braid{m}$ for $m\in\{1,2\}$. However Theorem~\ref{mainthm:holo} and Theorem~\ref{mainthm:natleast5} are trivial in these cases. Indeed, $\Aff$ acts 2-transitively on $\mathbb{C}$, and $\braid{2}\cong\ZZ$, whereas $\braid{1}$ is trivial. So if $n\geq1$ and $m\in\{1,2\}$, then all holomorphic maps $\conf{n}\to\conf{m}$ are affine equivalent to a constant map, and all homomorphisms $\braid{n}\to\braid{m}$ have cyclic image. For these reasons we generally assume that both $n\geq 3$ and $m\geq3$.

\p{On the notions of equivalence} Chen--Salter~\cite{ChenSalter2026} use a strictly finer notion of equivalence of holomorphic maps than our notion of affine equivalence.  However, it turns out that in the case that $n \geq 5$ and  $\Psi_1, \Psi_2:\conf{n} \rightarrow \conf{n}$ are two affine equivalent holomorphic maps that are not affine equivalent to a constant map, then $\Psi_1$ and $\Psi_2$ are equivalent under Chen--Salter's notion of equivalence. Furthermore, the holomorphic maps $\conf{n}\to\conf{m}$ that are affine equivalent to a constant map are well understood by Chen--Salter~\cite[Section 3.3]{ChenSalter2026}.  We discuss this more in Section~\ref{section:holomorphic}. Similarly, homomorphisms between braid groups are often said to be equivalent up to transvection \cite{CKM}.  We will discuss this notion of equivalence more in Section \ref{section:external}.

\p{Consequences of Theorem~\ref{mainthm:natleast5} for maps between hyperelliptic loci} Let $\calM_{g,r}$ denote the moduli space of genus $g\geq0$ Riemann surfaces with $r\geq0$ marked points, where $2g+r\geq3$. We view $\calM_{g,r}$ as a complex analytic orbifold. For $g \geq 1$, let $\calH_{g,1}$ be the hyperelliptic locus in $\calM_{g,1}$, i.e., the moduli space of genus $g\geq1$ hyperelliptic curves equipped with a point that is fixed by some hyperelliptic involution. Note that this is \emph{not} the universal family of genus $g$ hyperelliptic curves. Rather when $g\geq2$ it is a degree $2g + 2$ cover of the hyperelliptic locus $\calH_g$ in $\calM_g$.  Using Theorem~\ref{mainthm:natleast5}, we prove the following.

\begin{mainthm}\label{mainthm:hyp-to-hyp}
    Let $g\geq2$ and $h\geq1$.  Let $\Psi\colon\calH_{g,1}\to\calH_{h,1}$ be a holomorphic map such that $\Psi$ is non-constant as a map of coarse moduli spaces. Then $g=h$ and $\Psi$ agrees with the identity when viewed as a map between the coarse moduli spaces.
\end{mainthm}  

\p{Holomorphic maps between moduli spaces} Theorems~\ref{mainthm:holo} and~\ref{mainthm:hyp-to-hyp}  are examples of results that classify the holomorphic maps between parameter or moduli spaces that arise in algebraic geometry. In Section~\ref{section:holomorphic} we leverage Theorem~\ref{mainthm:natleast5} to classify maps between other moduli spaces.  Farb has assembled many pre-existing results and formulated a number of conjectures in this area~\cite{Far23}. Chen--Salter's work~\cite{ChenSalter2026}, which includes a suite of results classifying holomorphic maps between configuration spaces of points on Riemann surfaces, is closely related to ours and also belongs to this general program.  We have the following question, which we believe to be of general interest.

\begin{mainquestion}\label{question:moduli}
Let $g,r, h,s \geq 0$ with $2g+r\geq3$ and $2h+s\geq3$.  What are the holomorphic maps $\calM_{g,r} \rightarrow \calM_{h,s}$?
\end{mainquestion}  This was asked by Farb for $r=s=0$~\cite[Question 4.5]{Far23}. Antonakoudis--Aramayona--Souto~\cite[Theorem 1.1]{AAS18} showed that for $g\geq6$ and $h\leq 2g-2$, the only non-constant holomorphic maps occur when $g=h$ and $s\leq r$.  Furthermore, these are all forgetful maps, i.e.\ maps that permute and possibly forget some of the marked points. The first author and Souto extended this to $g\geq4$ and $h\leq 3\cdot2^{g-3}$~\cite[Theorem 1.1]{DPS24}. Benirschke--Serv\'{a}n \cite[Corollary 1.4]{Servan} classify isometric embeddings $\calM_{g,r} \rightarrow \calM_{h,s}$ when $\dim(\calM_{g,r}) \geq 2$.

The work of Aramayona--Souto, and that of the first author and Souto, uses a strategy similar to ours.  They use group theoretic results, of Aramayona--Souto~\cite{AS12}, and the first author~\cite{DP25}, respectively, about the corresponding mapping class groups. Antonakoudis--Aramayona--Souto also ask whether the holomorphic maps $\calM_{g,r}\to\calM_{h,s}$ for $g\geq3$ realize all irreducible homomorphisms between the corresponding mapping class groups~\cite[Question 1.3]{AAS18}. This is an analog of Chen--Kordek--Margalit's conjecture~\cite{CKMIcerm} that our Theorem~\ref{mainthm:natleast5} resolves.

\p{Overview of the paper and notation} The bulk of the paper proves Theorem~\ref{mainthm:natleast5}.  There are also a few pieces of recurring notation that we use throughout the paper, which are all listed here for the convenience of the reader. Most importantly, $\Phi$ is exclusively used to denote a homomorphism $\braid{n}\to\braid{m}$ of braid groups, and much of the notation we introduce refers to objects defined in terms of $\Phi$.

Section~\ref{section:braidprelim} introduces notation and assembles fundamental results about the braid group. The notation $\extgen{n}=\{s_1,\ldots,s_n\}$ for the extended generating set of $\braid{n}$ is introduced here.  We introduce an indexing convention for the elements of $\extgen{n}$ that we use repeatedly.  We also introduce elements $a_1,a_2 \in \braid{n}$ that we use repeatedly.

In Section~\ref{section:crs} we discuss the Nielsen--Thurston classification for braid groups and prove some basic results about homomorphisms between braid groups.

In Section~\ref{section:external}, we show in Lemma~\ref{lem:external:anosov} that any irreducible non-cyclic homomorphism is externally periodic.  We also define the exterior part $\Ext_M(f)$ of a braid and the canonical exterior part $\Ext(f)$.  This is related to Chen--Kordek--Margalit's notion of exterior map~\cite[Section 3.1] {CKM}.

Section~\ref{section:alpha} establishes key properties of special curves in $\disk{m}$ that we call $\Phi$-maximal.  The notations $\extremal$ and $\extremalarg{s_i}$ for $\Phi$-maximal curves are introduced here.

Section~\ref{section:punctures} defines an invariant of punctures in $\disk{m}$ called type. In Theorem~\ref{thm:alpha:welltypedpuncture} we classify the possible types of punctures and describe how $\Phi(s_i)$ acts on punctures of different types.  The notation $\calT_{\Phi}(p)$ for the type of a puncture $p$ is introduced here.

Section~\ref{section:curves} proves Proposition~\ref{prop:curves:valid}, which is a technical result about curves in the interior of $\Phi$-maximal curves.  The notation $\Delta(\Phi)$ is introduced here.

Section~\ref{section:onemaximal} introduces the notion of a simple homomorphism, and proves in Lemma~\ref{lem:onemaximal:twopunc} that these homomorphisms satisfy a certain technical property that is useful in Section~\ref{section:extcent}.

Section~\ref{section:extcent} proves Theorem~\ref{thm:extcent}, which states that Theorem~\ref{mainthm:natleast5} holds if we additionally assume that the homomorphism $\Phi$ is externally central.

Section~\ref{section:mainthm} completes the proof of Theorem~\ref{mainthm:natleast5}.  The proof proceeds by proving Lemma~\ref{lem:mainthm:externallytrivial}, which says that an irreducible homomorphism is centrally equivalent to an externally central homomorphism. 

Section~\ref{section:holomorphic} derives holomorphic rigidity consequences of Theorem~\ref{mainthm:natleast5}, including Theorem~\ref{mainthm:holo} and Theorem~\ref{mainthm:hyp-to-hyp}.

\subsection*{Acknowledgments}
The authors would like to thank Dan Margalit for providing extensive comments on an earlier draft of the paper, and would like to thank Noah Caplinger, Benson Farb, Andrew Putman, Sidhanth Raman, Nick Salter, Juan Souto, and Katherine Williams Booth for helpful conversations and comments.

\section{Generators of the braid group}\label{section:braidprelim}

Let $n \geq 2$. The braid group on $n$ strands, denoted $\braid{n}$, is by definition $\pi_1(\conf{n})$. Artin~\cite{Artin} proved that $\braid{n}$ admits the following presentation:
\begin{align*}
    \braid{n} = \langle s_1, \ldots, s_{n-1} \mid \;& s_is_{i+1}s_i = s_{i+1}s_is_{i+1} \text{ for } 1 \leq i \leq n-2,\\ & [s_i, s_j]=1 \text{ for } |i-j| >1 \;\rangle.
\end{align*}
The elements $s_1, \ldots, s_{n-1}$ are said to be a \defn{standard generating set} for $\braid{n}$  and the relation $s_is_{i+1}s_i = s_{i+1}s_is_{i+1}$ is known as the \defn{braid relation}. It follows from the presentation that the abelianization of $\braid{n}$ is $\braid{n}^{\ab}\cong\mathbb{Z}$ and that $s_is_{i+1}^rs_{i}^{-1} = s_{i+1}^{-1} s_i^r s_{i+1}$ for every  $r\geq 1$ and $1\leq i\leq n-1$.  Furthermore, we have $s_i=(s_{i+1}s_i)s_{i+1}(s_{i+1}s_i)^{-1}$ and $s_{i+1}=(s_is_{i+1})s_i(s_is_{i+1})^{-1}$ by the braid relation.

\p{Automorphisms of braid groups} Let $n \geq 2$.  Dyer--Grossman~\cite{DyerGrossman} proved that there is a short exact sequence
\[
  1 \rightarrow \Inn(\braid{n}) \rightarrow \Aut(\braid{n}) \rightarrow \ZZ/2\ZZ \rightarrow 1
\]
where $\Inn(\braid{n})$ denotes the inner automorphisms of $\braid{n}$. This short exact sequence splits because there is an outer automorphism $\kappa:\braid{n} \rightarrow \braid{n}$ of order 2 defined by
\[
  \kappa(s_i) = s_{i}^{-1}
\]
for all $1 \leq i \leq n-1$.

\p{Mapping class groups} Let $S_{g,n}^b$ be an orientable surface of genus $g$ with $b$ boundary components and $n$ punctures. Throughout, we always assume that surfaces are orientable.  We omit $n$ or $b$ when they are zero. The group $\Homeo^+(S_{g,n}^b, \partial S_{g,n}^b)$ is the group of orientation preserving homeomorphisms of $S_{g,n}^b$ that fix the boundary pointwise.  The \defn{mapping class group} $\Mod(S_{g,n}^b)$ is the quotient of $\Homeo^+(S_{g,n}^b,\partial S_{g,n}^b)$ by isotopies that fix $\partial S_{g,n}^b$ pointwise. We write $\Mod_{g,n}^b$ for the group $\Mod(S_{g,n}^b)$, and again omit $n$ or $b$ when they are zero. For more on mapping class groups, see~\cite{FM}.

\p{Braid groups as mapping class groups} Let $\disk{n}=S_{0,n}^1$ denote the $n$-times punctured disk. By a slight abuse of notation, we refer to a puncture $p$ of the disk as a puncture $p \in \disk{n}$. Let $\alpha$ be an embedded arc between two punctures of $\disk{n}$, and let $D_\alpha\cong\disk{2}$ be a closed regular neighborhood of $\alpha$. The \defn{half-twist} about $\alpha$, or equivalently the half-twist about $\partial D_\alpha$, is the mapping class, supported on $D_\alpha$, that interchanges the two endpoints of $\alpha$ in a counterclockwise manner, and whose square is the Dehn twist $T_{\partial D_\alpha}$ (see Figure~\ref{fig:crs:halftwistexample}).

 There is an isomorphism between $\braid{n}\cong\Mod(\disk{n})$~\cite[Theorem 1.1]{BirmanBrendle}, which we now describe. We number the punctures of $\disk{n}$ with the integers $1,\ldots,n$.  Consider embedded arcs $\alpha_1,\ldots,\alpha_{n-1}$ in $\disk{n}$ whose interiors are pairwise disjoint such that the endpoints of $\alpha_i$ are the punctures numbered $i$ and $i+1$. One obtains an isomorphism $\braid{n}\to\Mod(\disk{n})$ by sending the generator $s_i$ to the half twist about the arc $\alpha_i$ for each $1\leq i \leq n-1$. Any other labeling of the punctures and different choices of arcs can be achieved by applying an orientation preserving homeomorphism of $\disk{n}$, so this isomorphism is well-defined up to inner automorphisms.

The half-twist corresponding to the standard generator $s_1 \in \braid{3}$, identified with an element of $\Mod(\disk{n})$ for a specific choice of arc $\alpha_1$, can be seen in Figure~\ref{fig:crs:halftwistexample}.

\begin{figure}[ht]
    \centering
    \begin{tikzpicture}
        \node[anchor = south west] at (0,0){\includegraphics[scale=0.6]{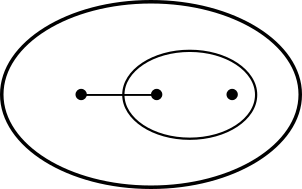}};
        \node at (4.2,1){\large $\gamma$};
        \node at (1.8,1.8){\large $\alpha_1$};
        \node at (6,1.6){\huge $\longrightarrow$};
        \node[anchor = south west] at (6.8,0){\includegraphics[scale=0.6]{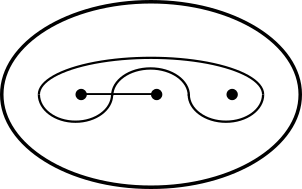}};
        \node at (10.8,1){\large $s_1(\gamma)$};
    \end{tikzpicture}
    \caption{The half-twist $s_1$ acting on a curve $\gamma$}\label{fig:crs:halftwistexample}
\end{figure}
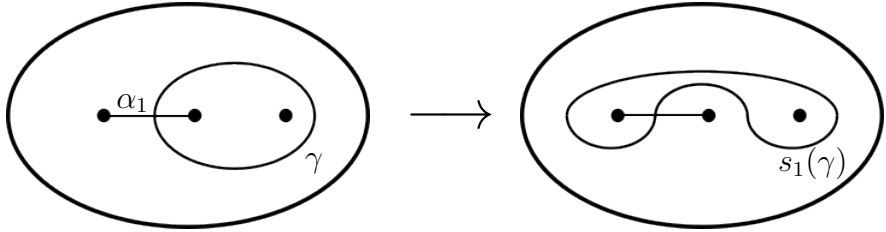

For the rest of the paper we fix a standard generating set $\{s_1, \ldots, s_{n-1}\}$ for $\braid{n}$ and an isomorphism $\braid{n}\cong\Mod(\disk{n})$ as above.  We use this isomorphism without comment. 

\p{Notational convention on composition} Our convention for the remainder of the paper is that composition is done functionally, i.e., if $f, f' \in \braid{n}$, then $f f' \in \braid{n}$ is the element of $\Mod(\disk{n})$ given by first applying $f'$ and then by applying $f$.

\p{Standard central roots} For $n\geq3$, the Dehn twist $T_{\partial \disk{n}}$ along the boundary component $\partial \disk{n}$ generates the center of the braid group $Z(\braid{n})=\langle T_{\partial \disk{n}}\rangle \cong \mathbb{Z}$~\cite[Theorem III]{Cho48}. Associated to the standard generating set $\{s_1, \ldots, s_{n-1}\} \subset \braid{n}$ are two special roots of $T_{\partial \disk{n}}$. First, we have
\[
  a_1 = s_1 \cdots s_{n-1}.
\]
This element is realized in $\Mod(\disk{n})$ by a rotation of all $n$ punctures (see Figure~\ref{fig:crs:nthrootexample}).
\begin{figure}[ht]
    \centering
    \begin{tikzpicture}
        \node[anchor = south west] at (0,0){\includegraphics[scale=0.6]{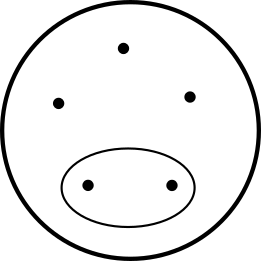}};
        \node at (2.2,2.1){\large $\gamma$};
        \node at (5.5,2.5){\huge $\longrightarrow$};
        \node[anchor = south west] at (6.8,0){\includegraphics[scale=0.6]{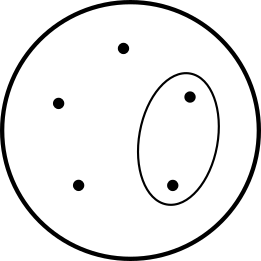}};
        \node at (8.6, 2.1){\large $a_1(\gamma)$};
    \end{tikzpicture}
    \caption{The element $a_1 \in \braid{5}$ acting on a curve $\gamma$}\label{fig:crs:nthrootexample}
\end{figure}
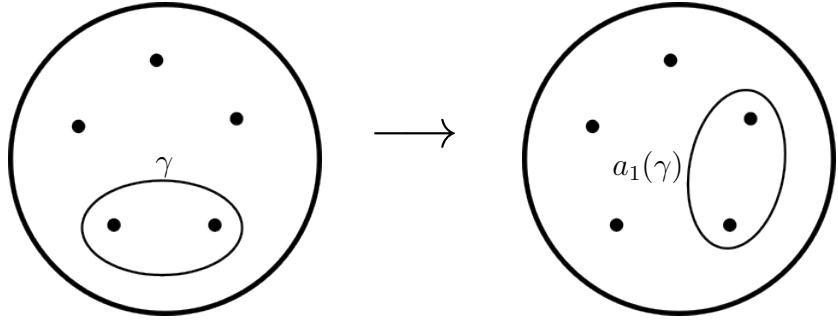
We call $a_1$ the \defn{standard central root of order $n$}.  Likewise, we have
\[
  a_2 = s_1^2 s_2 \cdots s_{n-1}.
\]
This is realized in $\Mod(\disk{n})$ by a rotation that fixes one puncture and rotates the other $n-1$ punctures (see Figure~\ref{fig:crs:nminusonerootexample}).
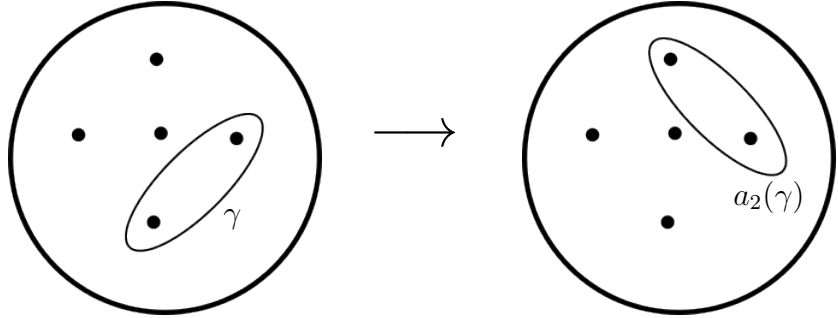
\begin{figure}[ht]
    \centering
    \begin{tikzpicture}
        \node[anchor = south west] at (0,0){\includegraphics[scale=0.6]{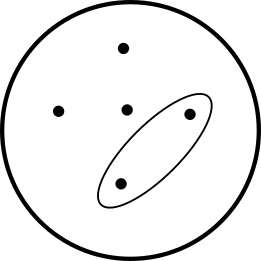}};
        \node at (3.1,1.4){\large $\gamma$};
        \node at (5.5,2.5){\huge $\longrightarrow$};
        \node[anchor = south west] at (6.8,0){\includegraphics[scale=0.6]{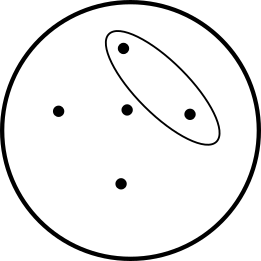}};
        \node at (10.2,1.7){\large $a_2(\gamma)$};
    \end{tikzpicture}
    \caption{The element $a_2 \in \braid{5}$ acting on a curve $\gamma$}\label{fig:crs:nminusonerootexample}
\end{figure}
We call the element $a_2$ the \defn{standard central root of order $n-1$}. These names are justified by the following observation:
\[a_1^n=a_2^{n-1}=T_{\partial\disk{n}}\in Z(\braid{n}).\]  We use the symbols $a_1$ and $a_2$ to denote these elements in $\braid{n}$ throughout the paper.

\p{Extended generating set} Let $s_n = a_1s_{n-1}a_1^{-1}$. The element $s_n $ is a half-twist that exchanges the punctures $n$ and $1$. Note that $a_1s_na_1^{-1} = s_1$. Note also that $s_n$ satisfies the following:
\begin{itemize}
    \item $s_ns_1s_n = s_1s_ns_1$ and $s_{n-1}s_ns_{n-1} = s_n s_{n-1} s_n$; and 
    \item $[s_n, s_i] = 1$ for all $2 \leq i \leq n-2$.
\end{itemize}The \defn{extended generating set} of $\braid{n}$ is defined to be the set $\extgen{n} \coloneqq \{s_1, \ldots, s_{n-1}, s_n\}$. We typically work with $\extgen{n}$, as opposed to the smaller standard generating set, in subsequent sections. However, we also note that $\extgen{n}\setminus\{s_i\}$ generates $\braid{n}$ for any $s_i\in\extgen{n}$.

\p{Indexing convention} For the sake of notational cleanliness, we define the element $s_k \in \extgen{n}$ for all $k \in \ZZ$.  In particular, we set $s_k = s_i$ for $i \equiv k \mod n$.  Under this indexing convention, we have
\[
  a_1s_ia_1^{-1}=s_{i+1} \qquad \text{and} \qquad s_is_{i+1}s_i=s_{i+1}s_is_{i+1},
\]
for all $i \in \ZZ$.  We repeatedly use this indexing convention throughout the remainder of the paper.

\p{Alternate expression for $a_1$} The central root $a_1\in \braid{n}$ admits different expressions using the indexing convention. If we conjugate the equality $a_1=s_1\ldots s_{n-1}$ by $a_1^{i}$, we obtain $a_1=s_{i+1}s_{i+2}\ldots s_{i+n-1}$ for every $i \in \ZZ$ using the above notational convention. 

\p{Commuting graph when $n \geq 5$}  Let $G$ be a group and let $S \subseteq G$ be a set of elements.  The \defn{commuting graph} $\Comm_G(S)$ is the graph where:
\begin{itemize}
    \item the vertex set is $V(\Comm_G(S)) = S$; and
    \item the edge set is $E(\Comm_G(S)) = \{\{s,s'\} \subseteq S: [s,s'] = 1,\ s\neq s'\}$.
\end{itemize} The graph $\Comm_{\braid{n}}(\extgen{n})$ is isomorphic to the complement of the $n$-cycle graph. The key reason that many of our arguments hold when $n \geq 5$ but fail when $n\in\{3,4\}$ is due to the following three properties of $\Comm_{\braid{n}}(\extgen{n})$, all of which hold when $n \geq 5$:
\begin{itemize}
    \item \emph{Connectivity:} $\Comm_{\braid{n}}(\extgen{n})$ is connected when $n \geq 5$;
    \item \emph{Small diameter:} $\Comm_{\braid{n}}(\extgen{n})$ has diameter 2 when $n \geq 5$; and
    \item \emph{Large neighborhoods:} for each two element subset $\{s_i,s_j\}\subset\extgen{n}$, there is at most one vertex $s_k \in \extgen{n}$ not adjacent to either one of $s_i$ or $s_j$ when $n \geq 5$.
\end{itemize}  This is closely related to the notion of totally symmetric sets due originally to Kordek--Margalit~\cite{KordekMargalit}.  We will refer to each of these facts about $\Comm_{\braid{n}}(\extgen{n})$ by name rather than citation. 
  
\p{Collision implies collapse} We make use of the following principle observed by Formanek.  See also the work of Kordek--Margalit~\cite{KordekMargalit} and Caplinger--Salter~\cite{CaplingerSalter} on totally symmetric sets.  Let $\Phi:\braid{n} \rightarrow \braid{m}$ be a homomorphism.  We say that $\Phi$ is \defn{cyclic} if the image of $\Phi$ is a cyclic subgroup of $\braid{m}$. 

\begin{lemma}[{\cite[Corollary 2]{Formanek}}]\label{lem:crs:collisioncollapse}
Let $n \geq 5$.  Let $G$ be a group and let $\Phi\colon\braid{n}\to G$ be a homomorphism. If $s_i, s_j \in \extgen{n}$ with $s_i \neq s_j$ and $\Phi(s_i)=\Phi(s_j)$, then $\Phi$ is cyclic.
\end{lemma}

\begin{proof}
    Formanek's result assumes additionally that $s_n\notin\{s_i,s_j\}$. However, without loss of generality we can assume we are in this situation by conjugating $s_i$ and $s_j$ by an appropriate power of $a_1$.  By~\cite[Corollary 2]{Formanek} the result holds.
\end{proof}

Note that the lemma is not true for $n=4$.  Indeed, there is a map $\braid{4} \rightarrow \braid{3}$ given by sending $s_1 \mapsto s_1$, $s_2 \mapsto s_2$ and $s_3 \mapsto s_1$. This is the pushforward $R_*$ of the Ferrari map $R: \conf{4} \rightarrow \conf{3}$ discussed in Section~\ref{section:intro}.  It turns out that the lemma holds when $n = 3$ since $\braid{3}$ is generated by $s_1,s_2$.

\p{Cyclic homomorphisms} Let $G$ be a group and $H \subseteq G$ a subset. 
We denote the centralizer of $H$ in $G$ by $\cent{H}{G}$.

\begin{lemma}\label{lem:crs:order}
Let $n \geq 5$ and $m \geq 3$.  Let $\Phi:\braid{n} \rightarrow \braid{m}$ be a homomorphism.  Suppose that at least one of the following holds:
\begin{itemize}
    \item there is an integer $1 \leq k \leq n-1$ such that $\Phi(a_1^k) \in \cent{\Phi(\braid{n})}{\braid{m}}$; or
    \item there is an integer $1 \leq k\leq n-2$ such that $\Phi(a_2^k)\in \cent{\Phi(\braid{n})}{\braid{m}}$.
\end{itemize}
Then $\Phi$ is cyclic.
\end{lemma} Note that this is false for $n = 4$.  For example, the Ferrari map $R_*\colon\braid{4}\to\braid{3}$ discussed previously is non-cyclic but satisfies $R_*(a_1^2)\in Z(\braid{3})$.

\begin{proof}[Proof of Lemma~\ref{lem:crs:order}]
    Suppose that $\Phi(a_1^k)\in \cent{\Phi(\braid{n})}{\braid{m}}$ for some $1 \leq k \leq n-1$. Then
    \[
      \Phi(s_{k+1}) = \Phi(a_1^ks_1a_1^{-k}) = \Phi(a_1^k)\Phi(s_1)\Phi(a_1^k)^{-1} = \Phi(s_1).
    \]
    By Lemma~\ref{lem:crs:collisioncollapse} the map $\Phi$ is cyclic. 
    
    Now suppose that $\Phi(a_2^k)\in \cent{\Phi(\braid{n})}{\braid{m}}$ for some $1\leq k\leq n-2$. Recall that $a_2^{n-1}\in Z(\braid{n})$. By possibly replacing $a_2^k$ with $a_2^{n-1-k}$, we may assume without loss of generality that in fact $1 \leq k \leq n-3$, so that $a_2^ks_2a_2^{-k}=s_{k+2}$. Hence
    \[
    \Phi(s_{k+2}) = \Phi(a_2^ks_2a_2^{-k}) = \Phi(a_2^k)\Phi(s_2)\Phi(a_2^k)^{-1} = \Phi(s_2).
    \]
    Since $1 \leq k \leq n-3$ we have $s_2 \neq s_{k+2}$.  By Lemma~\ref{lem:crs:collisioncollapse} the map $\Phi$ has cyclic image.
\end{proof}

\section{Nielsen--Thurston classification}\label{section:crs} Our goal now is to describe the Nielsen--Thurston classification for braid groups and some consequences of this classification for homomorphisms between braid groups.

\p{Preliminaries on curves} Hereafter, a \defn{curve} in $S_{g,n}^b$ refers to an isotopy class of essential simple closed curves unless otherwise specified. Recall that a curve is \defn{inessential} if it is homotopic to a boundary component or to a point (including a puncture). Otherwise we say that the curve is \defn{essential}.  We will say that two curves $\gamma$ and $\delta$ are \defn{distinct} if no representative of $\gamma$ is isotopic to any representative of $\delta$. An \defn{arc} is a homotopy class of embedded intervals in $S_{g,n}^b$ where the endpoints of the interval map into punctures or boundary components of $S_{g,n}^b$.  Here we require homotopies to fix the punctures and boundary components of $S_{g,n}^b$.  If $\gamma,\delta\subset S_{g,n}^b$ are curves, we denote their \defn{geometric intersection number} by $\iota(\gamma, \delta)\in\ZZ_{\geq0}$, which is the minimum number of intersections over all representatives of $\gamma$ and $\delta$. We say $\delta$ and $\gamma$ are  \defn{disjoint} if they have disjoint representatives, i.e. $\iota(\gamma, \delta)=0$. Otherwise, we say $\delta$ and $\gamma$ \defn{intersect}. Note that a curve $\gamma$ is always disjoint from itself.  If $\calN$ is a nonempty set of curves, we say that a curve $\gamma$ is disjoint from $\calN$ if $\gamma$ is disjoint from every $\gamma' \in \calN$. Likewise, two sets of curves $\calN$ and $\calN'$ are \defn{disjoint} if each $\gamma \in \calN$ is disjoint from $\calN'$.  A \defn{multicurve} $M$ is a set of pairwise disjoint curves.  An Euler characteristic argument implies that multicurves are always finite.  For a more thorough discussion, see~\cite{FM}.

\p{Canonical reduction systems} Let $f \in \Mod_{g,n}^b$.  A \defn{reducing system} for $f$ is a multicurve $M$ in $S_{g,n}^b$ such that $f(M) = M$.  Let $\calV_f$ denote the set of all maximal reducing systems $M$ of $f$, i.e.\ the set of all $M$ such that there is no reducing system $M'$ of $f$ with $M' \supsetneq M$. The \defn{canonical reduction system}~\cite[pg. 373]{FM} of $f$ is the intersection
\[
  \crs{f} \coloneqq \bigcap_{M \in \calV_f} M.
\]
Note that if $f \in \braid{n}$, then $\partial \disk{n}$ is not in $\crs{f}$ as $\partial \disk{n}$ is inessential.  An element $f \in \Mod_{g,n}^b$ is \defn{reducible} if it admits a non-empty reducing system, otherwise it is \defn{irreducible}.  See Figure~\ref{fig:reducingex} and Figure~\ref{fig:periodicreducingex} for examples of reducing systems.

\begin{figure}[ht]
\centering
\begin{minipage}{.45\textwidth}
\centering
\begin{tikzpicture}
    \node at (0,0)[anchor = south west]{\includegraphics[scale=0.5]{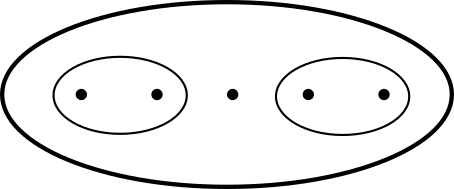}};
\end{tikzpicture}
\caption{The canonical reduction system of $s_1s_4 \in \braid{5}$.  This is also a reducing system for $s_1$ and $s_4$.}\label{fig:reducingex}
\end{minipage}
\begin{minipage}{.45\textwidth}
\centering
\begin{tikzpicture}
    \node at (0,0)[anchor = south west]{\includegraphics[scale=0.5]{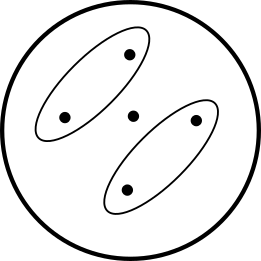}};
\end{tikzpicture}
\caption{A reducing system for $a_2^2 \in \braid{5}$.  Note $\crs{a_2^2} = \emptyset$.}\label{fig:periodicreducingex}
\end{minipage}
\end{figure}

As observed in~\cite[Section 13.2.2]{FM}, Birman--Lubotzky--McCarthy's definition of the canonical reduction system is equivalent to the one provided above~\cite{BirmanLubotzkyMcCarthy}. They define $\crs{f}$ to be the collection of curves $c$ satisfying the following two properties:
\begin{itemize}
    \item the $f$-orbit of $c$ is finite; and
    \item the $f$-orbit of any curve $b$ with $\iota(b,c) > 0$ is infinite.  
\end{itemize} The following standard facts follow from either Farb--Margalit's definition of the canonical reduction system, or from the definition due to Birman--Lubotzky--McCarthy~\cite{BirmanLubotzkyMcCarthy}. We state them without proof.

\begin{lemma}\label{lem:crs:infdisj}
Let $f \in \Mod_{g,n}^b$ and $\delta \subset S_{g,n}^b$ a curve.  Suppose that $\crs{f}$ intersects $\delta$. Then the $f$-orbit of $\delta$ is infinite.
\end{lemma}

\begin{lemma}\label{lem:crs:functorial}
Let $f,h \in \Mod_{g,n}^b$.  Then $\crs{hfh^{-1}} = h(\crs{f})$.
\end{lemma}

A particularly useful consequence of Lemma~\ref{lem:crs:functorial} is that if $f$ and $h$ satisfy the braid relation $fhf=hfh$, then $hf\crs{h} = \crs{f}$.

\begin{lemma}\label{lem:crs:commdisj}
Let $f, h \in \Mod_{g,n}^b$.  Suppose that $f$ and $h$ commute.  Then $\crs{h}$ is a reducing system for $f$ and the curves in $\crs{f}\cup\crs{h}$ are pairwise disjoint.  In particular, $\crs{f} \cup \crs{h}$ is a multicurve.
\end{lemma}

\begin{lemma}\label{lem:crs:power}
Let $f \in \Mod_{g,n}^b$ and $k \in \Z$ a non-zero integer.  Then $\crs{f^k} = \crs{f}$.
\end{lemma}

\p{Pseudo-Anosov mapping classes} Suppose that $\phi\in\Homeo^+(S_{g,n})$ leaves two transverse foliations invariant. Suppose further that the foliations are measured, and that there is a real number $\lambda>1$ such that $\phi$ scales one by $\lambda$ and the other by $\lambda^{-1}$. A mapping class $f\in\Mod_{g,n}$ is \defn{pseudo-Anosov} if it is represented by such a homeomorphism $\phi$. It follows readily from the definition that if $f\in\Mod_{g,n}$ is pseudo-Anosov, then $f$ is irreducible. For a detailed discussion of pseudo-Anosov mapping classes, we refer the reader to~\cite[Chapter 13--14]{FM}.

\p{Capping} Let $\sphere{m}$ denote the 2-sphere $S_{0,m+1}$ with $m+1$ punctures, where one of the punctures is distinguished. Let $\Mod(\sphere{m})$ denote the index $(m+1)$ subgroup of $\Mod(S_{0,m+1})$ that fixes the distinguished puncture. There is a map $\disk{m}\to\sphere{m}$ given by capping the unique boundary component of $\disk{m}$ with a once-punctured disk. This induces a short exact sequence~\cite[Theorem 3.18]{FM}
\[
1 \to \langle T_{\partial\disk{m}} \rangle \to \Mod(\disk{m}) \to \Mod(\sphere{m}) \to 1.
\]
Recall that $Z(\braid{m})=\langle T_{\partial\disk{m}}\rangle$, so the map $\Mod(\disk{m})\to\Mod(\sphere{m})$ above can be identified with the quotient map $\braid{m}\to\braid{m}/Z(\braid{m})$. See \cite[Section 3.6.2]{FM} for more on capping.

We say that $f\in\Mod(\disk{m})$ is \defn{pseudo-Anosov} if its image in $\Mod(\sphere{m})$ is pseudo-Anosov. We highlight the following consequence of the fact that nonzero powers of pseudo-Anosov elements are pseudo-Anosov, and hence irreducible.

\begin{lemma}\label{lem:crs:infpA}
    If $f \in \braid{m}$ is pseudo-Anosov, then $f^k(M)\neq M$ for all $k\in\ZZ_{>0}$ and all non-empty multicurves $M \subset\disk{m}$.
\end{lemma}

\p{Nielsen--Thurston classification theorem} Here we state a version of the Nielsen--Thurston classification theorem \cite{Thurston} in the special case of the braid group. See \cite[Section 13.3]{FM} for a discussion of the general version of the theorem.

\begin{theorem}\label{cor:crs:NTbraid}
Let $m\geq3$ and $f\in\braid{m}$. Then $f$ is exactly one of the following three types of elements:
    \begin{itemize}
        \item \defn{periodic}, i.e., there is some integer $k\geq1$ such that $f^k\in Z(\braid{m})$;
        \item \defn{aperiodic reducible}, i.e., $\crs{f}\neq\emptyset$; or
        \item \defn{pseudo-Anosov}, i.e., the image of $f$ in $\Mod(\sphere{m})$ is pseudo-Anosov.
    \end{itemize}
\end{theorem}

We remark that the Nielsen--Thurston classification theorem is typically stated for mapping class groups of surfaces $S_{g,n}$ without boundary. The notion of capping \cite[Section 3.6.2]{FM} provides a way to generalize this to surfaces with boundary. Furthermore, the theorem is not typically stated as a trichotomy, but rather that every mapping class is either periodic, reducible, or pseudo-Anosov: however periodic mapping classes can be reducible. It follows from work of Birman--Lubotzky--McCarthy that replacing reducible with the notion of aperiodic reducible as defined above gives a true trichotomy \cite{BirmanLubotzkyMcCarthy}.

\p{Orders of periodic elements} We stress that periodic elements of $\braid{m}$ are not finite order in $\braid{m}$. In fact, Fadell--Neuwirth showed $\braid{m}$ is torsion free~\cite[Theorem 8]{FN62}. Rather, the periodic elements of $\braid{m}$ are precisely the elements that are finite order in $\braid{m}/Z(\braid{m})$. Despite this, we say that the \defn{order} of a periodic element $a \in \braid{m}$ is the minimal $k \in \ZZ_{>0}$ such that $a^k \in Z(\braid{m})$.

\p{Classification of periodic elements} As far as the authors are aware, the following classification of periodic elements of $\braid{m}$ was first explicitly stated by Murasugi~\cite[Theorem C]{Murasugi}.  See also \cite[Theorem 7.1]{FM} and the following discussion.
\begin{lemma}\label{lem:crs:periodic}
Let $f \in \braid{m}$ be periodic.  Then $f$ is conjugate to a power of either $a_1 \in \braid{m}$ or $a_2 \in \braid{m}$.
\end{lemma}  We make repeated use of the following consequence of Lemma~\ref{lem:crs:periodic}.

\begin{lemma}\label{lem:crs:fixtwopunc}
    Let $f \in \braid{m}$ be periodic.  Suppose that there are two distinct punctures $p,q \in \disk{m}$ with $f(p) = p$ and $f(q) = q$.  Then $f \in Z(\braid{m})$.
\end{lemma}

\begin{proof}
    The element $f$ is conjugate to a power of $a_1$ or $a_2$ by Lemma~\ref{lem:crs:periodic}. Note that $a_1$ permutes the punctures of $\disk{m}$ in an $m$-cycle, whereas $a_2$ permutes the punctures of $\disk{m}$ in an $(m-1)$-cycle. Therefore, the only way that $f$ can fix two punctures is if it is conjugate to a power of $a_1^m=a_2^{m-1}=T_{\partial\disk{m}}$, i.e.\ if $f\in Z(\braid{m})$.
\end{proof}

\p{Basic properties of braid group homomorphisms} Let $\Phi:\braid{n}\to \braid{m}$ be a homomorphism. The map $\Phi$ is \defn{reducible} if there is a non-empty multicurve $M$ in $\disk{m}$ such that $f(M) = M$ for all $f \in \Phi(\braid{n})$. Such an $M$ is called a \defn{reducing system} of $\Phi$.  If $\Phi$ is not reducible, we say $\Phi$ is \defn{irreducible}.  

\p{Emptiness convention}  Note that our convention is that a reducing system for a mapping class $f \in \braid{m}$ is allowed to be empty, while a reducing system for a homomorphism $\Phi:\braid{n} \rightarrow \braid{m}$ must be nonempty.

We now use Theorem~\ref{cor:crs:NTbraid} to start analyzing homomorphisms $\Phi:\braid{n} \rightarrow \braid{m}$.

\begin{lemma}\label{lem:crs:centralclassification}
    Let $n \geq 5$ and $m \geq 3$.  Let $\Phi:\braid{n} \rightarrow \braid{m}$ be an irreducible and non-cyclic homomorphism.  Let $a \in \braid{n}$ be a periodic element. Then $\Phi(a)$ is periodic.
\end{lemma}

\begin{proof}
    It suffices to prove this in the case where $a\in Z(\braid{n})$, since for any $k\in\ZZ_{>0}$, the element $f\in\braid{m}$ is periodic if and only if $f^k$ is periodic. The hypothesis that $a\in Z(\braid{n})$ implies in particular that $\Phi(\braid{n})\subseteq\cent{\Phi(a)}{\braid{m}}$.
    
    Suppose by way of contradiction that there is some $a \in Z(\braid{n})$ such that $\Phi(a)$ is not periodic. By Theorem~\ref{cor:crs:NTbraid}, either $\crs{\Phi(a)}\neq\emptyset$ or $\Phi(a)$ is pseudo-Anosov.
    
    Suppose that $\crs{\Phi(a)}\neq\emptyset$. Since $\Phi(\braid{n})\subseteq\cent{\Phi(a)}{\braid{m}}$, Lemma~\ref{lem:crs:commdisj} implies that $\crs{\Phi(a)}$ is a reducing system for $\Phi(\braid{n})$. This contradicts irreducibility.

    Suppose now that $\Phi(a)$ is pseudo-Anosov. Since pseudo-Anosov braids have centralizer $\Z^2$~\cite[Theorem 1.1(a)]{GMW}, the fact that $\Phi(\braid{n})\subseteq\cent{\Phi(a)}{\braid{m}}$ implies that $\Phi(\braid{n})$ is abelian. Since $\braid{n}^{\ab}\cong\Z$, we conclude that $\Phi$ is cyclic, a contradiction.
\end{proof}

Lin showed that if $n\neq4$ and $m<n$, then any homomorphism $\Phi\colon\braid{n}\to\braid{m}$ is cyclic~\cite[Theorem 3.1(b)]{Lin2004braid}. For completeness, we reprove Lin's result here for irreducible homomorphisms. We reprove it in full in Theorem~\ref{thm:lin:mlessn}.

\begin{lemma}\label{lem:fewer-irred-cyclic}
    Let $n\geq2$ and $m\geq2$ with $n\neq4$. Assume that $m<n$. Let $\Phi\colon\braid{n}\to\braid{m}$ be an irreducible homomorphism. Then $\Phi$ is cyclic.
\end{lemma}

\begin{proof}
    It suffices to prove this when $n\geq5$ and $m\geq3$, otherwise the hypotheses imply that $\braid{m}$ itself is cyclic. The maximum order of a periodic element in $\braid{m}$ is $m$ by Lemma~\ref{lem:crs:periodic}. However, $\Phi(a_1)$ is a periodic element of $\braid{m}$ by Lemma~\ref{lem:crs:centralclassification}. Therefore, $\Phi(a_1^k)\in Z(\braid{m})$ for some integer $k$ satisfying $1\leq k\leq m \leq n-1$. Therefore, $\Phi$ is cyclic by Lemma~\ref{lem:crs:order}.
\end{proof}

The hypothesis $n\neq4$ above is required, since the pushfoward $R_*\colon\braid{4}\to\braid{3}$ of Ferrari's map $R$ is irreducible and non-cyclic.

\section{External braids and centralizers}\label{section:external}

We write $\disk{m}\cut\gamma$ for the disconnected surface $\disk{m}\setminus  U_\gamma$, where $U_\gamma$ is an open regular neighborhood of a smooth representative of $\gamma$ with no self intersection. Up to homeomorphism this does not depend on the choice of representative or regular neighborhood.  We emphasize that each component of $\disk{m}\cut\gamma$ has a boundary component homotopic to $\gamma$. See Figure~\ref{fig:cutex} for an example. If $M$ is a multicurve, we define $\disk{m} \cut M$ similarly.  We also define $\disk{m} \cut (M_1 \cup M_2)$ for $M_1$ and $M_2$ two multicurves with $M_1 \cap M_2 = \emptyset$ as follows.  Choose smooth minimally intersecting representatives $C_1$ and $C_2$ of $M_1$ and $M_2$ respectively that do not intersect $\partial \disk{m}$. Choose $U$ a regular neighborhood of $C_1 \cup C_2$.  The cut surface $\disk{m} \cut (M_1 \cup M_2)$ is the disjoint union of the connected components of $\disk{m} \setminus U$ that are not homeomorphic to disks.

\begin{figure}[ht]
\centering
   \begin{tikzpicture}
        \node[anchor = south west] at (-1,0){\includegraphics[scale=0.6]{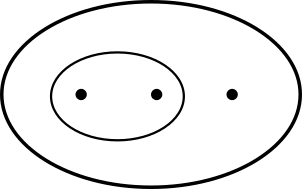}};
        \node at (1,2) {\large $\gamma$};
        \node at (5,1.6) {\Huge $\Rightarrow$};
        \node[anchor = south west] at (6,0){\includegraphics[scale=0.6]{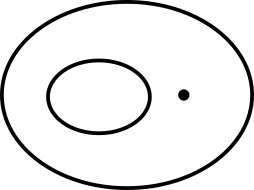}};
        \node at (10.9, 1.6) {\huge $\sqcup$};
        \node[anchor = south west] at (11.5,0.8){\includegraphics[scale=0.6]{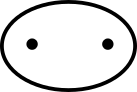}};
    \end{tikzpicture}
    \caption{Cutting $\gamma$ yields a disconnected surface: the disjoint union of a once punctured annulus and a twice punctured disk}\label{fig:cutex} 
\end{figure}
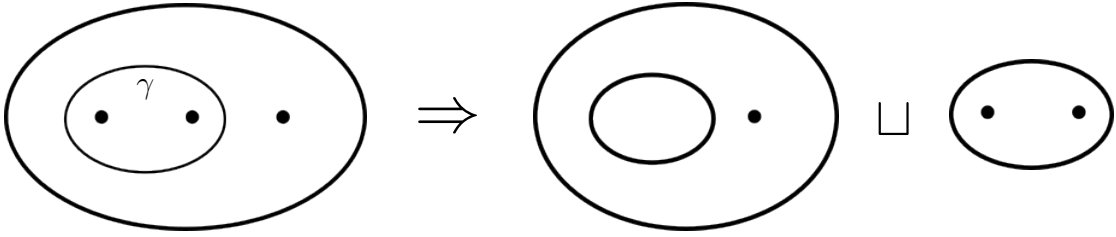

\p{Restrictions of curves to cut-open surfaces} Let $M \subset \disk{m}$ be a multicurve and let $\gamma \subset \disk{m}$ be a curve disjoint from $M$.  We can canonically identify $\gamma$ with a curve in $\disk{m} \cut M$ by choosing disjoint representatives of $M$ and $\gamma$.  We often make this identification implicitly without comment.

\p{Partial orders on curves}  Let $\gamma$ and $\gamma'$ be two distinct curves in $\disk{m}$.  We say that $\gamma'<\gamma$ if the following hold: 
\begin{itemize}
    \item $\gamma'$ and $\gamma$ are disjoint; and
    \item $\gamma'$ and $\partial \disk{m}$ are in different connected components of $\disk{m} \cut \gamma$.
\end{itemize}  Similarly, we say that $\gamma' \leq \gamma$ if either $\gamma' < \gamma$ or $\gamma' = \gamma$. This defines a partial order on the set of curves in $\disk{m}$. Note that the action of $\braid{m}$ on curves respects this partial order. That is, if $f\in\braid{m}$ and $\gamma'<\gamma$, then $f(\gamma')<f(\gamma)$. We use this property frequently without comment.

If $\gamma' \leq \gamma$ we say that $\gamma'$ is \defn{interior} to $\gamma$.  If $\calN$ is a set of curves, we let $\calN^{\outmulti}$ denote the \defn{exterior part} of $\calN$, which by definition is the subset of $\calN$ consisting of all curves $\gamma'$ such that there is no $\gamma \in \calN$ with $\gamma' < \gamma$.   We say that a curve $\gamma'$ is \defn{interior} to a multicurve $M$ if $\gamma' \leq \gamma$ for some $\gamma \in M$.  See Figure~\ref{fig:interiorex} for an example.   

\begin{figure}[ht]
    \centering
    \begin{tikzpicture}
        \node[anchor = south west] at (0,0){\includegraphics[scale=0.6]{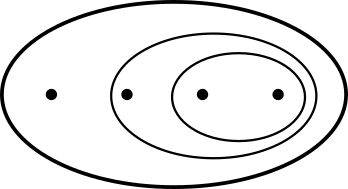}};
        \node at (2.7,1.8){\large $\gamma'$};
        \node at (1.7,1.8){\large $\gamma$};
    \end{tikzpicture}
    \caption{The curve $\gamma'$ is interior to $\gamma$}\label{fig:interiorex} 
\end{figure}

Let $\calN$ be a nonempty set of curves and let $\gamma$ be a curve disjoint from $\calN$. If there does not exist $\gamma'\in\calN$ with $\gamma<\gamma'$, we say that $\gamma$ is \defn{exterior} to $\calN$. Note that $\gamma$ can be exterior to $\calN$ without any curves in $\calN$ being interior to $\gamma$. See Figure~\ref{fig:exteriorex} for an example. 

\begin{figure}[ht]
    \centering
    \begin{tikzpicture}
        \node[anchor = south west] at (0,0){\includegraphics[scale=0.6]{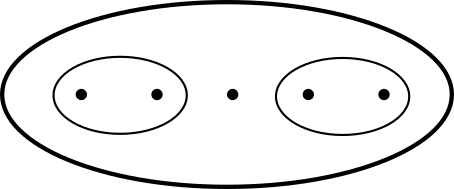}};
        \node at (2.9,2.4){\large $\gamma'$};
        \node at (6,2.4){\large $\gamma$};
    \end{tikzpicture}
    \caption{The curves $\gamma'$ and $\gamma$ are both exterior to each other}\label{fig:exteriorex}
\end{figure}  Similarly, we say that a multicurve $M$ is \defn{exterior} to a nonempty set of curves $\calN$ if every $\gamma\in M$ is exterior to $\calN$.  Note that a curve $\gamma$ is both interior and exterior to itself.  We say that a multicurve $M$ is \defn{non-nested} if $M = M^{\outmulti}$.   We also say that two sets of curves $\calN$ and $\calN'$ are \defn{mutually exterior} if every $\gamma \in \calN$ is exterior to $\calN'$ and vice versa.

\p{Inessential curves}  We also use the above terminology for inessential curves, or multicurves containing inessential curves, but only those that are homotopic to a puncture.  Since curves are always assumed to essential unless otherwise specified, it will be made clear when we are using the above terminology in reference to inessential curves.

\p{Curves containing punctures}  Similarly, we say a puncture $p \in \disk{m}$ is \defn{interior} to a curve $\gamma$ if $p$ lies in the connected component of $\disk{m} \cut \gamma$ not containing $\partial \disk{m}$. We write $\Int(\gamma)$ for the set of punctures interior to $\gamma$.  If $M$ is a non-nested multicurve, then we say that $p$ is \defn{interior} to $M$ if $p$ is interior to some $\gamma \in M$.  See Figure~\ref{fig:intpunc} for an example. We denote by $\Int(M)$ the set of punctures interior to a multicurve $M$. 

\begin{figure}
    \centering
    \begin{tikzpicture}
        \node at (0,0)[anchor = south west]{\includegraphics[scale=0.6]{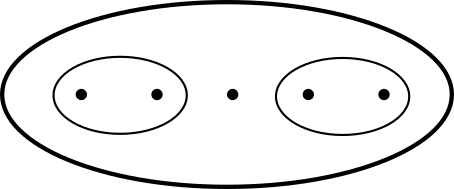}};
        \node at (1.8,1.7) {$p$};
        \node at (4.1,1.7) {$q$};
        \node at (2, 2.5) {$\gamma_1$};
        \node at (5, 2.5) {$\gamma_2$};
    \end{tikzpicture}
    \caption{$p$ is interior to the multicurve $\{\gamma_1, \gamma_2\}$, while $q$ is not}\label{fig:intpunc}
\end{figure}

\p{Braids fixing multicurves} Let $M \subset \disk{m}$ be a multicurve.  For $G\leq\braid{m}$, let $\Stab_G(M)$ denote the subgroup of $G$ consisting of elements that fix $M$ setwise. Assume now that $M$ is also non-nested. Kordek--Margalit~\cite[Lemma 3.1]{KordekMargalit} (see also the work of the second and fourth authors~\cite[Lemma 3.1]{HuxfordSchillewaert2025}) show that if the internal connected components of $\disk{m} \cut M$ are the disks $\disk{m_1},\ldots, \disk{m_k}$, then there is an exact sequence
\[
  1 \rightarrow \braid{m_1} \times \cdots \times \braid{m_k} \rightarrow \Stab_{\braid{m}}(M) \xrightarrow{\Ext_M} \braid{m - m_1 - \cdots - m_k+k}.
\]
This sequence in fact splits (non-canonically) over its image.  Let
\[
  \Ext_M:\Stab_{\braid{m}}(M) \rightarrow \braid{m-m_1-\cdots-m_k+k}
\]
denote the rightmost map in the above sequence and let $\braid{M} \coloneqq \Image(\Ext_M)$.  Let $\disk{M}$ denote the disk with $m - m_1 - \cdots - m_k+k$ punctures. If $f \in \Stab_{\braid{m}}(M)$, we call $\Ext_M(f)$ the \defn{external part} of $f$, which is the image of $f$ in $\braid{M}$.  Let $\Ext(f)$ denote $\Ext_{\crs{f}^{\outmulti}}(f)$, which we call the \defn{canonical external part} of $f$. Note that if $\crs{f} = \emptyset$, then $\Ext(f) = f$.  If the multicurve $M$ is not non-nested, then we define $\Ext_M \coloneqq \Ext_{M^{\outmulti}}$, $\braid{M}\coloneqq\braid{M^{\outmulti}}$, and $\disk{M}\coloneqq\disk{M^{\outmulti}}$.

\p{Equivariance of $\Ext$}  Let $M$ be a non-nested multicurve in $\disk{m}$, and $\delta$ a curve exterior to $M$.  Let $f \in \Stab_{\braid{m}}(M)$.  Let $\bar{\delta}$ denote the image of $\delta$ in the cut open surface $\disk{m} \cut M$.  By construction, $f(\delta)$ is equal to the natural lift of the curve $\Ext_M(f)(\bar{\delta})$ to $\disk{m}$.  We will use this observation repeatedly without comment.

\p{Basic properties of $\Ext$} We now record the following properties of $\Ext$.

\begin{lemma}\label{lem:centralizer:irreducible}Let $m\geq3$. If $f \in \braid{m}$, then $\Ext(f)$ is either pseudo-Anosov or periodic.
\end{lemma}

\begin{proof}
Suppose otherwise, so $M:=\crs{\Ext(f)}\neq \emptyset$ by Theorem~\ref{cor:crs:NTbraid}.  If $\gamma \subset \disk{\crs{f}}$ is a curve of $M$, then $\gamma$ lifts to a curve $\widehat{\gamma} \subset \disk{m}$ which must lie in $\crs{f}^{\outmulti}$.  But this implies that $\gamma$ is inessential in $\disk{\crs{f}}$, a contradiction.
\end{proof}  

We also obtain the following.
\begin{lemma}\label{lem:centralizer:typepreserve}
Let $m\geq3$. Let $M \subset \disk{m}$ be a multicurve and $f \in \Stab_{\braid{m}}(M)$ a periodic element.  Then $\Ext_M(f)$ is also periodic.
\end{lemma}

\begin{proof}
    Suppose by way of contradiction that $\Ext_M(f)$ is not periodic, so that by Lemma~\ref{lem:centralizer:irreducible} it is pseudo-Anosov.  This means that there is a curve $\delta \subset \disk{M}$ such that $\Ext_M(f)^k(\delta) \neq \delta$ for all $k \geq 1$ by Lemma~\ref{lem:crs:infpA}.  Now, let $\widehat{\delta}\subset\disk{m}$ be a lift of $\delta$.  We see that $f^k(\widehat{\delta}) \neq \widehat{\delta}$ for all $k \geq 1$, but this contradicts our hypothesis that $f$ is periodic.
    \end{proof}

\p{Degenerate multicurves} In order to apply the homomorphism $\Ext_M$ multiple times for different multicurves, we extend the notation $\Ext_M(f)$, $\braid{M}$, and $\disk{M}$, to allow for the possibility that a multicurve may contain inessential curves homotopic to a puncture. We identify such an inessential curve with the puncture it encircles.  For $M=M'\cup P$ the union of a multicurve $M'$ and a set $P$ of punctures, we say $M$ is \defn{non-nested} if $M'$ is non-nested and $\Int(M')\cap P=\emptyset$.

\begin{lemma}\label{lem:centralizer:sheaf}
Let $m\geq3$ and $M \subset \disk{m}$ be a non-nested multicurve.  Let $M_1$ and $M_2$ be two multicurves such that $M = M_1 \cup M_2$. We identify $M_1$ with its image in $\disk{M_2}$ after collapsing $M_2$ and similarly we identify $M_2$ with its image in $\disk{M_1}$ and $M$ with its image in $\disk{M_1\cap M_2}$. Let $f_1\in\Stab_{\braid{M_1}}(M_2)$ and $f_2\in\Stab_{\braid{M_2}}(M_1)$ be such that $\Ext_{M_2}(f_1) = \Ext_{M_1}(f_2)$.  Then there is a unique $f \in \Stab_{\braid{M_1 \cap M_2}}(M_1)\cap\Stab_{\braid{M_1\cap M_2}}(M_2)$ with $\Ext_{M_i}(f) = f_i$ for $i =1,2$.
\end{lemma}

\begin{proof}
Let $M_0=M_1\cap M_2$. We have the following diagram
\begin{center}
\begin{tikzcd}
\Stab_{\braid{M_0}}(M_1) \cap\Stab_{\braid{M_0}}(M_2) \arrow["\Ext_{M_1}"]{r} \arrow["\Ext_{M_2}"]{d}& \Stab_{\braid{M_1}}(M_2) \arrow["\Ext_{M_2}"]{d}\\
\Stab_{\braid{M_2}}(M_1) \arrow["\Ext_{M_1}"]{r} & \braid{M}.
\end{tikzcd}
\end{center}
This diagram commutes because the two composite homomorphisms are the restriction of $\Ext_M\colon\Stab_{\braid{M_0}}(M)\to\braid{M}$ to the subgroup $\Stab_{\braid{M_0}}(M_1)\cap\Stab_{\braid{M_0}}(M_2)$. The lemma amounts to saying that the above diagram is a pullback diagram. Define
\[
K \coloneqq \ker(\Ext_M\colon\Stab_{\braid{M_0}}(M)\to\braid{M}) \leq \Stab_{\braid{M_0}}(M_1)\cap\Stab_{\braid{M_0}}(M_2),
\]
and
\[
K_1 \coloneqq \ker(\Ext_{M_1}\colon\Stab_{\braid{M_2}}(M_1)\to\braid{M}), \qquad K_2 \coloneqq \ker(\Ext_{M_2}\colon\Stab_{\braid{M_1}}(M_2)\to\braid{M}).
\]
Since all the maps in the above commutative diagram above are surjective, it suffices to show that the following diagram obtained by restriction
\begin{center}
\begin{tikzcd}
K \arrow[r]\arrow[d] & K_2 \arrow[d] \\
K_1 \arrow[r] & \{1\}
\end{tikzcd}
\end{center}
is a pullback diagram. This follows from the description of the group $K$ due to Kordek--Margalit~\cite[Lemma 3.1]{KordekMargalit}, who show that $K\cong K_1\times K_2$.
\end{proof}

We also need the following auxiliary lemma which is a consequence of the classification of periodic elements in $\braid{m}$.  

\begin{lemma}\label{lem:centralizer:fixtwocurves}
Let $m\geq3$. Let $f \in \braid{m}$ be a periodic element.  Let $M \subset \disk{m}$ be a possibly degenerate multicurve.  Assume that $\left|M\right| \geq 2$ and that $f$ fixes each curve and puncture in $M$.  Then $f\in Z(\braid{m})$.
\end{lemma}

\begin{proof}
It suffices to prove this in the case that $\left|M\right| = 2$.  Let $M \coloneqq \{\gamma,\gamma'\}$. By Lemma~\ref{lem:centralizer:typepreserve}, the element $\Ext_M(f)$ is also periodic. Since $\Ext_M(f)$ fixes two punctures of $\disk{M}$, by Lemma~\ref{lem:crs:fixtwopunc} we have $\Ext_M(f)\in\langle T_{\disk{M}}\rangle$. Let $k,\ell\in\Z$ with $k\neq0$ be such that $f^k=T_{\partial\disk{m}}^\ell$. Since $\Ext_M(T_{\partial\disk{m}})=T_{\partial\disk{M}}$, we have $T_{\partial\disk{M}}^\ell=\Ext_M(f^k)\in\langle T_{\disk{M}}^k\rangle$. Hence $k\mid\ell$, and so by the classification of periodic elements from Lemma~\ref{lem:crs:periodic} we have $f=T_{\partial\disk{m}}^{\ell/k}\in Z(\braid{m})$.
\end{proof}

\p{Centralizers of braids}  Gonz\'alez-Meneses--Wiest~\cite[Theorem 1.1]{GMW} describe the centralizer of a braid for a general $f \in \braid{m}$.  We need the following specific statements. 

\begin{lemma}\label{lem:centralizer:basics}
Let $m\geq3$. Let $f \in \braid{m}$ be a periodic element.  If $g \in \braid{m}$ is conjugate to $f$ and commutes with $f$ then $f=g$.
\end{lemma}
\begin{proof}
   Suppose that $f$ is periodic of order $k$ and conjugate to $a_i$, where $i\in\{1,2\}$. Let $d = \frac{m-i+1}{k}\in\Z$. Gonz{\'a}lez-Meneses--Wiest show there is an injective homomorphism $\varphi_f\colon\cent{f}{\braid{m}}\to\braid{d+1}$ such that $\varphi_f(f)=T_{\partial\disk{d+1}}$~\cite[Theorems 3.2 and 3.4]{GMW}. Suppose that $g\in\cent{f}{\braid{m}}$ is conjugate to $f$. The fact that $f$ and $g$ are conjugate periodic elements of order $k$ implies that $f^k=g^k$, thus $\varphi_f(g^k)=T_{\partial\disk{d+1}}^k$. By the classification of periodic elements in $\braid{d+1}$ from Lemma~\ref{lem:crs:periodic}, $\varphi_f(g)=T_{\partial\disk{d+1}}$, and hence $f=g$.
\end{proof}

We make use of the above lemma to prove the following.  This is a result of Chen--Kordek--Margalit~\cite[Proposition 7.1]{CKM}, which we include a proof of for completeness.

\begin{lemma}\label{lem:centralizer:emptycrs}
    Let $n\geq5$ and $m\geq3$. Let $\Phi:\braid{n} \rightarrow \braid{m}$ be a homomorphism. Let $s_i\in \extgen{n}$.  If $\crs{\Phi(s_i)}=\emptyset$, then $\Phi$ is cyclic.
\end{lemma}

\begin{proof}
    Since $\crs{\Phi(s_i)}=\emptyset$, by Lemma~\ref{lem:centralizer:irreducible} the braid $\Phi(s_i)$ is either pseudo-Anosov or periodic. Recall that the standard generators $s_j\in\extgen{n}$ are all conjugate.
    
    \p{Case 1: $\Phi(s_i)$ is periodic} Since $n\geq5$ we have $[s_i,s_{i+2}]=1$, hence $\Phi(s_i)=\Phi(s_{i+2})$ by Lemma~\ref{lem:centralizer:basics}. Since $s_i\neq s_{i+2}$ we can apply Lemma~\ref{lem:crs:collisioncollapse} to conclude $\Phi$ is cyclic.

    \p{Case 2: $\Phi(s_i)$ is pseudo-Anosov} For $j\in\Z$ let
    \[
      A_j = \cent{\Phi(s_j)}{\braid{m}}.
    \]
    By Gonz{\'a}lez-Meneses--Wiest~\cite[Proposition 4.1]{GMW} each $A_j$ is abelian. If $[s_j,s_k] = 1$, then $s_j \in A_k$ and $s_k \in A_j$.  Since $A_k$ is abelian, any element $f \in A_k$ satisfies $[f,s_j] = 1$, so $A_k \subseteq A_j$.  Similarly any $g \in A_j$ satisfies $[g,s_k] = 1$, so $A_j \subseteq A_k$.  We conclude that $A_j = A_k$.  Since $n \geq 5$ the graph $\Comm_{\braid{n}}(\extgen{n})$ is connected.  In particular, $A_i = A_j$ for all $j\in\Z$.  Therefore, $\Phi(\braid{n}) \subseteq A_i$.  Since $\braid{n}^{\ab} \cong \ZZ$, we conclude that $\Phi$ is cyclic.
\end{proof}

We obtain the following criterion for a homomorphism to be reducible.

\begin{lemma}\label{lem:external:commonmax}
    Let $n\geq5$ and $m\geq3$. Let $\Phi:\braid{n} \rightarrow \braid{m}$ be a non-cyclic homomorphism.  Assume that $\crs{\Phi(s_i)}^{\outmulti} = \crs{\Phi(s_j)}^{\outmulti}$ for two distinct $s_i,s_j \in \extgen{n}$.  Then $\Phi$ is reducible.
\end{lemma}

\begin{proof}
    Since $n\geq5$, the commuting graph $\Comm_{\braid{n}}(\extgen{n})$ has large neighborhoods. In particular, there is some $s_k\in\extgen{n}$ such that each $s_\ell\in\extgen{n}\setminus\{s_k\}$ commutes with either $s_i$ or $s_j$. Therefore, by Lemma~\ref{lem:crs:commdisj} the braid $\Phi(s_\ell)$ preserves $\crs{\Phi(s_i)}^{\outmulti}=\crs{\Phi(s_j)}^{\outmulti}$ for each $s_\ell\in\extgen{n}\setminus\{s_k\}$. Note that $\crs{\Phi(s_i)}^{\outmulti}$ is non-empty by Lemma~\ref{lem:centralizer:emptycrs}. Since $\extgen{n}\setminus\{s_k\}$ is a generating set for $\braid{n}$, it follows that $\crs{\Phi(s_i)}^{\outmulti}$ is a reducing system for $\Phi$.
\end{proof}

We say that a homomorphism $\Phi:\braid{n} \rightarrow \braid{m}$ is \defn{externally periodic} if $\Ext(\Phi(s_i))$ is periodic for some, and hence for all, $s_i \in \extgen{n}$. 

\begin{lemma}\label{lem:external:anosov}
  Let $\Phi: \braid{n} \rightarrow \braid{m}$ be an irreducible and non-cyclic homomorphism with $n \geq 5$ and $m\geq 3$. Then $\Phi$ is externally periodic.
\end{lemma}

\begin{proof}
    Let $f_i\coloneqq\Ext(\Phi(s_i))$ for each $i\in\Z$. By Lemma~\ref{lem:centralizer:irreducible} it suffices to show that $f_1$ is not pseudo-Anosov. Suppose for a contradiction that $f_1$ is pseudo-Anosov, and hence that $f_i$ is pseudo-Anosov for all $i\in\Z$.

    For $i,j\in\Z$, let $M_{i,j}$ denote the multicurve in $\disk{\crs{\Phi(s_i)}}$ consisting of $\delta \in \crs{\Phi(s_j)}$ such that $\delta$ is exterior to $\crs{\Phi(s_i)}$. If $[s_i,s_j]=1$, then by Lemma~\ref{lem:crs:commdisj} the curves in
    \[
        \crs{\Phi(s_i)} \cup \crs{\Phi(s_j)}
    \]
    are pairwise disjoint and $\crs{\Phi(s_j)}$ is a reducing system for $\Phi(s_i)$. It follows that $M_{1,3}$ is a reducing system for $f_1$, and $M_{3,1}$ is a reducing system for $f_3$. Since we have assumed these elements are pseudo-Anosov, we must have $M_{1,3}=\emptyset$ and $M_{3,1}=\emptyset$ by Lemma~\ref{lem:crs:infpA}. Since the curves in $\crs{\Phi(s_1)}$ and $\crs{\Phi(s_3)}$ are pairwise disjoint, it follows that
    \[
    \crs{\Phi(s_1)}^{\outmulti} = \crs{\Phi(s_3)}^{\outmulti}.
    \]
    Since $\Phi$ is non-cyclic this contradicts the irreducibility of $\Phi$ by Lemma~\ref{lem:external:commonmax}. Therefore $f_1$ is not pseudo-Anosov, and thus $\Phi$ is externally periodic.
\end{proof}

Henceforth we only need to consider externally periodic homomorphisms in our proof of Theorem~\ref{mainthm:natleast5}.  

Before proceeding with the proof of Theorem~\ref{mainthm:natleast5}, we handle the special case $m<n$ in Theorem~\ref{thm:lin:mlessn}, reproving a result originally due to Lin~\cite[Theorem 3.1(b)]{Lin2004braid}. Let $\Phi:\braid{n} \rightarrow G$ be a homomorphism and let $f \in \cent{G}{\braid{n}}$. The \defn{transvection} of $\Phi$ along $f$ is the homomorphism
\[
  \Phi_f:\braid{n} \rightarrow G
\]
defined by $\Phi_f(s_i) = s_i f$ for all $s_i \in \extgen{n}$ (see Chen--Kordek--Margalit~\cite{CKM}). We require one auxiliary lemma.

\begin{lemma}\label{lem:alpha:transvectioncyclic}
Let $n \geq 3$ and let $G$ be a group.  Let $\Phi:\braid{n} \rightarrow G$ be a cyclic homomorphism and let $f \in \cent{G}{\Phi(\braid{n})}$.  Then the transvection $\Phi_f$ is cyclic.
\end{lemma}

\begin{proof}
    For all $s_i, s_i \in \extgen{n}$ we have $\Phi(s_i) = \Phi(s_j)$.  Therefore $\Phi(s_i)f = \Phi(s_j)f$ so $\Phi(\braid{n}) = \langle \Phi(s_i)f \rangle$ for some $s_i \in \extgen{n}$.  In particular, $\Phi_f$ is cyclic.
\end{proof}

Part (a) of the following theorem is a result originally due to Lin~\cite[Theorem 3.1(b)]{Lin2004braid}.

\begin{theorem}\label{thm:lin:mlessn}
Let $n\geq2$ and $m\geq1$ with $n\neq4$. Assume that $m\leq n$. Let $\Phi\colon\braid{n}\to\braid{m}$ be a homomorphism. Assume further that either
\begin{enumerate}
    \item[(a)] $m<n$ (see~\cite[Theorem 3.1(b)]{Lin2004braid}); or
    \item[(b)] $m=n$ and $\Phi$ is reducible.
\end{enumerate}
Then $\Phi$ is cyclic.
\end{theorem}

\begin{proof}
    Fix $n\geq2$ with $n\neq4$. We proceed by induction on $m\geq2$. We handle cases (a) and (b) simultaneously.

    \p{Base case: $m=2$} In this case $\braid{m}$ is cyclic, so the result holds.

    \p{Inductive step: $m \geq 3$} Our inductive hypothesis is that every homomorphism $\braid{n}\to\braid{m'}$ with $m'<m$ is cyclic. If $m<n$ and $\Phi$ is irreducible, then $\Phi$ is cyclic by Lemma~\ref{lem:fewer-irred-cyclic}. Suppose then that $m\leq n$ and $\Phi$ is reducible, and let $M$ be a reducing system for $\Phi$. Since $M$ is non-empty, $\disk{M}$ has strictly fewer than $m$ punctures. By our inductive hypothesis $\Ext_M \circ \Phi$ is cyclic.  Let
    \[
      \overline{f} = (\Ext_M \circ \Phi)(s_i)
    \]
    for some (and hence all) $s_i \in \extgen{n}$.  Write
    \[
      \ker(\Ext_M)= \braid{m_1} \times \cdots \times \braid{m_k},
    \]
    and note that each $m_\ell$ is strictly smaller than $m$ since $M$ is non-empty. Choose a lift $f\in\Stab_{\braid{m}}(M^{\outmulti})$ of $\overline{f}$ that permutes the factors of $\ker(\Ext_M)$ (see the work of the second and fourth author~\cite[Lemma 3.1]{HuxfordSchillewaert2025}).
    
    By construction $\Ext_M(f) = \Ext_M(\Phi(s_i))$ for all $s_i \in \extgen{n}$, and $f \in \cent{\Phi(\braid{n})}{\braid{n}}$. In particular, the transvection $\Phi_{f^{-1}}$ is well defined and also has $M^{\outmulti}$ as a reducing system. Furthermore, $\Ext_M\circ\Phi_{f^{-1}}$ is trivial, so the image of $\Phi_{f^{-1}}$ lies in $\ker(\Ext_M)=\braid{m_1} \times \cdots \times \braid{m_k}$. By our inductive hypothesis, any homomorphism $\braid{n}\to\braid{m_\ell}$ is cyclic for all $1\leq \ell\leq k$. Therefore $\Phi_{f^{-1}}$ has abelian image. Since $\braid{n}^{\ab}\cong\Z$ this implies $\Phi_{f^{-1}}$ is cyclic. By Lemma~\ref{lem:alpha:transvectioncyclic}, we have $\Phi=(\Phi_{f^{-1}})_f$ is also cyclic.
\end{proof}

The hypothesis $n\neq4$ above is required, since the pushfoward $R_*\colon\braid{4}\to\braid{3}$ of Ferrari's map is a counterexample to part (a), and post-composing it with the standard inclusion $\braid{3}\hookrightarrow\braid{4}$ gives a counterexample to part (b).

\section{Maximal curves}\label{section:alpha}

We now focus on the case of externally periodic homomorphism $\Phi:\braid{n} \rightarrow \braid{m}$.  We begin with the following, which we use repeatedly.

\begin{lemma}\label{lem:alpha:exterioragreement}
  Let $n \geq 4$ and $m \geq 3$.  Let $\Phi:\braid{n} \rightarrow \braid{m}$ be an externally periodic homomorphism.  Let $s_i, s_j \in \extgen{n}$ with $[s_i,s_j]=1$. Let $M = \crs{\Phi(s_i)}^{\outmulti} \cup \crs{\Phi(s_j)}^{\outmulti}$. Then
  \[
    \Ext_M(\Phi(s_i)) = \Ext_M(\Phi(s_j)).
  \]
  In particular, if $\delta \subset \disk{m}$ is a possibly inessential curve exterior to $M$, then $\Phi(s_i)(\delta) = \Phi(s_j)(\delta)$.
\end{lemma}
\begin{proof}
    Note that $M$ is indeed a multicurve by Lemma~\ref{lem:crs:commdisj}. The elements $\Ext_M(\Phi(s_i))$ and $\Ext_M(\Phi(s_j))$ are periodic by Lemma~\ref{lem:centralizer:typepreserve}. They also commute. By Lemma~\ref{lem:centralizer:basics} it suffices to show that they are also conjugate.
    
    By the change of coordinates principle, there is some $f\in\braid{n}$ such that $fs_if^{-1}=s_j$ and $fs_jf^{-1}=s_i$. We have $\Phi(f)(M)=M$ by Lemma~\ref{lem:crs:functorial}. Thus $\Ext_M(\Phi(s_i))$ and $\Ext_M(\Phi(s_j))$ are indeed conjugate by the braid $\Ext_M(\Phi(f))$.  
\end{proof}

As a consequence, we can give a criterion for a map $\Phi:\braid{n} \rightarrow \braid{m}$ to be reducible. Let $s_i\in \extgen{n}$. We say that a curve $\gamma \in \crs{\Phi(s_i)}^{\outmulti}$ is \defn{$\Phi$-maximal} if there is no $s_j \in \extgen{n}$ and $\gamma' \in \crs{\Phi(s_j)}$ such that $\gamma' > \gamma$.  The set of $\Phi$-maximal curves is be denoted by $\extremal$. Note that
\[
  \extremal = \left(\bigcup_{s_i\in\extgen{n}}\crs{\Phi(s_i)}\right)^{\outmulti} = \left(\bigcup_{s_i\in\extgen{n}}\crs{\Phi(s_i)}^{\outmulti}\right)^{\outmulti}.
\]
For $s_i \in \extgen{n}$ we let
\[
  \extremalarg{s_i} \coloneqq \extremal \cap \crs{\Phi(s_i)}=\extremal \cap \crs{\Phi(s_i)}^{\outmulti}.
\]
We prove in Lemma~\ref{lem:curves:nonesting} that $ \extremalarg{s_i} = \crs{\Phi(s_i)}^{\outmulti}$ when $\Phi:\braid{n} \rightarrow \braid{m}$ is irreducible and non-cyclic with $n \geq 5$ and $m \geq 3$.

We have the following result about the action of $\Phi(a_1)$ on $\extremal$.

\begin{lemma}\label{lem:alpha:standardrootaction}
  Let $n \geq 3$ and $m \geq 3$.  Let $\Phi:\braid{n} \rightarrow \braid{m}$ be a homomorphism.  If $s_i \in \extgen{n}$, then
  \[
    \Phi(a_1)(\extremalarg{s_i}) = \extremalarg{s_{i+1}}.
  \]
\end{lemma}

\begin{proof}
  Since $a_1 s_j a_1^{-1} = s_{j+1}$ for all $s_j\in\extgen{n}$, we have by Lemma~\ref{lem:crs:functorial}
  \[
    \Phi(a_1)(\crs{\Phi(s_j)}) = \crs{\Phi(s_{j+1})} \quad \text{and} \quad \Phi(a_1^{-1})(\crs{\Phi(s_{j+1})})=\crs{\Phi(s_j)}.
  \]
    It remains to show that $\Phi(a_1)(\extremal)=\extremal$. From the above, both $\Phi(a_1)$ and $\Phi(a_1^{-1})$ send the union $\bigcup_{s_i\in\extgen{n}}\crs{\Phi(s_i)}$ to itself. Since both $\Phi(a_1)$ and $\Phi(a_1^{-1})$ preserve the partial order $\leq$ on curves in $\disk{m}$, it follows that $\Phi(a_1)$ preserves the set of elements of this union that are maximal with respect to $\leq$. In other words, $\Phi(a_1)(\extremal)=\extremal$, as required.
\end{proof} We now study the action of $\Phi(s_j)$ on $\extremal$.

\begin{lemma}\label{lem:alpha:extremalaction}
Let $n \geq 5$ and $m \geq 3$.  Let $\Phi:\braid{n} \rightarrow \braid{m}$ be an externally periodic homomorphism.  If $s_i,s_j \in \extgen{n}$ satisfy $[s_i,s_j] = 1$, then $\Phi(s_j)(\extremalarg{s_i})=\extremalarg{s_i}$.
\end{lemma}

\begin{proof}
    Since $\extremalarg{s_i}$ is finite, it suffices to show that $\Phi(s_j)(\extremalarg{s_i})\subseteq\extremalarg{s_i}$. In fact, by Lemma~\ref{lem:crs:commdisj} it suffices to show that $\Phi(s_j)(\extremalarg{s_i})\subseteq\extremal$. Let $\gamma\in\extremalarg{s_i}$. Assume for a contradiction that $\Phi(s_j)(\gamma)\notin\extremal$. Let $s_k\in \extgen{n}$ and $\gamma' \in \crs{\Phi(s_k)}$ be a curve with $\gamma' > \Phi(s_j)(\gamma)$.  Assume without loss of generality that $\gamma' \in \extremalarg{s_k}$.
    There are three cases to consider. Note that $\Phi(s_j)(\gamma)=\Phi(s_i)(\gamma)$ by Lemma~\ref{lem:alpha:exterioragreement}.

    \p{Case 1: $[s_k,s_j] = 1$} We have $\gamma<\Phi(s_j)^{-1}(\gamma')$, and $\Phi(s_j)^{-1}(\gamma')\in\crs{\Phi(s_k)}$ by Lemma~\ref{lem:crs:commdisj}, which contradicts $\gamma\in\extremal$.

    \p{Case 2: $[s_k,s_i] = 1$}  By Lemma~\ref{lem:alpha:exterioragreement} we have $\Phi(s_k)(\gamma) = \Phi(s_i)(\gamma)$, and so $\gamma < \Phi(s_k)^{-1}(\gamma')$.  Since $\Phi(s_k)^{-1}(\gamma') \in \crs{\Phi(s_k)}$, this contradicts $\gamma\in\extremal$.
    
    \p{Case 3: $[s_k,s_i] \neq 1$ and $[s_k, s_j] \neq 1$} 
    Since $n \geq 5$ and $[s_k,s_i]\neq 1$, the fact that $\Comm_{\braid{n}}(\extgen{n})$ has diameter 2 implies that there is an $s_\ell \in \extgen{n}$ with $[s_\ell,s_i]=[s_\ell,s_k]=1$. Applying Lemma~\ref{lem:alpha:exterioragreement} we see that $\Phi(s_i)(\gamma)=\Phi(s_\ell)(\gamma)$. Thus, $\gamma<\Phi(s_\ell)^{-1}(\gamma')$.  Since $[s_\ell, s_k] = 1$, we have
    \[
      \Phi(s_\ell)^{-1}(\gamma') \in \crs{\Phi(s_k)}
    \]
    by Lemma~\ref{lem:crs:commdisj}.  This contradicts $\gamma\in\extremal$.
\end{proof} In general, if $\gamma \in \extremalarg{s_i}$, the curve $\Phi(s_{i+1})(\gamma)$ is not $\Phi$-maximal. Indeed, this is not the case for the identity homomorphism. We now explain how the braid relation interacts with $\Phi$-maximal curves.

\begin{lemma}\label{lem:alpha:extremalbraid}
Let $n \geq 5$ and $m \geq 3$.  Let $\Phi:\braid{n} \rightarrow \braid{m}$ be an externally periodic homomorphism.  If $s_i\in \extgen{n}$ then $\Phi(s_is_{i+1})(\extremalarg{s_i})=\extremalarg{s_{i+1}}$.
\end{lemma}

\begin{proof}
    Let $\gamma\in\extremalarg{s_i}$. Note that
    \[\Phi(a_1)(\gamma)=\Phi(s_{i+3}\cdots s_{i+n-1}s_{i+n}s_{i+n+1})(\gamma)=\Phi(s_{i+3}\cdots s_{i+n-1})(\Phi(s_is_{i+1})(\gamma)).\]
    If $[s_k,s_{i+1}]=1$, then $\Phi(s_k)(\extremalarg{s_{i+1}})=\extremalarg{s_{i+1}}$ by Lemma~\ref{lem:alpha:extremalaction}, hence $\Phi(s_k)^{-1}(\extremalarg{s_{i+1}})=\extremalarg{s_{i+1}}$. Furthermore, $\Phi(a_1)(\gamma)\in \extremalarg{s_{i+1}}$ by Lemma~\ref{lem:alpha:standardrootaction}. By the above equality we have $\Phi(s_is_{i+1})(\gamma)\in \extremalarg{s_{i+1}}$. Therefore $\Phi(s_is_{i+1})(\extremalarg{s_i})\subseteq\extremalarg{s_{i+1}}$. By Lemma~\ref{lem:alpha:standardrootaction}, the finite sets $\extremalarg{s_i}$ and $\extremalarg{s_{i+1}}$ have the same cardinality, so in fact $\Phi(s_is_{i+1})(\extremalarg{s_i}) = \extremalarg{s_{i+1}}$.
\end{proof} 

\p{Remark} Note that the relation $\Phi(s_{i+1}s_i)(\extremalarg{s_{i+1}}) = \extremalarg{s_i}$ does not \emph{a priori} hold for $s_i\in\extgen{n}$.

Lemma~\ref{lem:alpha:extremalbraid} implies the following reducibility criterion.

\begin{lemma}\label{lem:alpha:extremalcollision}
  Let $n\geq5$ and $m\geq3$. Let $\Phi:\braid{n} \rightarrow \braid{m}$ be an externally periodic homomorphism.  Let $s_i,s_j\in \extgen{n}$ with $s_i\neq s_j$, and suppose that
  \[
    \extremalarg{s_i} \cap \extremalarg{s_j} \neq \emptyset.
  \]
  Then $\Phi$ is reducible.  Furthermore, $\extremalarg{s_i} \cap \extremalarg{s_j} \subseteq \extremalarg{s_k}$ for all $s_k \in \extgen{n}$.
\end{lemma}

\begin{proof}
  Let $M = \extremalarg{s_i} \cap \extremalarg{s_j}$.  We show that $M$ is a reducing system for $\Phi$.  Since $n \geq 5$, $\Comm_{\braid{n}}(\extgen{n})$ has large neighborhoods. In particular, by conjugating by a power of $a_1$ and applying Lemma~\ref{lem:alpha:standardrootaction}, we may assume without loss of generality that all $s_k \in \extgen{n}\setminus\{s_n\}$ commute with at least one of $s_i$ or $s_j$.  Let $s_k \in \extgen{n} \setminus \{s_n\}$.  Without loss of generality, assume that $[s_k,s_i]=1$. By Lemma~\ref{lem:alpha:extremalaction},
  \[
    \Phi(s_k)(M) \subseteq \extremalarg{s_i}.
  \]
  Since $n\geq5$, the graph $\Comm_{\braid{n}}(\extgen{n})$ has diameter 2. Therefore, there is some $s_{\ell}\in \extgen{n}$, possibly equal to $s_k$ or $s_j$, such that $[s_k,s_{\ell}]=[s_j,s_{\ell}]=1$. By Lemma~\ref{lem:alpha:exterioragreement} we have
  \[\Phi(s_k)(M) = \Phi(s_{\ell})(M).\]
  By Lemma~\ref{lem:alpha:extremalaction}, $\Phi(s_{\ell})(M) \subseteq \extremalarg{s_j}$.  Therefore $\Phi(s_k)(M) = \Phi(s_{\ell})(M) \subseteq M$. Since $\extgen{n}\setminus\{s_n\}$ generates $\braid{n}$, this implies that $M$ is a reducing system for $\Phi$.  
    
    The second part of the lemma follows by noting that $\Phi(a_1^k)(M) = M$ for all $k \in \ZZ$ since $M$ is a reducing system, but also that $\Phi(a_1^k)(M) \subseteq \extremalarg{s_{i+k}}$ by Lemma~\ref{lem:alpha:standardrootaction}.
\end{proof}  

\section{Types of punctures}\label{section:punctures}

We now study the action of images of externally periodic homomorphisms on the punctures in $\disk{m}$.  Let $\Phi:\braid{n} \rightarrow \braid{m}$ be a homomorphism and let $p \in \disk{m}$ be a puncture.  For each $s_i \in \extgen{n}$, let 
\[
  \extremalarg{s_i}(p) \coloneqq \{\gamma \in \extremalarg{s_i}: p \in \Int(\gamma)\}.
\]
Note that by the definition of $\extremal$ the set $\extremalarg{s_i}(p)$ is always either empty or a singleton.  We also let $\extremal(p) \coloneqq \bigcup_{s_i \in \extgen{n}} \extremalarg{s_i}(p)$.  We say that the \defn{type} of a puncture $p \in \disk{m}$ is the set
\[
  \calT_{\Phi}(p) \coloneqq \{i: \extremalarg{s_i}(p)\neq\emptyset\}.
\]
As with the indexing convention for $\extgen{n}$, if $i \in \calT_{\Phi}(p)$ then $i + kn \in \calT_{\Phi}(p)$ for all $k \in \ZZ$. For the sake of notational convenience, we consider $\calT_{\Phi}(p)$ to be a set of residue classes modulo $n$, but we still write its elements as integers.  See Figure~\ref{fig:typepuncex} for an example of the type of a puncture.

\begin{figure}[ht]
    \centering
    \begin{tikzpicture}
        \node at (0,0)[anchor = south west]{\includegraphics[scale=0.6]{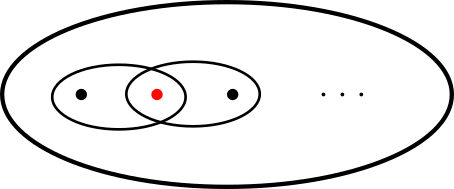}};
        \node at (1.7,2.3){$\gamma_1$};
        \node at (3.5,2.3){$\gamma_2$};
        \node at (2.35,1.7){$p$};
    \end{tikzpicture}
    \caption{If $\gamma_1 \in \extremalarg{s_1}$ and $\gamma_2 \in \extremalarg{s_2}$, then $\calT_{\Phi}(p)= \{1,2\}$}\label{fig:typepuncex}
\end{figure}

Our goal in Section~\ref{section:punctures} is to prove Theorem~\ref{thm:alpha:welltypedpuncture}, which classifies the possible types of punctures and describes how $\Phi(s_i)$ acts on punctures of different types.  We also prove Theorem~\ref{thm:case:m:leq:n}, which resolves Theorem~\ref{mainthm:natleast5} in the case that $n \geq 5$ and $3\leq m\leq n$.

A homomorphism $\Phi:\braid{n} \rightarrow \braid{m}$ is \defn{minimally typed} if for every puncture $p \in \disk{m}$, there is some $s_i \in \extgen{n}$ such that $\calT_{\Phi}(p) \subseteq \{i,i+1\}$. Note that possibly $\calT_{\Phi}(p) = \emptyset$.  

\begin{lemma}\label{lem:alpha:nonwelltypedirred}
    Let $n \geq 5$ and $m \geq 3$.  Let $\Phi:\braid{n} \rightarrow \braid{m}$ be a homomorphism.  Suppose that $\Phi$ is not minimally typed.  Then $\Phi$ is reducible.
\end{lemma}
\begin{proof}
    Let $p \in \disk{m}$ be a puncture such that $\calT_{\Phi}(p) \not \subseteq  \{i,i+1\}$ for all $s_i \in \extgen{n}$.  Then there are $i,j \in \calT_{\Phi}(p)$ such that $s_i\neq s_j$ and $[s_i,s_j]=1$.  Let $\gamma_i \in \extremalarg{s_i}(p)$ and $\gamma_j \in \extremalarg{s_j}(p)$.  Since $[s_i,s_j]=1$, the curves $\gamma_i$ and $\gamma_j$ are disjoint by Lemma~\ref{lem:crs:commdisj}.  Since $\gamma_i,\gamma_j\in\extremal$ and $p \in \Int(\gamma_i)\cap\Int(\gamma_j)$, we obtain $\gamma_i = \gamma_j$.  Hence $\extremalarg{s_i} \cap \extremalarg{s_j} \neq \emptyset$, so $\Phi$ is reducible by Lemma~\ref{lem:alpha:extremalcollision}.
\end{proof} In particular, if $\Phi$ is irreducible, we see that any $p \in \disk{m}$ has $\calT_{\Phi}(p)$ equal to one of $\emptyset$, $\{i\}$, or $\{i,i+1\}$, for some $s_i \in \extgen{n}$.  We show in Lemma~\ref{lem:alpha:welltypednouni} that we cannot have $\calT_{\Phi}(p) = \{i\}$.  We begin with the following.  
\begin{lemma}\label{lem:alpha:periodictypepreserve}
  Let $n \geq 5$ and $m \geq 3$.  Let $\Phi:\braid{n} \rightarrow \braid{m}$ be an externally periodic homomorphism. For any puncture $p \in \disk{m}$ we have
  \[
    \calT_{\Phi}(\Phi(a_1)(p)) = \{i + 1: i \in \calT_{\Phi}(p)\}.
  \]
\end{lemma} 
\begin{proof}
It suffices to show that $\Phi(a_1)(\extremalarg{s_i}(p))=\extremalarg{s_{i+1}}(\Phi(a_1)(p))$ for all $s_i\in\extgen{n}$ and all punctures $p\in\disk{m}$. Clearly, $p$ is interior to a curve $\gamma$ if and only if $\Phi(a_1)(p)$ is interior to $\Phi(a_1)(\gamma)$. Therefore, it suffices to show that $\Phi(a_1)(\extremalarg{s_i})=\extremalarg{s_{i+1}}$, which is precisely the content of Lemma~\ref{lem:alpha:standardrootaction}.
\end{proof} We now study the action of $\Phi(s_j)$ on punctures $p$ with $i \in \calT_{\Phi}(p)$ for $[s_i,s_j] = 1$.

\begin{lemma}\label{lem:alpha:externaltypepreservecomm}
Let $n \geq 5$ and $m \geq 3$.  Let $\Phi:\braid{n} \rightarrow \braid{m}$ be an externally periodic homomorphism.  Let $p \in \disk{m}$ be a puncture.  Let $s_i \in \extgen{n}$ such that $i \in \calT_{\Phi}(p)$ and let $s_j \in \extgen{n}$ such that $[s_j, s_i] = 1$.  Then $i \in \calT_{\Phi}\left(\Phi(s_j)(p)\right)$ and $i \in \calT_{\Phi} \big(\Phi(s_j^{-1})(p)\big)$.
\end{lemma}

\begin{proof}
Let $\gamma \in \extremalarg{s_i}(p)$.  Then $\Phi(s_j)(\gamma) \in \extremalarg{s_i}$ by Lemma~\ref{lem:alpha:extremalaction}.  Furthermore we have $\Phi(s_j)(p) \in \Int(\Phi(s_j)(\gamma))$ by definition.  We conclude that $i \in \calT_{\Phi}(\Phi(s_j)(p))$ as desired.  Similarly, $\Phi(s_j^{-1})(\gamma) \in \extremalarg{s_i}$ by Lemma~\ref{lem:alpha:extremalaction}, so $i\in\calT_{\Phi}\big(\Phi(s_j^{-1})(p)\big)$.
\end{proof}

Let $n \geq 5$ and $m \geq 3$ and $\Phi:\braid{n} \rightarrow \braid{m}$ is an externally periodic and minimally typed homomorphism. Our goal now is to prove Lemma~\ref{lem:alpha:welltypednouni}, which says that $\calT_{\Phi}(p)$ is never a singleton. See Figure~\ref{fig:unitypepunc} for an example of what we seek to rule out.
\begin{figure}[ht]
    \centering
    \begin{tikzpicture}
        \node at (0,0)[anchor = south west]{\includegraphics[scale=0.6]{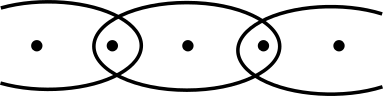}};
        \node at (-0.5, 0.9){\large $\cdots$};
        \node at (6.7, 0.9){\large$\cdots$};
        \node at (0.7, 1.8){$\gamma_1$};
        \node at (5.2,1.8){$\gamma_3$};
        \node at (3,1.8){$\gamma_2$};
        \node at (2.8,0.9){$p$};
    \end{tikzpicture}
    \caption{$\gamma_i \in \extremalarg{s_i}$ and $\calT_{\Phi}(p) = \{2\}$}\label{fig:unitypepunc}
\end{figure}

We now introduce some notation.  Let
\[
  \calN_i(p) = \{\delta \in \crs{\Phi(s_i)}: p \in \Int(\delta)\}.
\]
Let $\calN_i(p)^{\outmulti}$ denote the set of maximal curves in $\calN_i(p)$, where we think of $\calN_i(p)$ as a partially ordered set with the ordering $\leq$ from Section~\ref{section:external}.  Note that if $\calN_i(p)$ is nonempty then $\calN_i(p)$ has a unique maximal element.  Let $\calN(p) \coloneqq \bigcup_{i=1}^n \calN_i(p)$. 

\begin{lemma}\label{lem:alpha:welltypednouniaux}
Let $n \geq 5$ and $m \geq 3$.  Let $\Phi:\braid{n} \rightarrow \braid{m}$ be an externally periodic and minimally typed homomorphism. Let $p \in \disk{m}$ be a puncture and suppose that $\{i\} =\calT_{\Phi}(p)$ for some $s_i \in \extgen{n}$. Let $\gamma \in \extremalarg{s_i}(p)$.  Then $\delta\leq\gamma$ for all $\delta \in \calN(p)$.
\end{lemma}

\begin{proof}
    Suppose by way of contradiction we have an $s_j \in \extgen{n}$ and $\delta \in \calN_j(p)$ such that $\delta$ intersects $\gamma$.  Since $\crs{\Phi(s_i)}$ is a multicurve, we must have $s_j \neq s_i$.  Assume that $\delta$ is maximal among all such counterexamples.  By assumption $\calT_{\Phi}(p) = \{i\}$, so $\delta$ cannot be $\Phi$-maximal.  Therefore there is some $s_k \in \extgen{n}$  and $\gamma_k \in \extremalarg{s_k}$ such that $\gamma_k > \delta$. Since $\delta$ is a maximal counterexample to the lemma, $\gamma_k$ is disjoint from $\gamma$.  Furthermore, $p \in \Int(\delta) \subseteq \Int(\gamma_k)$ so $\gamma_k \geq \gamma$.  Since $\gamma\in\extremal$ we therefore have $\gamma_k = \gamma$. By Lemma~\ref{lem:alpha:extremalcollision} we have $\gamma \in \extremalarg{s_\ell}$ for all $s_\ell \in \extgen{n}$.  In particular, this implies that $\Phi$ is not minimally typed, a contradiction.
\end{proof}

We now study the type of $\Phi(f)(p)$ for $f\in\braid{n}$ when $\calT_{\Phi}(p)$ is a singleton.

\begin{lemma}\label{lem:alpha:nontypepreserve}
Let $n \geq 5$ and $m \geq 3$.  Let $\Phi:\braid{n} \rightarrow \braid{m}$ be an externally periodic homomorphism.  Let $p \in \disk{m}$ be a puncture such that $\calT_{\Phi}(p) = \{i\}$. Then $i\in\calT_{\Phi}(\Phi(f)(p))$ for all $f\in\braid{n}$.
\end{lemma}

\begin{proof}
    It suffices to prove this when $f\in\{s_j,s_j^{-1}\}$ for some $s_j\in\extgen{n}$.
    If $[s_i,s_j] = 1$ then this follows from Lemma~\ref{lem:alpha:externaltypepreservecomm}, so assume that $[s_i,s_j] \neq 1$. By applying the following argument with $s_j$ replaced by $s_{j}^{-1}$, it suffices to prove that $i \in \calT_{\Phi}(\Phi(s_j)(p))$.  Since $\Comm_{\braid{n}}(\extgen{n})$ has diameter 2, there is an $s_k \in \extgen{n}$ such that $[s_k,s_i] = [s_k,s_j] = 1$.  Let
    \[
      \delta \in \left(\calN_j(p) \cup \calN_k(p)\right)^{\outmulti}.
    \]
    Such a $\delta$ is unique if it exists since $\calN_j(p) \cup \calN_k(p)$ is a multicurve by Lemma~\ref{lem:crs:commdisj}.  If no such $\delta$ exists, i.e., if $\calN_j(p) \cup \calN_k(p)=\emptyset$, we let $\delta$ denote the inessential curve bounding the puncture $p$.  Let $\gamma \in \extremalarg{s_i}(p)$.  Since $\calT_{\Phi}(p)$ is a singleton, Lemma~\ref{lem:alpha:welltypednouniaux} implies that $\delta \leq \gamma$.
    
    The curve $\delta$ is exterior to $\crs{\Phi(s_j)}$ and $\crs{\Phi(s_k)}$ by construction, so $\Phi(s_j)(\delta) = \Phi(s_k)(\delta)$ by Lemma~\ref{lem:alpha:exterioragreement}.  In particular, we have
    \[
      \Phi(s_j)(p) \in \Int(\Phi(s_k)(\delta)).
    \]
    On the other hand, since $[s_k,s_i] = 1$ we have $\Phi(s_k)(\gamma) \in \extremalarg{s_i}$ by Lemma~\ref{lem:alpha:extremalaction}.  Therefore
    \[
      \Phi(s_j)(\delta) = \Phi(s_k)(\delta) \leq \Phi(s_k)(\gamma) \in \extremalarg{s_i}.
    \]
    We conclude that $i \in \calT_{\Phi}(\Phi(s_j)(p))$.

    The argument for $\Phi(s_j^{-1})$ follows by the same line of reasoning.
\end{proof}

We now show that $\calT_{\Phi}(p)$ can never be a singleton.

\begin{lemma}\label{lem:alpha:welltypednouni}
    Let $n \geq 5$ and $m \geq 3$.  Let $\Phi:\braid{n} \rightarrow \braid{m}$ be an externally periodic and minimally typed homomorphism. Let $p \in \disk{m}$ be a puncture and suppose that $i \in \calT_{\Phi}(p)$ for some $s_i \in \extgen{n}$.  Then $\calT_{\Phi}(p)=\{i-1,i\}$ or $\calT_{\Phi}(p)=\{i,i+1\}$.
\end{lemma}

\begin{proof}
  Suppose by way of contradiction that $\calT_{\Phi}(p) = \{i\}$ for some $s_i \in \extgen{n}$.  Let
  \[
    \delta \in \left(\calN_{i-1}(p) \cup \calN_{i+1}(p)\right)^{\outmulti}
  \]
  or set $\delta$ to be the inessential curve surrounding $p$ if no such curve exists.  Since $[s_{i-1}, s_{i+1}] = 1$, the set $\calN_{i-1}(p) \cup \calN_{i+1}(p)$ is a multicurve by Lemma~\ref{lem:crs:commdisj}.  In particular, $\delta$ must be unique.  Furthermore, by construction $\delta$ is exterior to $\crs{\Phi(s_{i-1})}$ and $\crs{\Phi(s_{i+1})}$.  By Lemma~\ref{lem:alpha:exterioragreement} we conclude that $\Phi(s_{i-1})(\delta) = \Phi(s_{i+1})(\delta)$.  Let $\delta' \coloneqq \Phi(s_{i-1})(\delta)$.

  Let $\gamma \in \extremalarg{s_i}(p)$.  By Lemma~\ref{lem:alpha:welltypednouniaux} we have $\delta \leq \gamma$.  Consider the two curves
  \[
    \gamma_{i-1} \coloneqq \Phi(s_{i}s_{i-1})(\gamma) \text{ and } \gamma_{i+1} \coloneqq \Phi(s_{i}s_{i+1})(\gamma).
  \]
  By the braid relation and Lemma~\ref{lem:crs:functorial}, we have $\gamma_{i-1} \in \crs{\Phi(s_{i-1})}$.  Furthermore, by Lemma~\ref{lem:alpha:extremalbraid} we have $\gamma_{i+1} \in \extremalarg{s_{i+1}}$.  In particular, by Lemma~\ref{lem:crs:commdisj} we have $\gamma_{i-1}$ disjoint from $\gamma_{i+1}$.  By construction we have
  \[
    \Phi(s_{i})(\delta') \leq \gamma_{i-1} \text{ and } \Phi(s_i)(\delta') \leq \gamma_{i+1}.
  \]
  In particular, we must have $\gamma_{i-1} \leq \gamma_{i+1}$ since $\gamma_{i+1} \in \extremalarg{s_{i+1}}$.  Since $\gamma_{i-1}$ and $\gamma_{i+1}$ both have the same number of interior punctures as $\gamma$, we conclude that $\gamma_{i-1} = \gamma_{i+1}$.  However this implies that $\Phi$ is not minimally typed, a contradiction.  
\end{proof}  

As a consequence of Lemma~\ref{lem:alpha:welltypednouni}, we see that if $\Phi:\braid{n} \rightarrow \braid{m}$ is an externally periodic and minimally typed homomorphism then any  puncture $p$ in  $\disk{m}$ either has $\calT_{\Phi}(p) =\{i,i+1\}$ for some $s_i \in \extgen{n}$ or has $\calT_{\Phi}(p) = \emptyset$.  We can now describe how $\Phi(s_j)$ acts on any puncture $p$ with $j \not \in \calT_{\Phi}(p)\neq\emptyset$. 

\begin{lemma}\label{lem:alpha:welltypeactionnotintype}
    Let $n \geq 5$ and $m \geq 3$.  Let $\Phi:\braid{n} \rightarrow \braid{m}$ be an externally periodic and minimally typed homomorphism.  Let $p \in \disk{m}$ be a puncture with $\calT_{\Phi}(p) = \{i,i+1\}$ and let $s_j \in \extgen{n}$ such that $j \not \in \calT_{\Phi}(p)$.  Then $\Phi(s_j)(p)$ and $\Phi(s_j^{-1})(p)$ are also of type $\{i,i+1\}$.
\end{lemma}

\begin{proof}
    If $[s_j, s_i] = [s_j,s_{i+1}] = 1$, then this follows from Lemma~\ref{lem:alpha:extremalaction}. Therefore it suffices to deal with the cases $j \in\{i-1, i+2\}$.  We begin with $s_j = s_{i-1}$.  The proof proceeds the same way for $s_{i-1}$ and $s_{i-1}^{-1}$, so we indicate this by $s_{i-1}^{\pm}$.  Let $p' = \Phi(s_{i-1}^{\pm1})(p)$.  Suppose by way of contradiction that $p'$ is not of type $\{i,i+1\}$.  Note that since $n \geq 5$, $s_{i-1}$ and $s_{i+1}$ commute, so $i+1 \in \calT_{\Phi}(p')$ by Lemma~\ref{lem:alpha:externaltypepreservecomm}.  By Lemma~\ref{lem:alpha:welltypednouni} we have $\calT_{\Phi}(p') = \{i+1,i+2\}$.  But then  since $n \geq 5$, we have $[s_{i-1}, s_{i+2}] = 1$, so by Lemma~\ref{lem:alpha:externaltypepreservecomm} and since $p=\Phi(s_{i-1}^{\mp1})(p')$ we have $i+2\in\calT_{\Phi}(p)$.  This contradicts our hypothesis that $\calT_\Phi(p)=\{i,i+1\}$.

    The case that $s_j = s_{i+2}$ follows similarly, except we show that if $\Phi(s_{i+2}^{\pm1})(p)\neq\{i,i+1\}$, then $i-1\in\calT_{\Phi}(p)$ which is a contradiction.
\end{proof}

We now assemble the previous lemmas into the following description of externally periodic and minimally typed homomorphisms.  

\begin{theorem}\label{thm:alpha:welltypedpuncture}
    Let $n \geq 5$ and $m\geq 3$. Let $\Phi:\braid{n} \rightarrow \braid{m}$ be an externally periodic and minimally typed homomorphism.  Let $p \in \disk{m}$ be a puncture.  Then either:
    \begin{itemize}
        \item $\calT_{\Phi}(p) = \emptyset$; or
        \item $\calT_{\Phi}(p) = \{i,i+1\}$ for some $s_i \in \extgen{n}$.
    \end{itemize} Furthermore, suppose that $\calT_{\Phi}(p) = \{i,i+1\}$ for some $s_i \in \extgen{n}$.  The following hold:
    \begin{itemize}
        \item $\Phi(s_{i+1})(p)$ and $\Phi(s_{i+1}^{-1})(p)$ are of type $\{i+1,i+2\}$;
        \item $\Phi(s_{i})(p)$ and $\Phi(s_i^{-1})(p)$ are of type $\{i-1,i\}$; and
        \item $\Phi(s_j)(p)$ and $\Phi(s_j^{-1})(p)$ are of type $\{i,i+1\}$ for all $s_j\in\extgen{n}\setminus\{s_i,s_{i+1}\}$.
    \end{itemize}
\end{theorem}

\begin{proof}
    The first list follows from Lemma~\ref{lem:alpha:welltypednouni} and the definition of minimally typed.

    The third point in the second list is the content of Lemma~\ref{lem:alpha:welltypeactionnotintype}, so it suffices to prove the first two points.
    
    In light of Lemmas~\ref{lem:alpha:externaltypepreservecomm} and~\ref{lem:alpha:welltypednouni} it suffices to show that neither $\Phi(s_{i+1})$, $\Phi(s_i)$, nor their inverses, preserve the type of any puncture of type $\{i,i+1\}$. In fact it suffices to just show that $\Phi(s_{i+1})$ and $\Phi(s_i)$ do not preserve the type of any puncture of type $\{i,i+1\}$, since $\Phi(s_k^{-1})$ preserves the type of $p$ if and only if $\Phi(s_k)$ preserves the type of $\Phi(s_k^{-1})(p)$.

    Suppose that $\calT_\Phi(p)=\{i,i+1\}$. We first show $\Phi(s_{i+1})(p)$ does not have type $\{i,i+1\}$. Write
    \[
    s_{i+1} = a_1s_{i+n-1}^{-1}s_{i+n-2}^{-1}\cdots s_{i+2}^{-1}.
    \]
    By Lemma~\ref{lem:alpha:welltypeactionnotintype} all terms in this expression for $s_{i+1}$, except for $a_1$, preserve the type of punctures of type $\{i,i+1\}$. However $\Phi(a_1)$ does not preserve the type of any puncture of type $\{i,i+1\}$ by Lemma~\ref{lem:alpha:periodictypepreserve}. Therefore $\Phi(s_{i+1})(p)$ does not have type $\{i,i+1\}$.

    It remains to show that $\calT_{\Phi}(\Phi(s_i)(p)) \neq \{i,i+1\}$. Suppose for a contradiction that it did. Write
    \[
      a_1 = s_{i+2}\cdots s_{i+n}.
    \]
    Every term other than $s_i=s_{i+n}$ in the above expression for $a_1$ preserves the type of punctures of type $\{i,i+1\}$ by Lemma~\ref{lem:alpha:welltypeactionnotintype}. Since we assumed that $\Phi(s_i)(p)$ has type $\{i,i+1\}$, it follows that $\Phi(a_1)(p)$ is of type $\{i,i+1\}$. However, this contradicts Lemma~\ref{lem:alpha:periodictypepreserve}. Therefore $\Phi(s_i)(p)$ does not have type $\{i,i+1\}$.
\end{proof}

We now recover the following theorem of Castel~\cite[Theorem 1.1.1(ii)]{Cas16} in the case $n \geq 6$ and Chen--Kordek--Margalit~\cite[Theorem 8.1]{CKM} for $n = 5$.

\begin{theorem}\label{thm:case:m:leq:n}
    Let $n\geq 5$ and $m\geq 3$.  Assume that $m \leq n$.  Let $\Phi:\braid{n}\to \braid{m}$ be a non-cyclic homomorphism. Then $m=n$ and $\Phi$ is centrally equivalent to the identity.
\end{theorem}
\begin{proof} By Theorem~\ref{thm:lin:mlessn}, $\Phi$ is irreducible and $m=n$. By Lemmas~\ref{lem:external:anosov} and~\ref{lem:alpha:nonwelltypedirred}, $\Phi$ is externally periodic and minimally typed. By Theorem~\ref{thm:alpha:welltypedpuncture}, for every puncture $p\in\disk{m}$ either $\calT_{\Phi}(p) = \emptyset$ or $\calT_{\Phi}(p) = \{i,i+1\}$ for some $s_i \in \extgen{n}$.  Since $\Phi$ is non-cyclic, Lemma~\ref{lem:centralizer:emptycrs} gives $\crs{\Phi(s_i)} \neq \emptyset$ for all $s_i \in \extgen{n}$.  Therefore for each $i \in \extgen{n}$ there are at least two $p \in \disk{m}$ with $i \in \calT_{\Phi}(p)$.  Since $m = n$ we conclude that:
    \begin{itemize}
        \item there are no distinct $p, p' \in \disk{m}$ with $\calT_{\Phi}(p) = \calT_{\Phi}(p')$;
        \item there are no $p\in\disk{m}$ with $\calT_{\Phi}(p)=\emptyset$;
        \item each $\crs{\Phi(s_i)}$ is a singleton; and
        \item the unique $\gamma_i \in \crs{\Phi(s_i)}$ contains exactly two punctures.
    \end{itemize}
    Theorem~\ref{thm:alpha:welltypedpuncture} and the first point above imply that the element $\Ext(\Phi(s_i))$ fixes at least two punctures, since $n\geq5$.  Since $\Phi$ is externally periodic, $\Ext(\Phi(s_i))\in\Ext_{\{\gamma_i\}}(Z(\braid{m}))$ by Lemma~\ref{lem:centralizer:fixtwocurves}. By Theorem~\ref{thm:alpha:welltypedpuncture}, $\Phi(s_i)$ exchanges the punctures in $\Int(\gamma_i)$, hence $\Phi(s_i)$ is the product $T_{\partial \disk{m}}^k$ with an odd power $h_{\gamma_i}^\ell$ of the half-twist $h_{\gamma_i}$ about $\gamma_i$, with $k$ and $\ell$ independent of $s_i\in\extgen{n}$. By passing to a double cover of $\disk{m}$ and applying the work of Birman--Hilden~\cite{BH71}, one obtains an injective homomorphism $\braid{m}\to\Mod_g$ with $g=\lceil m/2\rceil$, which maps $h_{\gamma_i}$ to a Dehn twist $T_{\alpha_i}$ about a non-separating curve $\alpha_i\subset S_g$ for each $s_i\in\extgen{n}$. Furthermore, if $\iota(\alpha_i,\alpha_{i+1})=1$ then $\iota(\gamma_i,\gamma_{i+1})=2$. By the braid relation and McCarthy's proof of~\cite[Lemma 4.3]{McCarthyDehnint} (see also~\cite[Section 3.5.1]{FM}) we indeed have that $\iota(\alpha_i,\alpha_{i+1})=1$ and also that $\ell\in\{-1,1\}$. Applying the change of coordinates principle we find that $\Phi$ is centrally equivalent to the identity.
\end{proof}

\section{Interior regions of maximal curves}\label{section:curves} 
Let $\Phi:\braid{n} \rightarrow \braid{m}$ be an externally periodic and minimally typed homomorphism with $n \geq 5$ and $m\geq 3$. Given two curves $\gamma, \gamma' \subset \disk{m}$  we define the finite set
\[\Delta(\gamma, \gamma') \coloneqq \left(\{\delta:\delta \leq \gamma \}\cap \{\delta:\delta \leq \gamma'\}\right)^{\outmulti}. \]  Note that if $\epsilon$ is a curve that intersects some $\delta \in \Delta(\gamma, \gamma')$ then $\epsilon$ intersects at least one of $\gamma$ or $\gamma'$. Indeed, such a $\delta$ must be the boundary of some connected component of $\disk{m}\cut(\gamma\cup\gamma')$. The main object in this section is the finite set
\[\Delta(\Phi) \coloneqq \bigcup_{\substack{\gamma, \gamma'\in \extremal \\ \gamma\ne \gamma'}} \Delta(\gamma,\gamma').\] If a homomorphism $\Phi$ has $\Delta(\Phi) = \emptyset$, then no curve $\delta$ is interior to two distinct $\gamma, \gamma' \in \extremal$.  We prove in Proposition~\ref{prop:curves:valid} that if $\Delta(\Phi)$ is non-empty, then $\Phi(\braid{n})$ is reducible.  See Figure~\ref{fig:typecurveex} for an example of a curve $\delta \in \Delta(\Phi)$.

\begin{figure}[ht]
    \centering
    \begin{tikzpicture}
        \node at (0,0)[anchor = south west]{\includegraphics[scale=0.6]{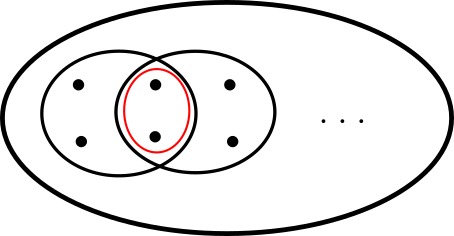}};
        \node at (1.1,2.1){$\gamma_1$};
        \node at (4.2,2.1){$\gamma_2$};
        \node at (2.3,2.1){{\color{red} $\delta$}};
    \end{tikzpicture}
    \caption{If $\gamma_1 \in \extremalarg{s_1}$ and $\gamma_2 \in \extremalarg{s_2}$, then $\{{\color{red}\delta}\} \in \Delta(\Phi)_{1,2}$}\label{fig:typecurveex}
\end{figure}

\p{Types of curves} Let $\delta \in \Delta(\Phi)$. We define $\extremalarg{s_i}(\delta)\coloneqq\{\gamma\in\extremalarg{s_i} : \delta\leq\gamma \}$, and $\extremal(\delta)\coloneqq\bigcup_{s_i\in\extgen{n}}\extremalarg{s_i}(\delta)$. Note that $\extremalarg{s_i}(\delta)$ is either empty or a singleton. We define the \defn{type} of $\delta$ to be
\[
  \calT_{\Phi}(\delta)\coloneqq\{i:\extremalarg{s_i}(\delta)\neq\emptyset\}.
\]
We show that this notion of type behaves similarly to the notion of the type of a puncture defined in Section~\ref{section:alpha}.

\begin{lemma}\label{lem:curves:mintype}
Let $n \geq 5$ and $m \geq 3$.  Let $\Phi\colon\braid{n}\to\braid{m}$ be an externally periodic and minimally typed homomorphism. Let $\delta\in\Delta(\Phi)$.  Then there is an $s_i \in \extgen{n}$ such that $\calT_{\Phi}(\delta) = \{i,i+1\}$.  Furthermore, $\calT_{\Phi}(p) = \calT_{\Phi}(\delta)$ for any $p \in \Int(\delta)$.
\end{lemma}

\begin{proof}
    Let $s_i,s_j\in\extgen{n}$, $\gamma\in\extremalarg{s_i}$, and $\gamma'\in\extremalarg{s_j}$ with $\gamma\neq\gamma'$ and $\delta\in\Delta(\gamma,\gamma')$. If $[s_i,s_j]=1$, then $\extremalarg{s_i}\cup\extremalarg{s_j}$ is a multicurve by Lemma~\ref{lem:crs:commdisj}, hence $\Delta(\gamma,\gamma')=\emptyset$. Therefore, without loss of generality we must have $j=i+1$, so $\{i,i+1\}\subseteq\calT_{\Phi}(\delta)$.

    Let $p\in\Int(\delta)$. By the definition of types we have $\calT_\Phi(\delta)\subseteq\calT_\Phi(p)$, and so
    \[
    \{i,i+1\} \subseteq \calT_\Phi(\delta) \subseteq \calT_\Phi(p).
    \]
    Since $\Phi$ is minimally typed, we must in fact have $\{i,i+1\}=\calT_\Phi(\delta)=\calT_\Phi(p)$, as required.
\end{proof} 
We let
\[
  \Delta(\Phi)_{i,i+1} = \bigcup_{\gamma \in \extremalarg{s_i}, \gamma' \in \extremalarg{s_{i+1}}} \Delta(\gamma, \gamma').
\]
As a consequence of Lemma~\ref{lem:curves:mintype} we have $\Delta(\Phi)_{i,i+1} = \left\{\delta \in \Delta(\Phi): \calT_{\Phi}(\delta) = \{i,i+1\}\right\}$.  Furthermore we have
\[
  \Delta(\Phi) = \bigcup_{s_i \in \extgen{n}}\Delta(\Phi)_{i,i+1}.
\]
We develop some more consequences of Lemma~\ref{lem:curves:mintype}.

\begin{lemma}\label{lem:curves:typeexterior}
Let $n \geq 5$ and $m \geq 3$.  Let $\Phi:\braid{n} \rightarrow \braid{m}$ be an externally periodic and minimally typed homomorphism.  Let $s_i \in \extgen{n}$ and $s_j\in\extgen{n}\setminus\{s_{i-1},s_i\}$. Then $\Delta(\Phi)_{i-1,i}$ is exterior to $\extremalarg{s_j}$.
\end{lemma}

\begin{proof}
    Let $\delta\in\Delta(\Phi)_{i-1,i}$ and $\gamma_j\in\extremalarg{s_j}$. Since $n \geq 5$, we have $[s_j,s_i] = 1$ or $[s_j, s_{i-1}] = 1$.  Without loss of generality, assume that  $[s_j,s_i] = 1$.  Pick $\gamma_i \in \extremalarg{s_i}$ with $\delta \leq \gamma_i$.  Then $\gamma_j$ is exterior to $\gamma_i$ by Lemma~\ref{lem:crs:commdisj} and the definition of $\extremal$.  Since $\delta \leq \gamma_i$, we conclude that $\delta$ is exterior to $\gamma_j$, as desired.
\end{proof}

We start by studying the action of $\Phi(s_j)$ on $\Delta(\Phi)_{i,i+1}$.  

\begin{lemma}\label{lem:fixing:mcal1}
    Let $n \geq 5$ and $m \geq 3$.  Let $\Phi:\braid{n}\to \braid{m}$ be an externally periodic and minimally typed homomorphism. Let $s_i, s_j \in \extgen{n}$ such that $[s_j, s_i]=[s_j, s_{i+1}]=1$.  If $\delta \in \Delta(\Phi)$ has type $\calT_{\Phi}(\delta)=\{i, i+1\}$, then $\Phi(s_j)(\delta)\in \Delta(\Phi)$ and $\calT_{\Phi}(\Phi(s_j)(\delta))=\calT_{\Phi}(\delta)$. 
\end{lemma}
\begin{proof}
    If $[s_j,s_i]=[s_j,s_{i+1}]=1$ then $\Phi(s_j)$ fixes both $\extremalarg{s_i}$ and $\extremalarg{s_{i+1}}$ setwise by Lemma~\ref{lem:alpha:extremalaction}. As a consequence, $\Phi(s_j)$ stabilizes the set of curves $\{\delta\in \Delta(\Phi):\;\calT_{\Phi}(\delta)=\{i, i+1\}\}$.
\end{proof}

\p{Reducing system} Suppose  that $\Phi:\braid{n} \rightarrow \braid{m}$ is an externally periodic and minimally typed homomorphism with $n \geq 5$ and $m \geq 3$.  Let $\delta \in \Delta(\Phi)_{i,i-1}$.  Consider the disconnected surface
\[
  \Sigma = \disk{m} \cut \left(\crs{\Phi(s_{i-2})} \cup \crs{\Phi(s_{i+1})} \cup \Delta(\Phi)_{i-1,i}\right).
\]
Let $\mu_\delta\subset\disk{m}$ be denote the unique curve (if it exists) satisfying the following:
\begin{itemize}
    \item $\mu_{\delta}$ is isotopic to a boundary component of $\partial \Sigma$;
    \item $\mu_{\delta}$ is exterior to $\crs{\Phi(s_{i-2})}$ and $\crs{\Phi(s_{i+1})}$;
    \item $\mu_{\delta} \geq \delta$; and
    \item $\mu_{\delta}$ is minimal over all curves that satisfy the above three properties.
\end{itemize} See Figure~\ref{fig:calWexample} for an example.  Note that \emph{a priori} $\mu_{\delta}$ might not exist.  In particular, we might have $\mu_{\delta}$ inessential by virtue of being homotopic to $\partial \disk{m}$.  We verify that this does not occur in the proof of Lemma~\ref{lem:curves:mumulticurve}.

\p{Alternative description of $\mu_{\delta}$} One can also think of $\mu_\delta$ in the following way.  Let $\delta \in \Delta(\Phi)_{i-1,i}$ for some $s_i \in \extgen{n}$.  Consider the finite set
\[
  K_{i-1,i}\coloneqq \crs{\Phi(s_{i-2})} \cup \Delta(\Phi)_{i-1,i} \cup \crs{\Phi(s_{i+1})}.
\]
Let $\calC_{\delta} \subseteq K_{i-1,i}$ denote the minimal subset of $K_{i-1,i}$ such that the following hold:
\begin{itemize}
    \item $\delta \in \calC_{\delta}$; 
    \item if $\epsilon \in \calC_{\delta}$ and $\epsilon' \in K_{i-i,i}$ satisfies $\epsilon \leq \epsilon'$ or $\epsilon' \leq \epsilon$ then $\epsilon' \in \calC_{\delta}$; and
    \item if $\epsilon \in\calC_{\delta}$ and $\epsilon' \in K_{i-1,i}$ such that $\epsilon$ intersects $\epsilon'$ then $\epsilon' \in \calC_{\delta}$.
\end{itemize} Then $\mu_{\delta}$ is the minimal curve (if it exists) that satisfies:
\begin{itemize}
    \item $\mu_{\delta} \geq \delta$; and
    \item $\mu_{\delta}$ is isotopic to a boundary component of $\disk{m} \cut \calC_{\delta}$.
\end{itemize}Note that we  have $\mu_{\delta} = \mu_{\delta'}$ for $\delta, \delta' \in \Delta(\Phi)$ with $\calC_{\delta} = \calC_{\delta'}$.  

\begin{figure}[ht]
    \centering
    \begin{tikzpicture}
        \node at (0,0)[anchor = south west]{\includegraphics[scale=0.6]{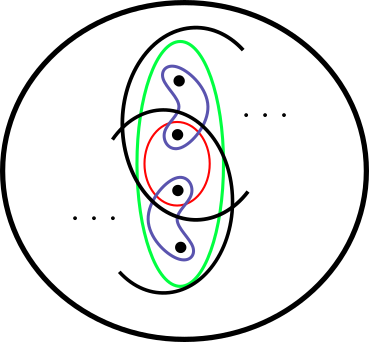}};
        \node at (3.5,1) {$\gamma_1$};
        \node at (1.9, 4) {$\gamma_2$};
        \node at (4.1, 4) {{\color{green}$\mu_{\delta}$}};
        \node at (3, 3) {{\color{red}$\delta$}};
        \node at (3.2, 4) {{\color{blue}$\epsilon_n$}};
        \node at (2.85, 2) {{\color{blue}$\epsilon_3$}};
    \end{tikzpicture}
    \caption{If {\color{red}$\delta$} $\in \Delta(\Phi)_{1,2}$ and {\color{blue} $\epsilon_n$} $\in \crs{\Phi(s_n)}$ and {\color{blue} $\epsilon_3$} $\in \crs{\Phi(s_3)}$ then {\color{green} $\mu_\delta$} $\in \calW(\Phi)_{1,2}$.  We have $\gamma_1 \in \extremalarg{s_1}$ and $\gamma_2\in \extremalarg{s_2}$}\label{fig:calWexample}
\end{figure} 

Let
\[
  \calW(\Phi)_{i,i-1} = \{\mu_\delta: \delta \in \Delta(\Phi)_{i,i-1}\}
\]
and let
\[
  \calW(\Phi) = \bigsqcup_{s_i \in \extgen{n}} \calW(\Phi)_{i,i-1}.
\]
Note that the set $\calW(\Phi)_{i,i-1}$ is finite since it is a subset of the set boundary components of a surface.  We will show that if $\Delta(\Phi) \neq \emptyset$ then $\calW(\Phi)$ is a reducing system for $\Phi$.  We begin with the following result.

\begin{lemma}\label{lem:curves:intersectioninterior}
Let $n \geq 5$ and $m \geq 3$.  Let $\Phi:\braid{n}\rightarrow\braid{m}$ be an externally periodic and minimally typed homomorphism. Let $s_i\in\extgen{n}$ and suppose there is some $\delta\in\Delta(\Phi)_{i-1,i}$. Suppose further that there is some $s_j\in\extgen{n}$ and $\epsilon\in\crs{\Phi(s_j)}$ such that $\epsilon$ intersects $\delta$. If $[s_i,s_j]=1$, then $s_j\in\{s_{i-2},s_i\}$ and there is a $\gamma_i\in\extremalarg{s_i}$ such that $\epsilon\leq\gamma_i$. Similarly, if $[s_{i-1},s_j]=1$, then $s_j\in\{s_{i-1},s_{i+1}\}$ and there is a $\gamma_{i-1}\in\extremalarg{s_{i-1}}$ such that $\epsilon\leq\gamma_{i-1}$.
\end{lemma}

\begin{proof}
    Without loss of generality, assume $[s_i,s_j]=1$. Let $\gamma_i \in \extremalarg{s_i}$ such that $\delta \leq \gamma_i$. Since $[s_i,s_j] = 1$,  we have $\epsilon$ disjoint from $\crs{\Phi(s_i)}$ by Lemma~\ref{lem:crs:commdisj}. In particular, $\epsilon$ is disjoint from $\gamma_i$, so the fact that $\epsilon$ intersects $\delta$ implies $\epsilon \leq \gamma_i$.  Furthermore, if $s_j \notin \{s_{i -2},s_i\}$, then $[s_j, s_{i-1}] = 1$ so $\epsilon$ would be disjoint from $\crs{\Phi(s_{i-1})}$ by Lemma~\ref{lem:crs:commdisj}.  In particular $\epsilon$ would be disjoint from $\extremalarg{s_{i-1}}$ and $\extremalarg{s_i}$, so we would have $\epsilon$ disjoint from $\delta$, a contradiction.
\end{proof}  We now combine the previous results to prove the following.

\begin{lemma}\label{lem:curves:muexterior}
Let $n \geq 5$ and $m \geq 3$.  Let $\Phi:\braid{n} \rightarrow \braid{m}$ be an externally periodic and minimally typed homomorphism. Let $s_i\in\extgen{n}$ and $s_j\in\extgen{n}\setminus\{s_{i-1},s_i\}$. Then $\calW(\Phi)_{i-1,i}$ is exterior to $\extremalarg{s_j}$.
\end{lemma}

\begin{proof}
  Let $\mu\in\calW(\Phi)_{i-1,i}$. Recall $\mu$ is by definition isotopic to a boundary component of
  \[
    \Sigma \coloneqq \disk{m} \cut \left(\Delta(\Phi)_{i-1,i} \cup \crs{\Phi(s_{i-2})} \cup \crs{\Phi(s_{i+1})}\right).
  \]
  Furthermore, $\mu \geq \delta$ for some $\delta \in \Delta(\Phi)_{i-1,i}$.  Now, $\Delta(\Phi)_{i-1,i}$ is exterior to $\extremalarg{s_j}$ by Lemma~\ref{lem:curves:typeexterior}.  It therefore suffices to show that if $\epsilon \in \crs{\Phi(s_{i-2})} \cup \crs{\Phi(s_{i+1})}$ is chosen such that $\epsilon$ intersects $\Delta(\Phi)_{i-1,i}$, then $\epsilon$ is also exterior to $\extremalarg{s_j}$.  We prove this for $\epsilon \in \crs{\Phi(s_{i+1})}$ and then explain how to change the argument for $\epsilon_{i-2}$.

    Let $\epsilon_{i+1} \in \crs{\Phi(s_{i+1})}$ such that $\epsilon_{i+1}$ intersects $\Delta(\Phi)_{i-1,i}$.  We have $[s_{i-1}, s_{i+1}] = 1$, so by Lemma~\ref{lem:curves:intersectioninterior}, $\epsilon_{i+1}$ is interior to some $\gamma_{i-1} \in \extremalarg{s_{i-1}}$.   We first deal with the case that $s_j = s_{i-2}$.  Since $n \geq 5$ we also have $[s_{i+1}, s_{i-2}]$ so $\epsilon_{i+1}$ is disjoint from $\extremalarg{s_{i-2}}$.  Now, since $\extremalarg{s_{i-2}}$ and $\extremalarg{s_i}$ are disjoint by Lemma~\ref{lem:crs:commdisj} and $\extremalarg{s_i} \cap \extremalarg{s_j} = \emptyset$ by Lemma~\ref{lem:alpha:extremalcollision}, we see that if $\epsilon_{i+1}$ were interior to $\extremalarg{s_{i-2}}$ then $\epsilon_{i+1}$ could not intersect $\extremalarg{s_i}$.  Therefore $\epsilon_{i+1}$ is exterior to $\extremalarg{s_{i-2}}$.  It remains to deal with the case that $s_j \neq s_{i-2}$.  Since $[s_{i-1},s_j] =  1$ for all such $s_j$ we have $\gamma_{i-1}$ disjoint from $\extremalarg{s_j}$ by Lemma~\ref{lem:crs:commdisj}.  The fact that $\Phi$ is minimally typed implies that $\gamma_{i-1} \not \in \extremalarg{s_j}$. In particular, $\epsilon_{i+1} \leq \gamma_{i-1}$ implies that $\epsilon_{i+1}$ is exterior to $\extremalarg{s_j}$, as desired.

    For $s_{i-2}$,  make the following changes in the above paragraph:
    \begin{itemize}
        \item $s_{i-2}$ replaces $s_{i+1}$;
        \item $s_{i-1}$ replaces $s_{i}$
        \item $s_{i}$ replaces $s_{i-1}$; and
        \item $s_{i+1}$ replaces $s_{i-2}$.
    \end{itemize}  The same argument as above then holds.
\end{proof} 
We provide conditions ensuring curves in $\crs{\Phi(s_{i-2})}$ and $\crs{\Phi(s_{i+2})}$ are disjoint. 

\begin{lemma}\label{lem:curves:munequals5}
    Let $n \geq 5$ and $m \geq 3$.  Let $\Phi:\braid{n} \rightarrow \braid{m}$ be an externally periodic and minimally typed homomorphism.  Let $s_i \in \extgen{n}$ and $\gamma_i \in \extremalarg{s_i}$.  Suppose there are $\epsilon_{i-2} \in \crs{\Phi(s_{i-2})}$ and $\epsilon_{i+2} \in \crs{\Phi(s_{i+2})}$ such that:
    \begin{itemize}
        \item $\epsilon_{i-2}$ and $\epsilon_{i+2}$ are interior to $\gamma_i$; 
        \item $\epsilon_{i-2}$ intersects $\extremalarg{s_{i-1}}$; and
        \item $\epsilon_{i+2}$ intersects $\extremalarg{s_{i+1}}$.
    \end{itemize} Then $\epsilon_{i-2}$ and $\epsilon_{i+2}$ are distinct and mutually exterior.
\end{lemma}

\begin{proof}
    If $n \geq 6$ then $[s_{i-2}, s_{i+2}] = 1$ so the result follows immediately from Lemma~\ref{lem:crs:commdisj}.  It therefore suffices to prove this in the case that $n = 5$.  Note that the following proof works for all $n \geq 5$. We begin with the following simpifying claim.
    
    \p{Claim} We may assume without loss of generality that there is no $\overline{\epsilon} \in \crs{\Phi(s_{i-2})}$ with $\overline{\epsilon} > \epsilon_{i-2}$ and similarly for $\epsilon_{i+2}$.  
    
    \p{Proof of claim} Indeed, since $\gamma_i\in\extremalarg{s_i}$ any $\overline{\epsilon} \in \crs{\Phi(s_{i-2})}$ with $\overline{\epsilon} > \epsilon_{i-2}$ has $\overline{\epsilon} \leq \gamma_i$, by Lemma~\ref{lem:crs:commdisj} and the definition of $\extremal$.  Since $\epsilon_{i-2}$ intersects $\extremalarg{s_{i-1}}$ we would have $\overline{\epsilon}$ still intersecting $\extremalarg{s_{i-1}}$.  Then if the lemma holds for $\overline{\epsilon}$ and  $\epsilon_{i+2}$ we have $\overline{\epsilon}$ exterior to $\epsilon_{i+2}$ and vice versa.  Since $\epsilon_{i-2} \leq \overline{\epsilon}$ we have $\epsilon_{i-2}$ and $\epsilon_{i+2}$ exterior to each other, as desired.\qed{}
    
    We now proceed with the proof.  Let $\epsilon' = \Phi(s_{i-2}^{-1})(\epsilon_{i-2})$ and let $\gamma_i' = \Phi(s_{i-2}^{-1})(\gamma_i)$.  Let ${\gamma}' = \Phi(s_{i-1}^{-1}s_{i+1})(\gamma_i')$.  The element
    \[
      s \coloneqq s_{i+1}s_{i-1}^{-1}s_is_{i-1}s_{i+1}^{-1}
    \]
    satisfies $[s, s_i] = 1$ so $\gamma'\in\crs{\Phi(s)}$ is disjoint from $\extremalarg{s_i}$ by Lemma~\ref{lem:crs:commdisj}.  Furthermore, since $\Phi$ is minimally typed, Theorem~\ref{thm:alpha:welltypedpuncture} implies that any $p \in \Int(\gamma')$ has $\calT_{\Phi}(p) = \{i+1,i+2\}$ or $\calT_{\Phi}(p) = \{i-2, i-1\}$.  Therefore, for any $p\in\Int(\gamma')$ one has $i \not \in \calT_{\Phi}(p)$, so $p$ is exterior to $\extremalarg{s_i}$.  Therefore $\gamma'$ is exterior to $\extremalarg{s_i}$, and thus the curve $\Phi(s_{i-1}^{-1}s_{i+1})(\epsilon')$ is also exterior to $\extremalarg{s_i}$.

    We first show that $\epsilon_{i-2}$ and $\epsilon_{i+2}$ are disjoint.  By way of contradiction, suppose that $\epsilon_{i-2}$ and $\epsilon_{i+2}$ intersect. See Figure~\ref{fig:5crossingexample} for the sort of crossing that might occur.

    \begin{figure}[ht]
    \centering
    \begin{tikzpicture}
        \node at (0,0)[anchor = south west]{\includegraphics[scale=0.6]{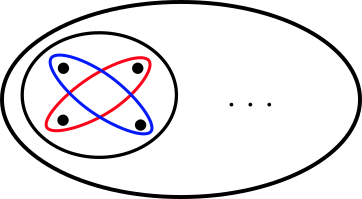}};
        \node at (2.8,2.6){$\gamma_3$};
        \node at (1.4,1.1){{\color{red}$\epsilon_1$}};
        \node at (1.4,2.5){{\color{blue}$\epsilon_5$}};
    \end{tikzpicture}
    \caption{We have {\color{red}$\epsilon_1$} $\in \crs{\Phi(s_1)}$ and {\color{blue} $\epsilon_5$} $\in \crs{\Phi(s_5)}$ and $\gamma_3\in \extremalarg{s_3}$.  The punctures interior to $\epsilon_5$ are of type $\{4,3\}$ and the punctures interior to $\epsilon_1$ are of type $\{2,3\}$}\label{fig:5crossingexample}
\end{figure}
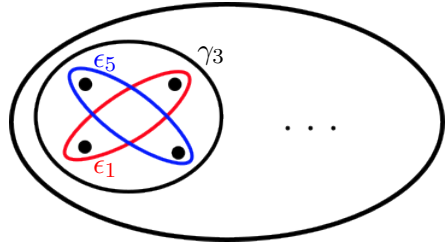
    
Note that since $\epsilon_{i-2}$ intersects $\extremalarg{s_{i-1}}$ we must have $\epsilon_{i-2}$ exterior to $\crs{\Phi(s_{i+1})}$, since $\epsilon_{i-2}$ and $\extremalarg{s_{i-1}}$ are disjoint from $\crs{\Phi(s_{i+1})}$ by Lemma~\ref{lem:crs:commdisj}, and $\extremalarg{s_{i-1}}$ is exterior to $\crs{\Phi(s_{i+1})}$ by definition. Similarly, we have $\epsilon_{i+2}$ exterior to $\crs{\Phi(s_{i-1})}$. Since $n \geq 5$ we have $[s_{i+1}, s_{i-2}] = 1$, so by Lemma~\ref{lem:alpha:exterioragreement} we obtain
\[
  \Phi(s_{i+1})(\epsilon') = \Phi(s_{i-2})(\epsilon') = \epsilon_{i-2}.
\]
Note also that $\Phi(s_{i-1}^{-1})(\epsilon_{i-2})$ intersects $\Phi(s_{i-1}^{-1})(\epsilon_{i+2})$ since we have assumed that $\epsilon_{i-2}$ and $\epsilon_{i+2}$ intersect.   Since $n \geq 5$ we have $[s_{i-1}, s_{i+2}] = 1$, so by Lemma~\ref{lem:alpha:exterioragreement}, we see that
\[
  \Phi(s_{i-1}^{-1})(\epsilon_{i+2}) = \Phi(s_{i+2}^{-1})(\epsilon_{i+2}).
\]
Since $[s_{i+2}, s_i] = 1$, we see that $\Phi(s_{i+2})(\epsilon_{i+2})$ is still interior to $\extremalarg{s_i}$ by Lemma~\ref{lem:alpha:exterioragreement}.  By construction, we therefore have $\Phi(s_{i-1}^{-1}s_{i+1})(\epsilon')=\Phi(s_{i-1}^{-1})(\epsilon_{i-2})$ intersects $\Phi(s_{i-1}^{-1})(\epsilon_{i+2})=\Phi(s_{i+2}^{-1})(\epsilon_{i+2})$, which is a curve interior to $\extremalarg{s_i}$.  This  contradicts our prior observation that $\Phi(s_{i-1}^{-1}s_{i+1})(\epsilon')$ is exterior to $\extremalarg{s_i}$.

It now remains to show $\epsilon_{i-2}$ and $\epsilon_{i+2}$ are distinct and mutually exterior.  By way of contradiction, suppose that $\epsilon_{i-2} \leq \epsilon_{i+2}$.  Since $\epsilon_{i+2}$ intersects $\extremalarg{s_{i+1}}$ by assumption and is disjoint from $\extremalarg{s_{i-1}}$ by Lemma~\ref{lem:crs:commdisj}, every $p \in \Int(\epsilon_{i+2})$ has $\calT_{\Phi}(p) = \{i,i+1\}$.  In particular, $\epsilon_{i+2}$ must be exterior to $\extremalarg{s_{i-1}}$.  But then $\epsilon_{i-2} \leq \epsilon_{i+2}$ implies that $\epsilon_{i-2}$ is also exterior to $\extremalarg{s_{i-1}}$ which contradicts our hypothesis.
\end{proof}

We can now show $\calW(\Phi)$ is a multicurve.

\begin{lemma}\label{lem:curves:mumulticurve}
    Let $n \geq 5$ and $m \geq 3$.  Let $\Phi:\braid{n} \rightarrow \braid{m}$ be an externally periodic and minimally typed homomorphism.  Then $\calW(\Phi)$ is a multicurve.   Furthermore, if $\Delta(\Phi) \neq \emptyset$ then $\calW(\Phi) \neq \emptyset$.
\end{lemma}

\begin{proof}
    Let $\delta, \delta' \in \Delta(\Phi)$ and $\mu_{\delta}, \mu_{\delta'} \in \calW(\Phi)$. We first show that $\mu_{\delta}$ and $\mu_{\delta'}$ are disjoint.  If $\calT_{\Phi}(\delta) = \calT_{\Phi}(\delta')$ then $\mu_{\delta}$ and $\mu_{\delta'}$ are isotopic to boundary components of the same surface and therefore disjoint.  Otherwise, let $s_i, s_j \in \extgen{n}$ such that $s_i \neq s_j$ and $\calT_{\Phi}(\delta) = \{i-1,i\}$, $\calT_{\Phi}(\delta') = \{j-1,j\}$. 
    
    Suppose by way of contradiction that $\mu_{\delta}$ intersects $\mu_{\delta'}$.  Since $n \geq 5$, we may assume without loss of generality that $[s_{i-1},s_j] = 1$ and $s_{i-1} \neq s_j$. Any curves in $\Delta(\Phi)_{i-1,i}$ are interior to $\extremalarg{s_{i-1}}$ and curves in $\Delta(\Phi)_{j-1,j}$ are interior to $\extremalarg{s_j}$.  Since $\Phi$ is minimally typed, we have $\extremalarg{s_{i-1}} \cap \extremalarg{s_j} = \emptyset$.  Since $[s_{i-1},s_j] =1$ the multicurves $\extremalarg{s_{i-1}}$ and $\extremalarg{s_j}$ are disjoint by Lemma~\ref{lem:crs:commdisj}.  In particular, any curve in $\Delta(\Phi)_{i-1,i}$ is disjoint from any curve in $\Delta(\Phi)_{j-1,j}$.  Without loss of generality, we now have two cases to consider.

    \p{Case 1: $\{s_{j-1},s_j\} \cap \{s_{i-1},s_i\} = \emptyset$}  In this case no intersection can occur.  Indeed, by Lemma~\ref{lem:curves:typeexterior} and Lemma~\ref{lem:curves:intersectioninterior}, the curve $\mu_{\delta}$ is the boundary of a neighborhood of curves interior to at least one of $\extremalarg{s_i}$ or $\extremalarg{s_{i-1}}$.  But then no intersections can occur as a consequence of Lemma~\ref{lem:curves:muexterior} since $\mu_{\delta'}$ is exterior to $\extremalarg{s_{i-1}}$ and $\extremalarg{s_{i}}$.

    \p{Case 2: $\{s_{j-1},s_j\} = \{s_i,s_{i+1}\}$}  This is equivalent to saying that $s_j = s_{i+1}$.  If $\mu_{\delta}$ and $\mu_{\delta'}$ intersect, then all intersections are interior to $\extremalarg{s_i}$.  Indeed, by Lemma~\ref{lem:curves:intersectioninterior} $\mu_{\delta}$ is the boundary of a neighborhood of curves interior to at least one of $\extremalarg{s_{i-1}}$ or $ \extremalarg{s_{i}}$ and similarly for $\mu_{\delta'}$.  Suppose by way of contradiction that $\mu_{\delta}$ and $\mu_{\delta'}$ intersect. Now, $\mu_{\delta}$ is exterior to $\extremalarg{s_{i+1}}$ and $\mu_{\delta'}$ is exterior to $\extremalarg{s_{i-1}}$ by Lemma~\ref{lem:curves:muexterior}. This means that if $\mu_{\delta}$ and $\mu_{\delta'}$ intersect, then there is an intersection among:
    \begin{itemize}
        \item a curve in $\crs{\Phi(s_{i-2})}$ with a curve in $\Delta(\Phi)_{i,i+1}$; or
        \item a curve in $\crs{\Phi(s_{i+2})}$ with a curve in $\Delta(\Phi)_{i-1,i}$; or
        \item a curve in $\crs{\Phi(s_{i-2})}$ that intersects $\Delta(\Phi)_{i-1,i}$, with a curve in $\crs{\Phi(s_{i+2})}$ that intersects $\Delta(\Phi)_{i,i+1}$.
    \end{itemize}  The first two cannot occur since the fact that $n \geq 5$ implies that such an intersection would contradict Lemma~\ref{lem:crs:commdisj}. Suppose that $\epsilon_{i-2}\in\crs{\Phi(s_{i-2})}$ and $\epsilon_{i+2}\in\crs{\Phi(s_{i+2})}$ satisfy the conditions of the last case. By Lemmas~\ref{lem:crs:commdisj} and~\ref{lem:curves:intersectioninterior} the curves $\epsilon_{i-2}$ and $\epsilon_{i+2}$ are both interior to, and individually disjoint from, $\extremalarg{s_i}$. Since $\epsilon_{i-2}$ and $\epsilon_{i+2}$ intersect, they must both be simultaneously interior to some $\gamma_i\in\extremalarg{s_i}$. But they also satisfy the hypotheses of Lemma~\ref{lem:curves:munequals5}, whose conclusion states that $\epsilon_{i-2}$ and $\epsilon_{i+2}$ are disjoint: a contradiction.

    It now remains to show that $\calW(\Phi)$ is nonempty when $\Delta(\Phi) \neq \emptyset$.  Let $s_{i} \in \extgen{n}$ and let $\delta \in \Delta(\Phi)_{i-1,i}$.  We show that $\mu_{\delta}$ is essential, i.e., that the definition of $\mu_{\delta}$ given above does not produce a curve homotopic to $\partial \disk{m}$.  Suppose by way of contradiction that $\mu_{\delta}$ is inessential.  Let us recall the alternate definition of $\mu_{\delta}$ discussed earlier.  Let $\calC_{\delta}$ denote the smallest subset of
    \[
      K_{i-1,i} \coloneqq \Delta(\Phi)_{i-1,i} \cup \crs{\Phi(s_{i-2})} \cup \crs{\Phi(s_{i+1})}
    \]
    that satisfies:
    \begin{itemize}
        \item item if $\epsilon \in \calC_{\delta}$ and $\epsilon' \in K_{i-i,i}$ satisfies     $\epsilon \leq \epsilon'$ or $\epsilon' \leq \epsilon$ then $\epsilon' \in \calC_{\delta}$; and
        \item $\delta \in \calC_{\delta}$; and
        \item if $\epsilon \in \calC_{\delta}$ and $\epsilon' \in \calV_{i-1,i}$ such that $\epsilon'$ intersects $\epsilon$ then $\epsilon' \in \calC_{\delta}$.
        \end{itemize}  As a consequence of the definition, $\mu_{\delta}$ is a boundary component of $\disk{m} \cut \calC_{\delta}$.  Now, let $\delta' = \Phi(a_1)(\delta)$.  As a consequence of Lemmas~\ref{lem:alpha:periodictypepreserve} and~\ref{lem:curves:mintype} we see that $\Phi(a_1)(\Delta(\Phi)_{i-1,i})=\Delta(\Phi)_{i,i+1}$, and in particular $\delta' \in \Delta(\Phi)_{i,i+1}$.  If we define $\calC_{\delta'}$ as above except with the set
        \[
          K_{i,i+1} \coloneqq \Delta(\Phi)_{i,i+1} \cup \crs{\Phi(s_{i-1})} \cup \crs{\Phi(s_{i+2})}
        \]
        then any $\epsilon \in \calC_{\delta}$ is disjoint from any $\epsilon' \in \calC_{\delta'}$ via the argument in Case 2 above.  On the other hand, if $\mu_{\delta}$ were homotopic to $\partial \disk{m}$, then $\mu_{\delta'}$ would also be homotopic to $\partial \disk{m}$ since $\Phi(a_1)(\mu_{\delta}) = \mu_{\Phi(a_1)(\delta)} = \mu_{\delta'}$ by applying Lemma~\ref{lem:crs:functorial} and Lemma~\ref{lem:alpha:standardrootaction} along with the definition of $\mu_{\delta}$ and $\mu_{\delta'}$. We have the following claim.

    \p{Claim} For any $\epsilon \in \calC_{\delta}$ and any $p \in \Int(\epsilon)$ we also have $\calT_{\Phi}(p) = \{i-1,i\}$.

    \p{Proof of claim} If $\epsilon\in\calC_{\delta}\cap\Delta(\Phi)_{i-1,i}$, then this follows from Lemma~\ref{lem:curves:mintype}. Suppose now that $\epsilon\in\calC_{\delta}\setminus\Delta(\Phi)_{i-1,i}$. Without loss of generality, we may assume that $\epsilon \in \calC_{\delta} \cap \crs{\Phi(s_{i+1})}$.  By construction, $\epsilon$ either contains, is contained in, or intersects some other curve in $\calC_{\delta}$. Since every curve in $\Delta(\Phi)_{i-1,i}$ is interior to $\extremalarg{s_{i-1}}$ by definition, the same is true for $\epsilon$ by Lemma~\ref{lem:crs:commdisj} since $[s_{i+1}, s_{i-1}] = 1$.  Therefore any $p \in \Int(\epsilon)$ has type $\{i-1,i\}$ or $\{i-2,i-1\}$ by Theorem~\ref{thm:alpha:welltypedpuncture}.  Suppose for a contradiction that some $p \in \Int(\epsilon)$ has $\calT_{\Phi}(p) = \{i-1,i-2\}$. Then $\epsilon$ must be interior to $\extremalarg{s_{i-2}}$ by Lemma~\ref{lem:crs:commdisj} since $[s_{i+1}, s_{i-2}] = 1$. By Lemma~\ref{lem:curves:typeexterior}, $\Delta(\Phi)_{i-1,i}$ is exterior to $\extremalarg{s_{i-2}}$.  In particular, the fact that $\epsilon$ is interior to $\extremalarg{s_{i-2}}$ implies that $\epsilon$ could not contain, be contained in, or intersect a curve in $\Delta(\Phi)_{i-1,i}$, a contradiction.\qed{}
    
    We now proceed with the proof.  Let
    \[
      \Sigma_{\delta} = \disk{m} \cut \calC_{\delta}.
    \]
    Given the claim, we have $\calC_{\delta} \cap \calC_{\delta'} = \emptyset$, since if $\epsilon \in \calC_{\delta} \cap \calC_{\delta'}$ then any $p \in \Int(\epsilon)$ would have both $\calT_\Phi(p)=\{i-1,i\}$ and $\calT_\Phi(p)=\{i,i+1\}$, which is not possible. Therefore every curve of $\calC_{\delta'}$ is supported on $\Sigma_{\delta}$.  We have supposed by way of contradiction that $\mu_{\delta}$ is homotopic to $\partial \disk{m}$.  Therefore the connected component of $\Sigma_{\delta}$ that contains $\mu_{\delta}$ is an annulus.  This component cannot contain any curves in $\calC_{\delta'}$, so there is a different connected component $\Sigma'$ of $\Sigma_{\delta}$ that contains $\calC_{\delta'}$.  The curve $\nu$ homotopic to the outermost boundary component of $\Sigma'$ contains the set of curves $\calC_{\delta'}$.  In particular, $\nu \geq \mu_{\delta'}$ by the definition of $\mu_{\delta'}$  Since $\nu$ is homotopic to a boundary component of $\Sigma' \subsetneq \Sigma_{\delta}$ that is not the annular component containing $\mu_\delta$, we have $\nu$ is essential.  Since $\nu \geq \mu_{\delta'}$, we have $\mu_{\delta'}$ essential, so $\mu_{\delta'}$ is not homotopy equivalent to $\partial \disk{m}$, a contradiction.
\end{proof}  Recall by Lemma~\ref{lem:curves:muexterior} that $\mu_{\delta} \in \calW(\Phi)_{i-1,i}$ is exterior to $\extremalarg{s_j}$ for $s_j \in \extgen{n} \setminus \{s_{i-1}, s_i\}$.  We extend this to the following.  

\begin{lemma}\label{lem:curves:mudisjointness}
    Let $n \geq 5$ and $m \geq 3$.  Let $\Phi:\braid{n} \rightarrow \braid{m}$ be an externally periodic and minimally typed homomorphism. Let $s_i\in\extgen{n}$ and $s_j\in\extgen{n}\setminus\{s_{i-1},s_i\}$. Then $\calW(\Phi)_{i-1,i}$ is exterior to $\crs{\Phi(s_j)}$.
\end{lemma}

\begin{proof}
     Let $\delta \in \Delta(\Phi)_{i-1,i}$ and let $\mu_{\delta} \in \calW(\Phi)_{i-1,i}$. Let $\epsilon_j \in \crs{\Phi(s_j)}$. By the definition of $\mu_{\delta}$, if $s_j \in \{s_{i-2}, s_{i+1}\}$ then $\mu_{\delta} \geq \epsilon_j$. Suppose that $s_j$ commutes with $s_{i-2}$, $s_{i-1}$, $s_i$, and $s_{i+1}$. Then $\epsilon_j$ is disjoint from $\mu_\delta$ by Lemma~\ref{lem:crs:commdisj}. If $\epsilon_j$ is interior to $\extremalarg{s_{j'}}$ for some $s_{j'}\in\extgen{n}\setminus\{s_{i-1},s_i\}$, then $\mu_\delta$ is is exterior to $\extremalarg{s_{j'}}$ by Lemma~\ref{lem:curves:muexterior}, and therefore also to $\epsilon_j$.  If not, then $\epsilon_j$ is interior to one or both of $\extremalarg{s_i}$ and $\extremalarg{s_{i-1}}$.  Since $s_j$ commutes with $s_{i-2}, s_{i-1}, s_{i}, s_{i+1}$ we have $\epsilon_j$ disjoint from $\extremalarg{s_{i-2}},\ldots,\extremalarg{s_{i+1}}$ by Lemma~\ref{lem:crs:commdisj}.  Therefore $\epsilon_j$ is interior to $\Delta(\Phi)_{i-1,i}$ since $\Phi$ is minimally typed. Since $\mu_{\delta}$ is exterior to $\Delta(\Phi)_{i-1,i}$  we conclude that $\mu_{\delta}$ is exterior to $\epsilon_j$.
    
    It therefore suffices to consider the situation where $s_j \in \{s_{i-3}, s_{i+2}\}$.  Without loss of generality, suppose that $s_j = s_{i+2}$. By way of contradiction, suppose $\epsilon_j=\epsilon_{i+2}$ intersects some $\mu_{\delta}$ non-trivially.  We have $\epsilon_{i+2}$ disjoint from $\extremalarg{s_i}$ and $\extremalarg{s_{i-1}}$ by Lemma~\ref{lem:crs:commdisj}.  Now, if $n \geq 6$ we have $[s_{i+2}, s_{i-2}] = 1$.  In particular, $\epsilon_{i+2}$ is disjoint from $\crs{\Phi(s_{i-2})}$ by Lemma~\ref{lem:crs:commdisj}, so we have that $\epsilon_{i+2}$ intersects some $\epsilon_{i+1} \in \crs{\Phi(s_{i+1})}$ with $\epsilon_{i+1}$ intersecting $\extremalarg{s_i}$.  In the case that $n = 5$, we could also have $\epsilon_{i+2}$ intersect some $\epsilon_{i-2} \in \crs{\Phi(s_{i-2})}$.  The latter situation contradicts Lemma~\ref{lem:curves:munequals5}, so it suffices to rule out that $\epsilon_{i+2}$ intersects $\epsilon_{i+1}$. 
    
    By Lemma~\ref{lem:curves:intersectioninterior}, there is some $\gamma_{i-1} \in \extremalarg{s_{i-1}}$ with $\epsilon_{i+1} \leq \gamma_{i-1}$.  In particular, if $\epsilon_{i+2}$ intersects $\epsilon_{i+1}$, then $\epsilon_{i+2}$ is interior to $\gamma_{i-1}$.  Now, if $n \geq 6$, we see that $[s_{i-2}, s_{i+2}] = 1$ so $\epsilon_{i+2}$ is interior to some $\delta \in \Delta(\Phi)_{i-1,i}$ or $\delta \in \Delta(\Phi)_{i-2,i-1}$ by Lemma~\ref{lem:curves:mintype}.  If $n = 5$, we have $\epsilon_{i+2}$ and $\epsilon_{i+1}$ intersecting in $\gamma_{i-1}$.  Now, if $\epsilon_{i+2}$ intersects $\extremalarg{s_{i-2}}$, then we contradict Lemma~\ref{lem:curves:munequals5}.  In particular, we must have $\epsilon_{i+2}$ disjoint from $\extremalarg{s_{i-2}}$, so we still have have $\epsilon_{i+2}$ interior to $\delta \in \Delta(\Phi)_{i-1,i}$ or $\Delta(\Phi)_{i,i+1}$.  In the latter case $\epsilon_{i-1}$ is exterior to $\Delta(\Phi)_{i-1,i}$ by the definition of $\Delta(\Phi)$ while in the former case $\epsilon_{i+2}$ is interior to $\mu_\delta$ by the definition of $\mu_{\delta}$.  In particular, we see that this situation cannot occur.  
 \end{proof}  The next two results describe the action of $\Phi(s_j)$ on $\mu_{\delta} \in \calW(\Phi)_{i-1,i}$ in the case that $s_j \in\extgen{n}\setminus\{s_{i-1},s_i\}$.

\begin{lemma}\label{lem:curves:muadjindexreducible}
Let $n \geq 5$ and $m \geq 3$.  Let $\Phi:\braid{n} \rightarrow \braid{m}$ be an externally periodic and minimally typed homomorphism. Let $s_j \in \{s_{i-2},s_{i+1}\}\subset\extgen{n}$.  Then $\Phi(s_j)\left(\calW(\Phi)_{i-1,i}\right)=\calW(\Phi)_{i-1,i}$.
\end{lemma}

\begin{proof}
  Let $M = \crs{\Phi(s_{i-2})} \cup \crs{\Phi(s_{i+1})}$.  Since $n \geq 5$ we have $[s_{i-2}, s_{i+1}] = 1$, so $M$ is a multicurve by Lemma~\ref{lem:crs:commdisj}.  Let
  \[
    M_{j} = \crs{\Phi(s_{j})} \cap M^{\outmulti}
  \]
  for $j \in \{i-2,i+1\}$.  As a consequence of Lemma~\ref{lem:alpha:exterioragreement} and Lemma~\ref{lem:crs:commdisj} we have
  \[
    \Phi(s_{i-2})(M_j) = \Phi(s_{i+1})(M_j)
  \]
  for $j \in \{i-2,i+1\}$.  Let $\overline{M} = M_{i-2} \cap M_{i+1}$.  Recall the map $\Ext_{\overline{M}}$ in Section~\ref{section:external}.  Let $\overline{f} \in \braid{\overline{M}}$ denote the unique element from Lemma~\ref{lem:centralizer:sheaf} that satisfies
  \[
    \Ext_{M_{i-2}}(\overline{f}) = \Ext_{M_{i-2}}(\Phi(s_{i-2})) \text{ and }\Ext_{M_{i+1}}(\overline{f}) = \Ext_{M_{i+1}}(\Phi(s_{i+1})).
  \]
  Such an element exists by Lemma~\ref{lem:alpha:exterioragreement}.  Note that since $M_{i-2}$ and $M_{i+1}$ might share curves, this element $\overline{f}$ is only defined in $\braid{\overline{M}}$.  Let $f\in\Stab_{\braid{m}}(\overline{M})\leq\braid{m}$ denote a lift of $\overline{f}\in\braid{\overline{M}}$. Let $\delta \in \Delta(\Phi)_{i-1,i}$ and $\mu_{\delta} \in \calW(\Phi)_{i-1,i}$.  By Lemma~\ref{lem:curves:mudisjointness} we have $\mu_{\delta}$ exterior to $\overline{M}$.  In particular, by Lemma~\ref{lem:alpha:exterioragreement} have
  \[
    \Phi(s_{i-2})(\mu_{\delta}) = \Phi(s_{i+1})(\mu_{\delta}).
  \]
  Furthermore, $f(M_{i-2}\cup M_{i+1})=M_{i-2}\cup M_{i+1}$ by construction.  Now, any $\gamma_i \in \extremalarg{s_i}$ is exterior to $\crs{\Phi(s_{i-2})}$ by the definition of $\extremalarg{s_i}$ and Lemma~\ref{lem:crs:commdisj}, so $f(\gamma_i) = \Phi(s_{i-2})(\gamma_i)$ for all $\gamma_i \in \extremalarg{s_i}$ by Lemma~\ref{lem:alpha:exterioragreement}.  By Lemma~\ref{lem:alpha:extremalaction} this implies that $f(\extremalarg{s_i})=\extremalarg{s_i}$.  Similarly $f(\extremalarg{s_{i-1}})=\extremalarg{s_{i-1}}$ since $\extremalarg{s_{i-1}}$ is exterior to $\crs{\Phi(s_{i+1})}$. Thus $f(\Delta(\Phi)_{i-1,i})=\Delta(\Phi)_{i-1,i}$.
    
    All of this implies, by the definition of $\mu_{f(\delta)}$, that $\mu_{f(\delta)}\leq f(\mu_\delta)$. By running the same argument with $f$ replaced by $f^{-1}$ and $\delta$ replaced by $f(\delta)$, we find that $\mu_\delta=\mu_{f^{-1}(f(\delta))}\leq f^{-1}(\mu_{f(\delta)})$, and hence $f(\mu_\delta)\leq\mu_{f(\delta)}$. Therefore $f(\mu_\delta)=\mu_{f(\delta)}$. Therefore $\Phi(s_j)(\mu_\delta)=f(\mu_\delta)\in\calW(\Phi)_{i-1,i}$, and thus $\Phi(s_j)(\calW(\Phi)_{i-1,i}) = \calW(\Phi)_{i-1,i}$.
\end{proof} We leverage the above to prove the following result.

\begin{lemma}\label{lem:curves:mootherindexirreducible}
Let $n \geq 5$ and $m \geq 3$.  Let $\Phi:\braid{n} \rightarrow \braid{m}$ be an externally periodic and minimally typed homomorphism.  Let $s_j \in \extgen{n}\setminus\{s_{i-2},s_{i-1},s_i,s_{i+1}\}$. Then $\Phi(s_j)(\calW(\Phi)_{i-1,i}) = \calW(\Phi)_{i-1,i}$.
\end{lemma}

\begin{proof}
  Let $\delta \in\Delta(\Phi)_{i-1,i}$ and let $\mu_{\delta} \in \calW(\Phi)_{i-1,i}$.  Since $\calW(\Phi)_{i-1,i}$ is a finite set, it suffices to prove that $\Phi(s_j)(\mu_{\delta}) \in \calW(\Phi)_{i-1,i}$.  By Lemma~\ref{lem:curves:muexterior} we have $\mu_{\delta}$  exterior to $\crs{\Phi(s_j)}$.  Furthermore, since $[s_j,s_{i-1}] = [s_j, s_i] = 1$, $\Phi(s_j)$ stabilizes $\Delta(\Phi)_{i-1,i}$ by Lemma~\ref{lem:fixing:mcal1}.  In particular,
  \[
    \Phi(s_j)(\mu_{\delta}) \geq \Phi(s_j)(\delta) \in \Delta(\Phi)_{i-1,i}.
  \]
  Our first goal is to show that $\Phi(s_j)(\mu_{\delta})$ is disjoint from $\calW(\Phi)_{i-1,i}$.  
    
  If $[s_j, s_{i+1}] = [s_j, s_{i-2}] = 1$ then $\Phi(s_j)(\mu_{\delta})$ is also disjoint from both $\crs{\Phi(s_{i-2})}$ and $\crs{\Phi(s_{i+1})}$ by Lemma~\ref{lem:crs:commdisj}.   Suppose now without loss of generality that $[s_j, s_{i+1}] \neq 1$.  Now, observe that
  \[
    \iota(\Phi(s_j)(\mu_{\delta}), \crs{\Phi(s_{i+1})}) = \iota(\mu_{\delta}, \Phi(s_{i+1})(\crs{\Phi(s_j)}))
  \]
  by the braid relation.  Applying $\Phi(s_{i+1}^{-1})$ on both sides and using Lemmas~\ref{lem:crs:commdisj},~\ref{lem:alpha:exterioragreement}, and~\ref{lem:curves:muexterior} we have
  \[
    \iota(\mu_{\delta}, \Phi(s_{i+1})(\crs{\Phi(s_j)})) =\iota(\Phi(s_{i+1}^{-1})(\mu_{\delta}), \crs{\Phi(s_j)}).
  \]
  This last expression equals $0$ since $\Phi(s_{i+1}^{-1})(\mu_{\delta}) \in \calW(\Phi)_{i-1,i}$ by Lemma~\ref{lem:curves:muadjindexreducible} and $\calW(\Phi)_{i-1,i}$ is disjoint from $\crs{\Phi(s_j)}$ by Lemma~\ref{lem:curves:mudisjointness}. Hence $\Phi(s_j)(\mu_{\delta})$ is disjoint from $\crs{\Phi(s_{i+1})}$.  By the first paragraph we have $\Phi(s_j)(\mu_{\delta})$ disjoint from $\Delta(\Phi)_{i-1,i}$.  By applying the above argument if $[s_j,s_{i-2}] \neq 1$ or by applying Lemma~\ref{lem:crs:commdisj} if $[s_j, s_{i-2}] = 1$, we have $\Phi(s_j)(\mu_{\delta})$ disjoint from $\crs{\Phi(s_{i-2})}$.  In particular, $\Phi(s_j)(\mu_{\delta})$ is disjoint from $\calW(\Phi)_{i-1,i}$.  
    
  It now remains to show that $\Phi(s_j)(\mu_{\delta}) = \mu_{\Phi(s_j)(\delta)}$.  Let $\delta' = \Phi(s_j)(\delta)$.  Since $\Phi(s_j)(\mu_{\delta}) \geq \delta'$ and $\delta'\in \Delta(\Phi)_{i-1,i}$ by Lemma~\ref{lem:fixing:mcal1}, we have
  \[
    \Phi(s_j)(\mu_{\delta}) \geq \mu_{\delta'}
  \]
  by the definition of $\calW(\Phi)$.  Now by the argument in the previous two paragraphs with $s_j$ replaced by $s_j^{-1}$ and $\mu_{\delta}$ replaced by $\mu_{\delta'}$, we see that $\Phi(s_j^{-1})(\mu_{\delta'})$ is disjoint from $\calW(\Phi)_{i-1,i}$. As above, we have
  \[
    \Phi(s_j^{-1})(\mu_{\delta'}) \geq \mu_{\Phi(s_j^{-1})(\delta')}.
  \]
  By definition, we have $\Phi(s_j^{-1})(\delta') = \delta$ so we have $\Phi(s_j^{-1}(\mu_{\delta'})) \geq \mu_{\delta}$.  Applying $\Phi(s_j)$ to both sides yields $\mu_{\delta'} \geq \Phi(s_j)(\mu_{\delta})$. Therefore $\Phi(s_j)(\mu_{\delta}) = \mu_{\delta'} \in \calW(\Phi)_{i-1,i}$, as desired.
\end{proof} We show in Lemma~\ref{lem:curves:mureducible} that $\Phi(s_i)$ and $\Phi(s_{i-1})$ send $\calW(\Phi)_{i-1,i}$ into $\calW(\Phi)$.  We need two auxiliary lemmas.

\begin{lemma}\label{lem:curves:kappamintype}
Let $n \geq 5$ and $m \geq 3$.  Let $\Phi:\braid{n} \rightarrow \braid{m}$ be an externally periodic and minimally typed homomorphism.  Let $\kappa \in \Aut(\braid{n})$ be defined by $\kappa(s_i) = s_i^{-1}$ for all $s_i \in \extgen{n} \setminus \{s_n\}$.  Then the homomorphism $\Phi \circ \kappa$ is also minimally typed.
\end{lemma}  

\begin{proof}
    Suppose by way of contradiction that $\Phi \circ \kappa$ is not minimally typed.  Then there are $s_i, s_j \in \extgen{n}$ with $s_i\neq s_j$ and $[s_i,s_j]=1$ such that $\calE_{\Phi \circ \kappa}^{s_i} \cap \calE_{\Phi \circ \kappa}^{s_j} \neq \emptyset$.  Lemma~\ref{lem:alpha:extremalcollision} implies that any $\gamma \in \calE_{\Phi \circ \kappa}^{s_i} \cap \calE_{\Phi \circ \kappa}^{s_j}$ satisfies $\gamma \in\calE_{\Phi \circ \kappa}^{s_k}$ for all $s_k \in \extgen{n}$.  Therefore, since $\Phi$ is minimally typed, we see that $\gamma \not \in \extremal$. Therefore, there is some $\gamma' \in \extremal$ with $\gamma < \gamma'$.  If $\gamma' \in \extremalarg{s_\ell}$ for $s_\ell \in \extgen{n} \setminus \{s_n\}$ then $\gamma' \in \crs{(\Phi \circ\kappa)(s_\ell)}$ by Lemma~\ref{lem:crs:power} so $\gamma \not \in \calE_{\Phi \circ \kappa}$, a contradiction.  Therefore, $\gamma' \in \extremalarg{s_n}$. Now $\gamma \in \crs{\Phi(s_1)}$ by Lemma~\ref{lem:crs:power}, since above we showed that $\gamma \in \calE_{\Phi\circ\kappa}^{s_1}$.  In particular, since $n \geq 5$ we have $[s_1,s_{n-1}] = 1$, so $\gamma$ is disjoint from $\extremalarg{s_1}$ and $\extremalarg{s_{n-1}}$ by Lemma~\ref{lem:crs:commdisj}.  But now $\Phi$ is minimally typed, so by Theorem~\ref{thm:alpha:welltypedpuncture} any $p \in \Int(\gamma')$
    has $\calT_{\Phi}(p) = \{n-1,n\}$ or $\{1,n\}$.  In particular, $\gamma \leq \gamma''$ for some $\gamma'' \in \extremalarg{s_1} \cup \extremalarg{s_{n-1}}$.  But then
    \[
      \gamma'' \in \crs{(\Phi \circ \kappa)(s_1)} \cup \crs{(\Phi \circ \kappa)(s_{n-1})}
    \]
    by Lemma~\ref{lem:crs:power}.  Therefore $\gamma$ is not $(\Phi \circ \kappa)$--maximal,  a contradiction.
 \end{proof}

We now continue with another auxiliary result we use to prove Lemma~\ref{lem:curves:mureducible}.

\begin{lemma}\label{lem:curves:outerauto}
Let $n \geq 5$ and $m \geq 3$.  Let $\Phi:\braid{n} \rightarrow \braid{m}$ be an externally periodic and minimally typed homomorphism.  Let $\kappa \in \Aut(\braid{n})$ be defined by $\kappa(s_i) = s_i^{-1}$ for all $s_i \in \extgen{n} \setminus \{s_n\}$.  Then $\extremalarg{s_i} = \calE_{\Phi \circ \kappa}^{s_i}$ for all $s_i \in \extgen{n} \setminus \{s_n\}$
\end{lemma}

\begin{proof} Since $\kappa^2 = 1$ it suffices to show that $\extremalarg{s_i} \subseteq \calE_{\Phi \circ \kappa}^{s_i}$.
  We have
  \[
    \crs{\Phi(s_i)} = \crs{(\Phi \circ \kappa)(s_i)}
  \]
  for $s_i \in \extgen{n} \setminus \{s_n\}$ by Lemma~\ref{lem:crs:power}.  Suppose by way of contradiction that $\gamma \in \extremalarg{s_i}$ satisfies $\gamma \not \in \calE_{\Phi \circ \kappa}^{s_i}$.  By definition we have $\widehat{\gamma} \in \calE_{\Phi \circ \kappa}$ with $\widehat{\gamma} > \gamma$.  Now by Lemma~\ref{lem:crs:power} we see that if $\widehat{\gamma} \notin \calE_{\Phi \circ \kappa}^{s_n}$ then $\gamma \not \in \extremal$.  In particular, we conclude that $\widehat{\gamma} \in \calE_{\Phi \circ \kappa}^{s_n}$.  By Lemma~\ref{lem:curves:kappamintype} we have $\Phi \circ \kappa$ minimally typed.  In particular, by Theorem~\ref{thm:alpha:welltypedpuncture} every $p \in \Int(\widehat{\gamma})$ has $\calT_{\Phi \circ \kappa}(p) = \{1,n\}$ or $\{n-1,n\}$.  In particular, if $\gamma$ is disjoint from $\extremalarg{s_{n-1}}$ and $\extremalarg{s_1}$, then $\gamma$ is interior to $\extremalarg{s_{n-1}} \cup \extremalarg{s_1}$, so $\gamma \not \in \extremal$.  Note that this always occurs if $[s_i, s_{n-1}] = [s_1, s_i] = 1$.

    Therefore, $\gamma$ intersects $\extremalarg{s_{n-1}} \cup \extremalarg{s_1}$, and so by Lemma~\ref{lem:crs:commdisj} we have $s_i\in\{s_2,s_{n-2}\}$.  We prove the following claim.
    
    \p{Claim} Let $\gamma' \in \extremalarg{s_i}$.  Then $\gamma'$ intersects $\extremalarg{s_{i+1}}$ and $\extremalarg{s_{i-1}}$.

    \p{Proof of claim} Without loss of generality, suppose by way of contradiction that $\gamma'$ is disjoint from $\extremalarg{s_{i+1}}$.  Since $n \geq 5$, we have $[s_{i-2},s_i] = [s_{i-2}, s_{i+1}] = 1$.  By Lemmas~\ref{lem:alpha:exterioragreement} and~\ref{lem:alpha:extremalaction} we have $\Phi(s_{i+1})(\gamma') \in \extremalarg{s_i}$.  By Lemma~\ref{lem:alpha:exterioragreement} again we have $\Phi(s_is_{i+1})(\gamma') \in \extremalarg{s_i}$.  On the other hand, we have $\Phi(s_is_{i+1})(\gamma') \in \extremalarg{s_{i+1}}$ by Lemma~\ref{lem:alpha:extremalbraid}.  We conclude that $\extremalarg{s_i} \cap \extremalarg{s_{i+1}} \neq \emptyset$. By Lemma~\ref{lem:alpha:extremalcollision} we see that $\extremalarg{s_i} \cap \extremalarg{s_{i+1}} \subseteq \extremalarg{s_j}$ for all $s_j \in \extgen{n}$, so $\Phi$ is not minimally typed, a contradiction. Note that in order to apply Lemma~\ref{lem:alpha:extremalbraid} in the case that $\gamma'$ is disjoint from $\extremalarg{s_{i-1}}$ we must use inverses of standard generators and invoke Lemma~\ref{lem:crs:power}.\qed{}

    Given the claim, we now proceed with the proof as follows.  Suppose first that $n \geq 6$.  By the claim, $\gamma$ intersects some $\gamma'' \in \extremalarg{s_j}$ for some $s_j \in \extgen{n}$ with $[s_j, s_{n-1}] = [s_j, s_n] = [s_j, s_{1}] = 1$.  Then by Lemma~\ref{lem:crs:commdisj} we are in the situation in the first paragraph, so we arrive at a contradiction.

    We now handle the case $n = 5$.  Without loss of generality, we assume that $s_i=s_2$, since the case $s_i=s_3$ is the same argument with $s_4$ exchanged for $s_1$ and $s_2$ exchanged for $s_3$.  Since $s_i=s_2$, we have $\gamma$ intersects $\extremalarg{s_1}$.  That is, there is some $\gamma_1 \in \extremalarg{s_1}$ such that $\gamma_1$ intersects $\gamma$.  Now, if $\gamma_1 \not \in \calE_{\Phi \circ \kappa}$, then since also $[s_1,s_1] = [s_1, s_{n-1}] = 1$ we are again in the situation described by the first paragraph which we have already resolved.  Therefore $\gamma_1 \in \calE_{\Phi \circ \kappa}^{s_1}$.  Similarly, applying the claim twice yields some $\gamma_3 \in \extremalarg{s_3}$ that intersects $\gamma$ and some $\gamma_4 \in \extremalarg{s_4}$ that intersects $\gamma_3$, such that $\gamma_4 \in \calE_{\Phi \circ \kappa}^{s_4}$.  But then $\gamma_3$ and $\gamma$ intersect interior to $\widehat{\gamma} \in \calE_{\Phi \circ \kappa}^{s_5}$, which contradicts Lemma~\ref{lem:curves:munequals5}.
\end{proof}  We now construct reducing systems for minimally typed homomorphisms.

\begin{lemma}\label{lem:curves:mureducible}
Let $n \geq 5$ and $m \geq 3$.  Let $\Phi:\braid{n} \rightarrow \braid{m}$ be an externally periodic and minimally typed homomorphism.  The set $\calW(\Phi)$ is preserved by $\Phi(\braid{n})$.
\end{lemma}

\begin{proof}
 By Lemma~\ref{lem:curves:muadjindexreducible} and Lemma~\ref{lem:curves:mootherindexirreducible} we have $\Phi(s_j)\left(\calW(\Phi)_{i-1,i}\right) = \calW(\Phi)_{i-1,i}$ for $s_j \in \extgen{n}$ with $s_j\in\extgen{n}\setminus\{s_{i-1},s_i\}$.  It therefore suffices to show that $\Phi(s_i)(\calW(\Phi)_{i-1,i}) \subseteq  \calW(\Phi)$ and $\Phi(s_{i-1})(\calW(\Phi)_{i-1,i}) \subseteq  \calW(\Phi)$.

 We have $a_1 = s_{i}s_{i+1} \ldots s_{i+n-2}$. We have $\Phi(a_1)(\extremalarg{s_j}) = \extremalarg{s_{j+1}}$ by Lemma~\ref{lem:alpha:periodictypepreserve} and
 \[
   \Phi(a_1)(\crs{\Phi(s_j)}) = \crs{\Phi(s_{j+1})}
 \]
 by Lemma~\ref{lem:crs:functorial}.  Therefore, as a consequence of the definition of $\calW(\Phi)$ and the fact that $\calW(\Phi)$ is finite, we have $\Phi(a_1)(\calW(\Phi)_{i-1,i}) = \calW(\Phi)_{i+1,i}$.  Combining this with Lemma~\ref{lem:curves:muadjindexreducible} and Lemma~\ref{lem:curves:mootherindexirreducible} we conclude that $\Phi(s_i)(\calW(\Phi)_{i-1,i}) =\calW(\Phi)_{i,i+1} \subseteq \calW(\Phi)$ as desired.
    
 It now remains to show that if $\mu_{\delta} \in \calW(\Phi)_{i,i+1}$ then $\Phi(s_i)(\mu_{\delta}) \in \calW(\Phi)_{i-1,i}$.  Without loss of generality we prove this for $i = 3$.  Let $\kappa:\braid{n} \rightarrow \braid{n}$ denote the inversion automorphism given by $\kappa(s_i) = s_i^{-1}$ for all $s_i \in \extgen{n} \setminus \{s_n\}$.  Lemma~\ref{lem:curves:outerauto} says that $\extremalarg{s_j} = \calE_{\Phi \circ \kappa}^{s_j}$ for $s_j \in \extgen{n} \setminus \{ s_n\}$, so
 \[
   \Delta(\Phi)_{j-1,j} = \Delta(\Phi \circ \kappa)_{j-1,j}
 \]
 for $2 \leq j \leq n-1$.  We also have $\crs{\Phi(s_j)} = \crs{\Phi(\kappa(s_j))}$ for $s_j \in \extgen{n} \setminus \{s_n\}$ by Lemma~\ref{lem:crs:power}.  We therefore have have
 \[
   \calW(\Phi \circ \kappa)_{j-1,j} = \calW(\Phi)_{j-1,j}
 \]
 for all $3 \leq j \leq n-2$.  In particular, we have $\Phi(\kappa(a_1^{-1}))(\mu_{\delta}) \in \calW(\Phi)_{2,3}$ for $\mu_{\delta} \in \calW(\Phi)_{3,4}$.  Now, we may write
 \[
   \kappa(a_1^{-1}) = s_{n-1} \cdot \cdots \cdot s_1.
 \]
 Let $\mu_{\delta} \in \calW(\Phi)_{3,4}$.  By Lemma~\ref{lem:curves:muadjindexreducible} and Lemma~\ref{lem:curves:mootherindexirreducible} we have
 \[
   \Phi(s_2s_1)\left(\calW(\Phi)_{3,4}\right) \subseteq \calW(\Phi)_{3,4}.
 \]
 Since $\calW(\Phi)_{3,4}$ is finite, we have $\Phi(s_2s_1)(\calW(\Phi)_{3,4}) = \calW(\Phi)_{3,4}$.  In particular, this implies that $\Phi(s_1^{-1}s_2^{-1})(\mu_{\delta}) \in \calW(\Phi)_{3,4}$.  Let
 \[
   \mu_{\delta'} = \Phi(s_1^{-1} s_2^{-1})(\mu_{\delta})
 \]
 so we have
 \[
   \Phi(s_2s_1)(\mu_{\delta'}) = \mu_{\delta}.
 \]
 In particular, we have
 \[
   \Phi(s_{n-1}s_{n-2} \cdot \cdots \cdot s_3)(\mu_{\delta}) = \mu_{\delta''} \in \calW(\Phi)_{2,3}
 \]
 by our observation about $\kappa(a_1^{-1})$.  Then by Lemma~\ref{lem:curves:mootherindexirreducible} the elements $\Phi(s_{j})$ preserve the set $\calW(\Phi)_{2,3}$ for $4\leq j\leq n-1$.  In particular, we see that
 \[
   \Phi(s_3)(\mu_{\delta}) = \Phi(s_4^{-1} \cdot \cdots \cdot s_{n-1}^{-1})(\mu_{\delta''}) \in \calW(\Phi)_{2,3}.
 \]
 In particular we have $\Phi(s_3)(\mu_\delta) \in \calW(\Phi)_{2,3}$.  The result now follows.
\end{proof}

We are now ready to show $\Delta(\Phi) = \emptyset$ for minimally typed irreducible homomorphisms.

\begin{proposition}\label{prop:curves:valid}
   Let $n \geq 5$ and $m \geq 3$. Let $\Phi:\braid{n}\to \braid{m}$ be an externally periodic and minimally typed homomorphism. If $\Phi$ is irreducible then $\Delta(\Phi) = \emptyset$.
\end{proposition}
\begin{proof}
    Suppose that $\Delta(\Phi) \neq \emptyset$.  Then $\calW(\Phi)$ is a nonempty multicurve by Lemma~\ref{lem:curves:mumulticurve}, so $\Phi$ is reducible by Lemma~\ref{lem:curves:mureducible}.  
\end{proof}  An important consequence of Proposition~\ref{prop:curves:valid} is that $\Phi$-maximal curves contain no curves in the canonical reduction system of other generators. 

\begin{lemma}\label{lem:curves:nonesting}
    Let $n \geq 5$ and $m \geq 3$.  Let $\Phi:\braid{n} \rightarrow \braid{m}$ be an externally periodic and minimally typed homomorphism with $\Delta(\Phi)=\emptyset$. Let $s_i,s_j \in \extgen{n}$.  Let $\gamma \in \extremalarg{s_i}$ and let $\delta \in \crs{\Phi(s_j)}$ such that $\delta\leq\gamma$.  Then $s_i = s_j$. In particular, $\extremalarg{s_i}=\crs{\Phi(s_i)}^{\outmulti}$.
\end{lemma}
\begin{proof}
    Seeking a contradiction suppose that $s_i \neq s_j$. Since $s_i \neq s_j$ and $n \geq 5$ element $s_j$ commutes either with $s_{i-1}$ or $s_{i+1}$.  Observe that if $[s_{j},s_{i-1}] = [s_{j}, s_{i+1}] = 1$, then $\delta$ is interior to one of $\extremalarg{s_{i-1}}$ or $\extremalarg{s_{i+1}}$ since any $p \in \Int(\delta)$ has $\calT_{\Phi}(p) = \{i,i+1\}$ or $\{i-1,i\}$ by Theorem~\ref{thm:alpha:welltypedpuncture}.  Without loss of generality, suppose that $[s_j, s_{i-1}] \neq 1$, so $[s_j, s_{i+1}] = 1$ and therefore $s_j = s_{i-2}$.  
    
    Observe that if $\delta$ is interior to some $\gamma' \in \extremalarg{s_{i+1}}$ then $\Delta(\gamma, \gamma') \neq \emptyset$, a contradiction.  By Lemma~\ref{lem:crs:commdisj} we therefore have $\delta$ exterior to $\extremalarg{s_{i+1}}$.  Since $\delta \leq \gamma$, Theorem~\ref{thm:alpha:welltypedpuncture} and the assumption that $\Phi$ is minimally typed implies that any $p \in \Int(\delta)$ has $\calT_{\Phi}(p) = \{i-1,i\}$.  However, by Lemma~\ref{lem:crs:commdisj} we have $\Phi(s_i)(\delta) \in \crs{\Phi(s_j)}$.  By Lemma~\ref{lem:alpha:extremalaction} we have $\Phi(s_i)(\delta)$ interior to $\extremalarg{s_i}$.  Then by Theorem~\ref{thm:alpha:welltypedpuncture} we have $\calT_{\Phi}(\Phi(s_i)(p)) = \{i,i+1\}$.  But then $\Phi(s_i)(\delta)$ is interior to $\extremalarg{s_{i+1}}$, so as discussed above this is a contradiction.
\end{proof}  We also obtain the following property about punctures. 

\begin{lemma}\label{lem:curves:ext_puncs}
    Let $n\geq 5$ and $m\geq3$. Let $\Phi:\braid{n} \rightarrow \braid{m}$ be an externally periodic and minimally typed homomorphism.  Assume that $\Delta(\Phi)=\emptyset$. If $p$ is a puncture of type $\calT_{\Phi}(p) = \{i, i+1\}$ and $s_j\in \extgen{n}$ is such that $j\not \in \calT_{\Phi}(p)$, then $p$ is exterior to $\crs{\Phi(s_j)}$. 
\end{lemma}
\begin{proof}
By definition $p$ is interior to a curve $\gamma_i \in \extremalarg{s_i}$ and to a curve $\gamma_{i+1}\in \extremalarg{s_{i+1}}$. Looking for a contradiction, assume $p$ is interior to a curve $\delta_j \in \crs{\Phi(s_j)}$.  Since $j\not \in \calT_{\Phi}(p)$ the element $s_j$ commutes with either $s_i$ or $s_{i+1}$. Without loss of generality, assume $s_j$ commutes with $s_i$. It follows from Lemma~\ref{lem:crs:commdisj} and the definition of $\extremal$ that $\delta_j\leq \gamma_i$, but this contradicts Lemma~\ref{lem:curves:nonesting}. 
\end{proof}  We also have the following useful observation.

\begin{lemma}\label{lem:curves:trivialnestingthing}
Let $n \geq 5$ and $m \geq 3$.  Let $\Phi:\braid{n} \rightarrow \braid{m}$ be an externally periodic and minimally typed homomorphism.  Assume that $\Delta(\Phi) = \emptyset$.  Then any curve $\delta \subset \disk{m}$ is interior to at most one $\gamma \in \extremal$.
\end{lemma}

\begin{proof}
    Suppose by way of contradiction that we have $\delta$ interior to two distinct $\gamma, \gamma' \in \extremal$.  This implies $\Delta(\gamma, \gamma') \neq \emptyset$.  By definition $\Delta(\Phi) \neq \emptyset$ which contradicts our assumption.
\end{proof}

\section{Homomorphisms with unique maximal curves}\label{section:onemaximal}

Our goal in this section is to prove a technical result that we leverage in the proof of Theorem~\ref{thm:extcent}.  Let $\Phi:\braid{n} \rightarrow \braid{m}$ be a homomorphism.  We say that an element $f \in \braid{m}$ is \defn{externally central} if $\Ext(f)\in\langle T_{\partial\disk{\crs{f}^{\outmulti}}}\rangle$, or equivalently $\Ext(f)=\Ext_{\crs{f}^{\outmulti}}(z)$ for some $ z\in Z(\braid{m})$. We say that a homomorphism $\Phi:\braid{n} \rightarrow \braid{m}$ is \defn{externally central} if some (hence all) $s_i \in \extgen{n}$ is (are) externally central.  We say that $\Phi$ is \defn{simple} if the following conditions hold:
\begin{itemize}
    \item $\Phi$ is externally central;
    \item $\Phi$ is minimally typed;
    \item $\Delta(\Phi) = \emptyset$; and
    \item for any $s_i \in \extgen{n}$, the set $\extremalarg{s_i}$ is a singleton.
\end{itemize}

Let us observe how these conditions relate to the assumptions of Theorem~\ref{mainthm:natleast5}. For $n \geq 5$, any non-cyclic irreducible homomorphism is minimally typed by Lemma~\ref{lem:external:anosov}. Likewise, we can see from Proposition~\ref{prop:curves:valid} that if $\Phi$ is non-cyclic and irreducible then $\Delta(\Phi)$ is empty.  The first and last conditions are not \emph{a priori} satisfied by irreducible homomorphisms.

Our goal in this section is to prove Lemma~\ref{lem:onemaximal:twopunc}, which is a technical result that we use in Section~\ref{section:extcent}. Let $p \in \disk{m}$ be a puncture and let $\gamma, \gamma' \subset \disk{m}$ be two curves. We  say that $p$ is \defn{bounded} by $\gamma$ and $\gamma'$ if for all arcs $\alpha$ with one endpoint $p$ and the other endpoint on $\partial \disk{m}$, $\alpha$ intersects one of $\gamma$ or $\gamma'$.  Note that $p$ does \emph{not} have to lie in $\Int(\gamma) \cup \Int(\gamma')$.  See Figure~\ref{fig:punctureboundingex} for an example.

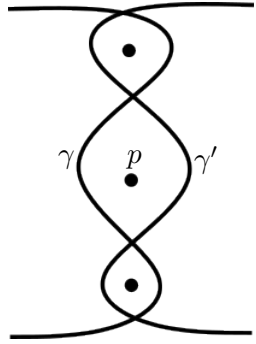
\begin{figure}[ht]
    \centering
    \begin{tikzpicture}
        \node at (0,0)[anchor = south west]{\includegraphics[scale=0.6]{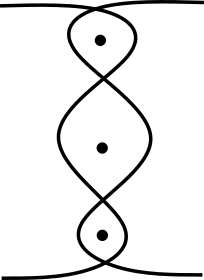}};
        \node at (1.8,2.5){$p$};
        \node at (0.9,2.5){$\gamma$};
        \node at (2.75,2.5){$\gamma'$};
    \end{tikzpicture}
    \caption{The curves $\gamma$ and $\gamma'$ bound the puncture $p$}\label{fig:punctureboundingex}
\end{figure}We now investigate intersections of arcs with $\Phi$-maximal curves.

\begin{lemma}\label{lem:onemaximal:arcint}
Let $n \geq 5$ and $m \geq 3$.  Let $\Phi:\braid{n} \rightarrow \braid{m}$ be a simple homomorphism.  Let $s_i \in \extgen{n}$ and $\gamma_i \in \extremalarg{s_i}$.  Let $\alpha$ be an arc in $\disk{m}$ with distinct endpoints $p,p' \in \Int(\gamma_i)$ that is disjoint from $\gamma_i$.  Then $\alpha$ intersects the unique $\gamma_{i-1} \in \extremalarg{s_{i-1}}$.
\end{lemma}

\begin{proof}
    Note first that since $\Delta(\Phi)$ is empty by the definition of simple, the arc $\alpha$ cannot be interior to $\gamma_{i-1} \in \extremalarg{s_{i-1}}$, since otherwise $\alpha$ would be interior to a curve $\delta$ with $\delta \leq \gamma_i$ and $\delta \leq \gamma_{i-1}$.  By Lemma~\ref{lem:curves:trivialnestingthing} this contradicts the hypothesis that $\Delta(\Phi) = \emptyset$.
    
    Suppose by way of contradiction that $\alpha$ is exterior to $\gamma_{i-1}$.  Since $\Phi$ is externally central by the definition of simple, we have $\Phi(s_{i-1})(\alpha) = \alpha$.  Therefore $\Phi(s_is_{i-1})(\alpha)\leq\Phi(s_i)(\gamma_i)=\gamma_{i}$.  By the braid relation and Lemma~\ref{lem:crs:functorial} we have $\Phi(s_is_{i-1})(\gamma_i)  \in \crs{\Phi(s_{i-1})}$.  Since $\extremalarg{s_{i-1}} = \{\gamma_{i-1}\}$ by the definition of simple and $\extremalarg{s_{i-1}} = \crs{\Phi(s_{i-1})}^{\outmulti}$ by Lemma~\ref{lem:curves:nonesting}, we see that
    \[
      \Phi(s_is_{i-1})(\gamma_i) \leq \gamma_{i-1}.
    \]
    But $\Phi(a_1)(\gamma_{i-1}) = \gamma_i$ by Lemma~\ref{lem:alpha:standardrootaction}.  We conclude that
    \[
      \left|\Int(\gamma_{i-1})\right| = \left|\Int(\gamma_i)\right|
    \]
    and therefore
    \[
      \Phi(s_is_{i-1})(\gamma_i) = \gamma_{i-1}.
    \]
    We conclude that $\Phi(s_is_{i-1})(\alpha)$ is disjoint from and has endpoints interior to both $\gamma_i$ and $\gamma_{i-1}$.  But we established above that no such arc exists, a contradiction. 
\end{proof}  We now construct punctures bounded by $\Phi$-maximal curves.  We begin by fixing some notation.  For any $s_j \in\extgen{n}$ and any simple homomorphism $\Phi:\braid{n} \rightarrow \braid{m}$ we let $\gamma_j$ denote the unique curve in $\extremalarg{s_j}$.

 \begin{lemma}\label{lem:onemaximal:twopunc}
Let $n \geq 5$ and $m \geq 3$.  Let $\Phi:\braid{n} \rightarrow \braid{m}$ be a simple homomorphism.  Suppose that $q,q'$ are distinct punctures in $\disk{m}$ with $\calT_{\Phi}(q) = \calT_{\Phi}(q')$ and $\calT_{\Phi}(q) \neq \emptyset$.  Then there are $s_i, s_{i+1} \in \extgen{n}$ and $\gamma_i ,\gamma_{i+1} \in \extremal$ such that $\gamma_i$ and $\gamma_{i+1}$ bound a puncture $p$ with $\calT_{\Phi}(p) \cap \{i,i+1\} = \emptyset$.
 \end{lemma}

 \begin{proof} For all $s_j \in \extgen{n}$, let $\extremalarg{s_j} = \{\gamma_j\}$.  Since $\Phi$ is minimally typed Theorem~\ref{thm:alpha:welltypedpuncture} implies that $\calT_{\Phi}(q) = \calT_{\Phi}(q') = \{j,j+1\}$ for some $s_j \in \extgen{n}$.  Let $\alpha$ be an arc interior to $\gamma_{j}$ with endpoints $q$ and $q'$.  By Lemma~\ref{lem:onemaximal:arcint}, $\alpha$ intersects $\gamma_{j-1}$.  Choose minimally intersecting representatives $a$ for $\alpha$ and $c_j$ and $c_{j-1}$ and $c_{j-2}$ for $\gamma_j$ and $\gamma_{j-1}$ and $\gamma_{j-2}$ respectively.  Let $a_{j-1} \subseteq c_{j-1}$ be a subarc such that:
 \begin{itemize}
     \item $a_{j-1}$ intersects $a$ exactly once;
     \item the endpoints of $a_{j-1}$ lie on $c_j$; and
     \item the interior of $a_{j-1}$ is disjoint from $c_j$.
 \end{itemize}that intersects $a$ exactly once and whose endpoints lie on $c_j$.  Note that since $\Delta(\Phi) = \emptyset$ by the definition of simple, the fact that $q,q' \in \Int(\gamma_j) \cap \Int(\gamma_{j+1})$ implies that $\iota(\gamma_j, \gamma_{j+1})\geq 4$.  Indeed, otherwise $\iota(\gamma_j, \gamma_{j+1}) = 2$, which implies that there is an arc $\alpha$ with endpoints $q$ and $q'$ that is disjoint from $\gamma_j$ and $\gamma_{j+1}$.  But the existence of such an arc implies that there is a curve $\delta$ with  $\delta \leq \gamma_j$ and $\delta \leq \gamma_{j+1}$, which by Lemma~\ref{lem:curves:trivialnestingthing} contradicts our assumption that $\Delta(\Phi) = \emptyset$.   Similarly, $\iota(\gamma_j, \gamma_{j-1})\geq 4$   by conjugating with $a_1^{-1}$ and applying Lemma~\ref{lem:alpha:standardrootaction}. In particular,  $|c_j\cap c_{j-1}| \geq 4$.  Now, let $a_{j-1}^+$ be a subarc of $c_{j-1}$ such that:
     \begin{itemize}
         \item $a_{j-1} \subseteq a_{j-1}^+$;
         \item the endpoints of $a_{j-1}^+$ lie on $c_j$;
         \item $a_{j-1}^+$ intersects $c_j$ exactly four times, including its endpoints; and
         \item the endpoints of $a_{j-1}$ lie in the interior of $a_{j-1}^+$.
     \end{itemize}  See Figure~\ref{fig:onemaximal:arcintpicsimple}.  
     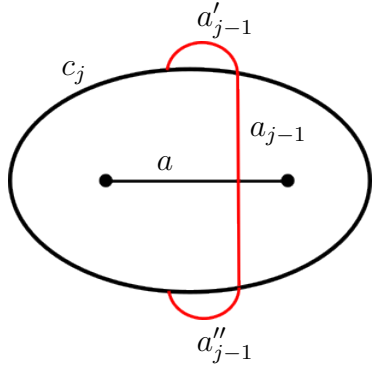
\begin{figure}[ht]
         \begin{tikzpicture}
             \node[anchor = south west] at (0,0) {\includegraphics[scale=0.6]{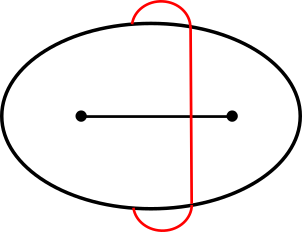}};
             \node at (1,3.4) {\large $c_j$};
             \node at (2.2,2.2) {\large $a$};
             \node at (3.7,2.6) {\large $a_{j-1}$};
             \node at (3,4.1) {\large $a_{j-1}'$};
             \node at (3,-0.2){\large $a_{j-1}''$};
         \end{tikzpicture}
         \caption{the arc $a_{j-1}^+$ and the curve $c_j$}\label{fig:onemaximal:arcintpicsimple}
     \end{figure}
     Now, let $a_{j-1}'$ and $a_{j-1}''$ be the two maximal subarcs of $a_{j-1}^+$ whose interior lies entirely outside of $c_j$.  Let $a_{j}'$ and $a_{j}''$ respectively be disjoint subarcs of $c_j$ such that $a_{j}'$ and $a_{j-1}'$ have the same endpoints, and similarly for $a_j''$ and $a_{j-1}''$.  Since $c_j$ and $c_{j-1}$ were chosen to intersect minimally, we see that $a_{j-1}'$ and $a_{j}'$ cannot bound a bigon, and similarly for $a_{j-1}''$ and $a_j''$.  Let $\delta'$ denote the isotopy class of embedded circles of $a_{j-1}' \cup a_j'$ and $\delta''$ denote the isotopy class of $a_{j-1}'' \cup a_j''$.  Note that $\delta'$ and $\delta''$ are have disjoint representatives and are possibly inessential by virtue of being homotopy equivalent to a puncture.  We have three cases to consider.

     \p{Case 1: there is $r' \in \Int(\delta')$ with $\calT_{\Phi}(r') \cap \{j-1,j\} = \emptyset$ or $r'' \in \Int(\delta'')$ with $\calT_{\Phi}(r'') \cap \{j-1,j\} = \emptyset$} In this case, we set $p$ equal to the appropriate $r'$ or $r''$.  Then $\gamma_j$ and $\gamma_{j-1}$ bound a puncture $p$ with $\calT_{\Phi}(p) \cap \{j-1,j\} = \emptyset$, so the lemma holds.
     
     \p{Case 2: $\delta'$ and $\delta''$ are non-nested} In this case, there are distinct $r' \in \Int(\delta')$ and $r'' \in \Int(\delta'')$.  Assuming we are not in Case 1, we see that $\calT_{\Phi}(r') = \{j-1,j-2\}$ and $\calT_{\Phi}(r'') = \{j-1,j-2\}$ by Theorem~\ref{thm:alpha:welltypedpuncture}, since $\Phi$ is minimally typed and $r'$ and $r''$ are exterior to $\gamma_{j}$ by construction.  Since $\Phi$ is simple, we have $r', r'' \in \Int(\gamma_{j-2})$.  Since $n \geq 5$ we have $[s_j,s_{j-2}] = 1$ so by Lemma~\ref{lem:crs:commdisj} the curves $\gamma_{j-2}$ and $\gamma_j$ are disjoint.  In particular, we see that $\gamma_{j-2}$ and $\gamma_{j-1}$ bound at least one endpoint of $\alpha$ which by construction have type $\{j,j+1\}$.  Setting $i = j-2$ yields the desired result. See Figure~\ref{fig:onemaximal:arcintpiccomplicated} for an example.

     \begin{figure}[ht]
        \begin{tikzpicture}
            \node[anchor = south west] at (0,0) {\includegraphics[scale=0.6]{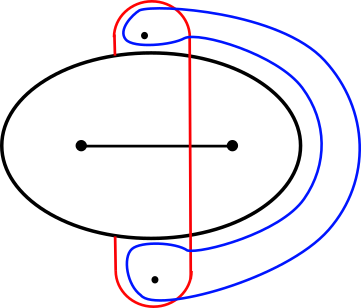}};
             \node at (0.75,3.95) {\large $c_j$};
             \node at (2.3,2.5) {\large $a$};
             \node at (4.1, 2.5) {\large $q$};
             \node at (6.4, 2.3) {\large $c_{j-2}$};
             \node at (3.65, 3.3) {\large $c_{j-1}$};
             \node at (2.5, 0.8) {\large $r''$};
             \node at (2.5, 4.7) {\large $r'$};
        \end{tikzpicture}
         \caption{Case 2: $q$ is bounded}\label{fig:onemaximal:arcintpiccomplicated}
     \end{figure}
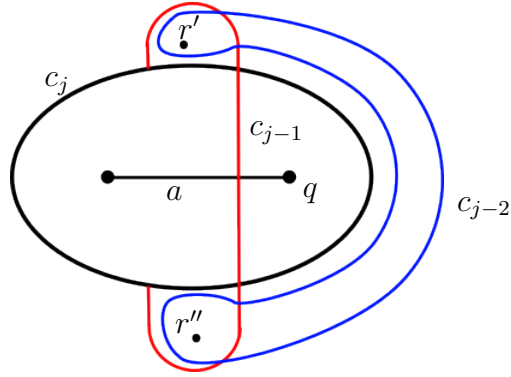

     \p{Case 3: $\delta'$ and $\delta''$ are nested} See Figure~\ref{fig:onemaximal:arcintpicmorecomplicated} for an example of this case.  
      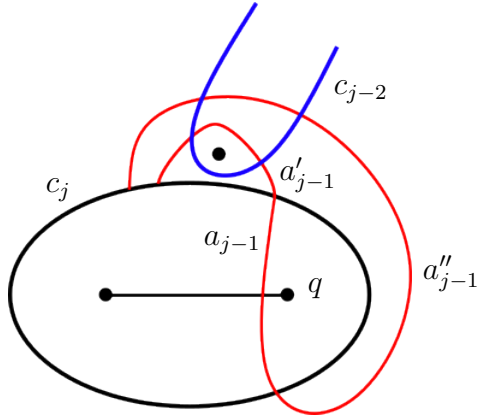
\begin{figure}[ht]
        \begin{tikzpicture}
            \node[anchor = south west] at (0,0) {\includegraphics[scale=0.6]{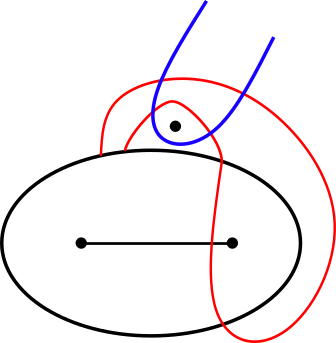}};
             \node at (0.8,3.1) {\large $c_j$};
             \node at (3.1,2.4) {\large $a_{j-1}$};
             \node at (4.1,3.3) {\large $a_{j-1}'$};
             \node at (6, 2) {\large $a_{j-1}''$};
              \node at (4.2, 1.8) {\large $q$};
              \node at (4.8, 4.4) {\large $c_{j-2}$};
        \end{tikzpicture}
         \caption{Case 3: $q$ is also bounded}\label{fig:onemaximal:arcintpicmorecomplicated}
     \end{figure} Without loss of generality, assume that $\delta' \leq \delta''$.  Since $c_j$ and $c_{j-1}$ intersect minimally, there is an $r' \in \Int(\delta')$.  Assuming we are not in Case 1, the same line of reasoning as Case 2 implies that we have $\calT_{\Phi}(r') = \{j-2,j-1\}$, so $r' \in \Int(\gamma_{j-2})$. Now, if $c_{j-2}$ does not intersect $a_{j-1}'$, we see that $\delta' \geq \gamma_{j-2}$ and thus $\gamma_{j-1}$ and $\gamma_j$ bound a puncture in $\gamma_{j-2}$ that is exterior to $\gamma_{j-1}$, and so specifically this puncture is of type $\{j-2,j-3\}$ by Theorem~\ref{thm:alpha:welltypedpuncture}.  Since $n \geq 5$ we have $\{j-2,j-3\} \cap \{j-1,j\} = \emptyset$, so we are in Case 1.  Similarly, if $c_{j-2}$ intersects $a_{j-1}'$ but does not intersect $a_{j-1}''$, we are in Case 1.  On the other hand if $\gamma_{j-2}$ intersects both $a_{j-1}'$ and $a_{j-1}''$, the fact that $\gamma_{j-2}$ does not intersect $\gamma_j$ implies that $\gamma_{j-1}$ and $\gamma_{j-2}$ bound at least one endpoint of $\alpha$.  The endpoints of $\alpha$ have type $\{j,j+1\}$ by construction.  Setting $i = j-2$ and using $n \geq 5$, we have we have$\{i,i+1\} \cap \{j,j+1\} = \emptyset$, so the lemma holds.  
 \end{proof}

\section{Externally central homomorphisms}\label{section:extcent}

Recall that a homomorphism $\Phi:\braid{n} \rightarrow \braid{m}$ is \defn{externally central} if $\Ext(\Phi(s_i))\in\Ext_{\crs{\Phi(s_i)}^{\outmulti}}(Z(\braid{m}))$ for some (and hence all) $s_i \in \extgen{n}$.  Our goal in Section~\ref{section:extcent} is to prove Theorem~\ref{thm:extcent}, which classifies externally central, irreducible, and non-cyclic homomorphisms $\Phi:\braid{n} \rightarrow \braid{m}$ with $n \geq 5$ and $m \geq 3$.  The proof of Theorem~\ref{thm:extcent} proceeds by examining the canonical reduction system of the element $s_is_{i+1}$.  We begin with the following basic structure property of externally central homomorphisms.

\begin{lemma}\label{lem:extremal:structural}
    Let $n\geq 5$ and $m\geq 3$.  Let $\Phi:\braid{n}\to \braid{m}$ be an externally central minimally typed homomorphism. Assume that $\Delta(\Phi) = \emptyset$.  Let $s_i,s_j \in \extgen{n}$ such that $[s_i,s_j]  = 1$.  Let $\gamma \in \extremalarg{s_j}$.  Then $\Phi(s_i)(\gamma) = \gamma$.  Furthermore, if $p \in \disk{m}$ has $i \not \in \calT_{\Phi}(p)$ then $\Phi(s_i)(p) = p$
\end{lemma}

\begin{proof}
    By the definition of externally central, it is enough to show that $\gamma$ and $p$ as above are exterior to $\extremalarg{s_i}$.  The former follows by Lemma~\ref{lem:crs:commdisj} and the definition of $\extremal$ while the latter follows from Lemma~\ref{lem:curves:ext_puncs}.  
\end{proof}  We now derive several consequences of Lemma~\ref{lem:extremal:structural}

\begin{lemma}\label{lem:extremal:sharepuncs:unique}
    Let $n\geq 5$ and $m\geq 3$.  Let $\Phi:\braid{n}\to \braid{m}$ be an externally central minimally typed homomorphism. Assume that $\Delta(\Phi) = \emptyset$.  Let $s_i \in \extgen{n}$ and let $\gamma\in \extremalarg{s_i}$. Then there is a unique curve $\gamma'$ in $\extremalarg{s_{i-1}}$ such that $|\Int(\gamma')\cap \Int(\gamma)|\neq \emptyset$.  Similarly, there is a unique $\gamma' \in \extremalarg{s_{i+1}}$ with $|\Int(\gamma')\cap \Int(\gamma)|\neq \emptyset$.  In particular $\Int(\gamma)$ contains punctures of type $\{i-1,i\}$ and $\{i,i+1\}$.
 \end{lemma}

\begin{proof}
    Suppose there are $p, q \in \Int(\gamma)$ such that $p\in \Int(\gamma')$ and $q\in \Int(\gamma'')$ with $\gamma', \gamma''\in \extremalarg{s_{i-1}}$. By Lemma~\ref{lem:extremal:structural}, it follows that $\Phi(s_i)(\{p,q\})\subset \Int(\gamma)$ and by Theorem~\ref{thm:alpha:welltypedpuncture} these punctures are of type $\{i+1, i\}$. We conclude that $\Phi(s_{i-1}s_i)(\{p, q\}) \subset \Int(\gamma)$ by Lemma~\ref{lem:extremal:structural}. Observe that $\Phi(s_{i-1}s_i)(p)$ is also interior to the curve $\Phi(s_{i-1}s_i)(\gamma') \in \extremalarg{s_i}$. However we have $\Phi(s_{i-1}s_i)(\gamma') \in \extremalarg{s_i}$ by Lemma~\ref{lem:alpha:extremalbraid}.  Therefore we have $\Phi(s_{i-1}s_i)(\gamma')=\gamma$. The same argument yields $\Phi(s_{i-1}s_i)(\gamma'')=\gamma$. We conclude that $\Phi(s_{i-1}s_i)(\gamma')=\Phi(s_{i-1}s_i)(\gamma'')$ and then $\gamma'=\gamma''$, as desired. 

    For the second part of the lemma, we apply some automorphism of $\braid{n}$ that sends $s_i$ to $s_{i}^{-1}$, and sends $s_{i-1}$ to $s_{i+1}^{-1}$ and vice versa.  The same line of reasoning as above then holds, so the result follows by combining this with Lemma~\ref{lem:crs:power}.

    For the third part of the lemma, observe that if $p \in \Int(\gamma)$ with $\calT_{\Phi}(p) = \{i,i+1\}$ then $\Phi(s_i)(p) \in \Int(\gamma)$ by Lemma~\ref{lem:extremal:structural} and also $\calT_{\Phi}(\Phi(s_i)(p)) = \{i-1,i\}$ by Theorem~\ref{thm:alpha:welltypedpuncture}.  The same argument holds for $\calT_{\Phi}(p) = \{i-1,i\}$.
\end{proof}  We can use Lemma~\ref{lem:extremal:sharepuncs:unique} to restrict the possible ways that $\Phi(a_1)$ can act on $\extremal$.  

\begin{lemma}\label{lem:extremal:forward}
    Let $n\geq 5$ and $m\geq 3$.  Let $\Phi:\braid{n}\to \braid{m}$ be an externally central minimally typed homomorphism. Assume that $\Delta(\Phi) = \emptyset$.  Let $s_i \in \extgen{n}$ and let $\gamma\in \extremalarg{s_i}$. Then $\Phi(s_{i}s_{i+1})(\gamma) = \Phi(a_1)(\gamma)$, and this is the unique $\gamma' \in \extremalarg{s_{i+1}}$ with $\Int(\gamma) \cap \Int(\gamma') \neq \emptyset$.  
\end{lemma}

\begin{proof}
  Write $a_1 = s_{i} s_{i+1} \cdot \ldots s_{i-2}$.  By Lemma~\ref{lem:extremal:structural}, $\Phi(s_j)(\gamma)=\gamma$ for $s_j\in\extgen{n}$ such that $[s_j,s_i] = 1$, so we have $\Phi(s_is_{i+1})(\gamma) = \Phi(a_1)(\gamma)$.  To show that this curve $\gamma'$ is the curve claimed in the lemma, choose $p \in \Int(\gamma)$ with $\calT(p) = \{i-1,i\}$.  Note that such a $p$ exists by Lemma~\ref{lem:extremal:sharepuncs:unique}.  Then by Lemma~\ref{lem:extremal:structural} we have $\Phi(s_{i+1})(p) = p$.  Furthermore, by Lemma~\ref{lem:extremal:structural} we have $\Phi(s_i)(\gamma) = \gamma$, and therefore $\Phi(s_i)(p) \in \Int(\gamma)$.  On the other hand, we have $\Phi(a_1)(p) \in \Int(\Phi(a_1)(\gamma))$ by definition.  Since $\Phi(s_i)(p) = \Phi(a_1)(p)$ by the above, we have
  \[
    \Phi(s_i)(p) \in \Int(\gamma) \cap \Int(\gamma').
  \]
  Therefore $\gamma'$ is the curve from Lemma~\ref{lem:extremal:sharepuncs:unique}.
\end{proof} 

Using the above, we can prove the following.

\begin{lemma}\label{lem:mainthm:factioncrs}
Let $n \geq 5$ and $m \geq 3$.  Let $\Phi:\braid{n} \rightarrow \braid{m}$ be an externally central and minimally typed homomorphism.  Assume that $\Delta(\Phi) = \emptyset$.  Let $f=s_is_{i+1}$ for some $s_i \in \extgen{n}$ and let $\gamma_i \in \extremalarg{s_i}$.  Then $\Phi(f^3)(\gamma_i) = \gamma_i$.  The same holds for $\gamma_{i+1} \in \extremalarg{s_{i+1}}$.
\end{lemma}

\begin{proof}
  Since $[f^3, s_i] = 1$ we have $\Phi(f^3)(\gamma_i) \in \crs{\Phi(s_i)}^{\outmulti}$ by Lemma~\ref{lem:crs:commdisj}.  Furthermore, $\crs{\Phi(s_i)}^{\outmulti} = \extremalarg{s_i}$ by Lemma~\ref{lem:curves:nonesting}, thus $\Phi(f^3)(\gamma_i) \in \extremalarg{s_i}$.  It therefore suffices to show that for some $p \in \Int(\gamma_i)$ we have $\Phi(f^3)(p) \in \Int(\gamma_i)$.  Choose $p \in \Int(\gamma_i)$ with $\calT_{\Phi}(p) = \{i-1,i\}$.  Such a $p$ exists by Lemma~\ref{lem:extremal:sharepuncs:unique}.  Using the braid relation, we may rewrite
  \[
    f^3 = s_i^2 s_{i+1} s_i^2 s_{i+1}.
  \]
  Using Lemma~\ref{lem:extremal:structural} we have
  \[
    \Phi(f^3)(p) = \Phi(s_i^2s_{i+1}s_i^2)(p).
  \]
  Using Lemma~\ref{lem:extremal:structural} again, we have $\Phi(s_i^2)(\gamma_i) = \gamma_i$ so $\Phi(s_i^2)(p)\in \Int(\gamma_i)$.  Now, let $p' = \Phi(s_i^2)(p')$.  By the same line of reasoning as above, we have $\Phi(s_i^2 s_{i+1})(p') \in \Int(\gamma_i)$.   Therefore $\Phi(f^3)(p) \in \Int(\gamma_i)$ so $\Phi(f^3)(\gamma_i)=\gamma_i$, as desired.
\end{proof}  We now show that we cannot have curves exterior to large subsets of $\extremal$.

\begin{lemma}\label{lem:extremal:filling}
    Let $n \geq 5$ and $m \geq 3$.  Let $\Phi:\braid{n} \rightarrow \braid{m}$ be an externally central and irreducible homomorphism.  Then there is no curve $\delta$ exterior to $\calS = \bigcup_{s_i \in \extgen{n}\setminus\{s_n\}} \extremalarg{s_i}$.
\end{lemma}

\begin{proof}
    Suppose by way of contradiction that such a $\delta$ exists.  By the definition of externally central, we have $\Phi(s_i)(\delta) = \delta$ for all $s_i \in \extgen{n} \setminus \{s_n\}$. Since $\extgen{n}\setminus\{s_n\}$ is a generating set for $\braid{n}$, it follows that $\{\delta\}$ is a reducing system for $\Phi$, a contradiction. 
\end{proof}  As a consequence, we have the following. Note that this lemma does not contradict Lemma~\ref{lem:extremal:sharepuncs:unique} since curves may intersect without having common interior punctures.

\begin{lemma}\label{lem:extcent:maxintersection}
    Let $n \geq 5$ and $m \geq 3$.  Let $\Phi:\braid{n} \rightarrow \braid{m}$ be an externally central and irreducible homomorphism.  Let $\gamma_i \in \extremalarg{s_i}$ and assume that $\left|\extremalarg{s_{i}}\right| \geq 2$.   Then $\gamma_i$ intersects at least three curves in $\extremal$.
\end{lemma}

\begin{proof}
  Without loss of generality, we prove the result for $i = 1$.  Suppose by way of contradiction that there is a $\gamma_1 \in \extremalarg{s_1}$ where $\gamma_1$ intersects only two curves in $\extremal$.  Note that $\gamma_1$ cannot intersect less than two curves in $\extremal$ by Lemma~\ref{lem:extremal:sharepuncs:unique}.  By Lemma~\ref{lem:alpha:standardrootaction}, $\Phi(a_1^k)(\gamma_1)$ intersects only two curves in $\extremal$ for all $k \in \ZZ$. Let $\overline{\extremal} = \bigcup_{s_i \in \extgen{n} \setminus \{s_n\}} \extremalarg{s_i}$, and let
  \[
    \calC = \left\{\Phi(a_1^k)(\gamma_1): 0 \leq k \leq n-2\right\}\subsetneq\overline{\extremal}.
  \]
  Let $\calC' = \overline{\extremal} \setminus \calC$.  Let $G$  denote the intersection graph of $\overline{\extremal}$, i.e., the graph with vertex set $\overline{\extremal}$, with edges between pairs of intersecting $\Phi$-maximal curves.  Let $H\subsetneq G$ be the subgraph spanned by vertices in $\calC$ and similarly $H' \subsetneq G$ spanned by vertices in $\calC'$.  By assumption, there is no edge $e$ in $G$ with one endpoint in $H$ and the other in $H'$.  Furthermore, by the definition of $\extremal$ there are no curves $\epsilon \in \calC$ and $\epsilon' \in \calC'$ with $\epsilon > \epsilon'$ or $\epsilon < \epsilon'$.  In particular, there is an essential curve $\nu \subset \disk{m}$ homotopic to a boundary component of $\disk{m} \cut \calC$ that is exterior to $\extremal$. Note that $\calC$ is the union of two multicurves, so $\disk{m}\cut\calC$ is indeed defined.  Since $\Phi$ is externally central, $\Phi(s_i)(\nu) = \nu$ for all $s_i \in \extgen{n} \setminus \{s_n\}$.  Since $\extgen{n} \setminus \{s_n\}$ generates $\braid{n}$, $\{\nu\}$ is a reducing system for $\Phi$, which is a contradiction.
\end{proof}  
We now study the canonical reduction systems of products of standard generators.
\begin{lemma}\label{lem:s1s2:no:curves:int:one}
    Let $n \geq 5$ and $m \geq 3$.  Let $\Phi:\braid{n} \rightarrow \braid{m}$ be an externally central and  minimally typed homomorphism.  Assume that $\Delta(\Phi) = \emptyset$. Let $f=s_{i}s_{i+1}$ for some $s_i \in \extgen{n}$.  Then $\crs{\Phi(f)}$ is exterior to $\crs{\Phi(s_j)}^{\outmulti}$ for $s_j\in\{s_i,s_{i+1}\}$.
\end{lemma}
\begin{proof}
   Let $\delta \in \crs{\Phi(f)}$. By Lemma~\ref{lem:crs:power} we have $\crs{\Phi(f^3)} = \crs{\Phi(f)}$.  Since $1=[f^3, s_i]=[f^3,s_{i+1}]$, the curve $\delta$ is disjoint from $\crs{\Phi(s_i)}$ and $\crs{\Phi(s_{i+1})}$ by Lemma~\ref{lem:crs:commdisj}. Looking for a contradiction, assume $\delta$ is interior to $\gamma_j\in \crs{\Phi(s_j)}$ for some $s_j \in \{s_i,s_{i+1}\}$. Without loss of generality assume $s_j=s_i$. Then, $\delta$ is not interior to $\extremalarg{s_{i+1}}$, as this would contradict Lemma~\ref{lem:curves:trivialnestingthing} since we have assumed that $\Delta(\Phi) = \emptyset$. But now, note that $\Phi(s_i)(\delta) \in \crs{\Phi(f^3)}$ by Lemma~\ref{lem:crs:commdisj}.  Therefore $\Phi(s_i)(\delta)$ is disjoint from $\extremalarg{s_{i+1}}$ by Lemma~\ref{lem:crs:commdisj} since $[f^3, s_{i+1}] = 1$.  On the other hand, $\Phi(s_i)(\delta)$ contains punctures of type $\{i,i+1\}$ by Theorem~\ref{thm:alpha:welltypedpuncture}.  This implies $\Phi(s_i)(\delta)$ is interior to $\extremalarg{s_{i+1}}$ which contradicts Lemma~\ref{lem:curves:trivialnestingthing} since we have assumed that $\Delta(\Phi) = \emptyset$.  
\end{proof}  We have the following consequence of Lemma~\ref{lem:s1s2:no:curves:int:one}.

\begin{lemma}\label{lem:extcent:fcubecrsfix}
Let $n \geq 5$ and $m \geq 3$.  Let $\Phi:\braid{n} \rightarrow \braid{m}$ be an externally central and minimally typed homomorphism.  Assume that $\Delta(\Phi) = \emptyset$. Let $f=s_{i}s_{i+1}$ for some $s_i \in \extgen{n}$.  Let $\delta \in \crs{\Phi(f)}$.  Then $\Phi(f)(\delta) = \delta$.
\end{lemma}

\begin{proof}
    By Lemma~\ref{lem:s1s2:no:curves:int:one} we have $\delta$ exterior to $\crs{\Phi(s_i)}^{\outmulti}$ and $\crs{\Phi(s_{i+1})}^{\outmulti}$.  Therefore $\Phi(s_i)(\delta) = \delta$ and $\Phi(s_{i+1})(\delta) = \delta$ by the definition of externally central, so the lemma holds.
\end{proof}We need the following auxiliary lemma. 

\begin{lemma}\label{lem:mainthm:cubefourth}
    Let $n \geq 5$ and $m \geq 3$.  Let $\Phi:\braid{n} \rightarrow \braid{m}$ be an externally central and minimally typed homomorphism.  Let $f=s_is_{i+1}$ for some $s_i \in \extgen{n}$. If $p\in \disk{m}$ is a puncture of type $\calT_{\Phi}(p)=\{i-1, i\}$, then $\Phi(f^3)(p)=\Phi(s_i^4)(p)$. 
\end{lemma}
\begin{proof}
  Using the braid relation we may rewrite $f^3$ as
  \[
    f^3=(s_is_{i+1})^3=s_i^2s_{i+1}s_i^2s_{i+1}.
  \]
  Lemma~\ref{lem:curves:ext_puncs} implies that any puncture $q \in \disk{m}$ with $\calT_{\Phi}(q) = \{i-1,i\}$ is exterior to $\crs{\Phi(s_{i+1})}$.  Since $\Phi(s_{i+1})$ is externally central, it fixes punctures of type $\{i-1, i\}$. Then Theorem~\ref{thm:alpha:welltypedpuncture} implies that for any $q \in \disk{m}$ with $\calT_\Phi(q) = \{i-1,i\}$ we have
  \[
    \calT_{\Phi}(\Phi(s_i^2)(q)) = \{i-1,i\}.
  \]
  It follows that $\Phi(f^3)(p)=\Phi(s_i^4)(p)$. 
\end{proof}  

Our goal is to understand the action of $\Phi(s_i)$ on punctures of type $\{i-1,i\}$.  If $h \in \Mod(\disk{m})$ is a mapping class and $\gamma \subset \disk{m}$ is a curve such that $h(\gamma) = \gamma$, we let $h|_{\diskinner{\gamma}}$ denote the restriction of $h$ to the interior disk $\diskinner{\gamma}$ bounded by $\gamma$.  Using Lemma~\ref{lem:mainthm:cubefourth}, we have the following.

\begin{lemma}\label{lem:action:on:puncs}
   Let $n \geq 5$ and $m \geq 3$. Let $\Phi:\braid{n} \rightarrow \braid{m}$ be an externally central and irreducible homomorphism.  If $p\in \disk{m}$ is a puncture with $i \in \calT_{\Phi}(p)$ then $\Phi(s_i^4)(p)=p$.
\end{lemma}
\begin{proof}
Since $\Phi$ is irreducible, $\Phi$ is minimally typed by Theorem~\ref{thm:alpha:welltypedpuncture}.  We prove this for $\calT_{\Phi}(p) = \{i-1,i\}$ as the other case is similar.  By Lemma~\ref{lem:mainthm:cubefourth} it is enough to show that $\Phi(f^3)(p) = p$ for $f = s_is_{i+1}$.  Choose $\gamma_i \in \extremalarg{s_i}$ and $p \in \Int(\gamma_i)$ with $\calT_{\Phi}(p) = \{i-1,i\}$.  Such a $p$ exists by Lemma~\ref{lem:extremal:sharepuncs:unique}.  Since $[f^3, s_i] = 1$, the curve $\gamma_i$ is disjoint from $\crs{\Phi(f)}$.  Let $\nu$ denote the minimal curve in $\crs{\Phi(f)}$ that contains $\gamma_i$, or let $\nu = \partial \disk{m}$ if no such curve exists.  Note that $\Phi(f)(\nu) = \nu$ by Lemma~\ref{lem:extcent:fcubecrsfix}.  Let $\diskinner{\nu} \subseteq \disk{m}$ denote the disk bounded by $\nu$ and let $M$ be the subset of $\crs{\Phi(f)}$ consisting of the curves that are interior to $\nu$. Note that $p$ is exterior to $M$ by construction. The braid $h = \Ext_{M^{\outmulti}}(\Phi(f)|_{\diskinner{\nu}})$ is either periodic or pseudo-Anosov by Lemma~\ref{lem:centralizer:irreducible}.  Since $\Phi(f^3)(\gamma_i) = \gamma_i$ by Lemma~\ref{lem:mainthm:factioncrs} we conclude that $h$ is periodic by Lemma~\ref{lem:crs:infpA}.  We now have two cases.

\p{Case 1: $\left|\extremalarg{s_{i+1}}\right| \geq 2$}  By Lemma~\ref{lem:extcent:maxintersection} $\gamma_i$ intersects at least three curves in $\extremal$.  If two such curves lie in $\extremalarg{s_{i+1}}$ then we proceed.  Otherwise if $\gamma_i$ intersects only a unique $\gamma_{i+1} \in \extremalarg{s_{i+1}}$ then $\gamma_{i+1}$ must intersect at least two curves in $\extremalarg{s_i}$.  Indeed, $\gamma_{i+1}$ is the unique curve intersecting $\gamma_i$ from Lemma~\ref{lem:extremal:sharepuncs:unique}, so $\gamma_{i+1} = \Phi(a_1)(\gamma_i)$ by Lemma~\ref{lem:extremal:forward}.  Therefore since $\gamma_{i}$ intersects at least two curves in $\extremalarg{s_{i-1}}$ then Lemma~\ref{lem:alpha:standardrootaction} implies that $\gamma_{i+1}$ intersects at least two curves in $\extremalarg{s_i}$. Then by Lemma~\ref{lem:s1s2:no:curves:int:one} we have $\crs{f}$ exterior to $\extremalarg{s_i}$ and $\extremalarg{s_{i+1}}$.  In particular, we see that any curve in $\extremalarg{s_{i+1}}$ that intersects $\gamma_i$ is interior to $\nu$ and similarly for any curve in $\extremalarg{s_i}$ that intersects $\gamma_{i+1}$.  Furthermore, since $\nu$ was chosen to be minimal over $\crs{f}$, we see actually that $\extremalarg{s_i}$ and $\extremalarg{s_{i+1}}$ are exterior to $M$.  Since at least one of $\gamma_i$ or $\gamma_{i+1}$ intersects multiple curves in $\extremalarg{s_{i+1}}$ or $\extremalarg{s_i}$ respectively by the above, $h^3$ fixes multiple curves by Lemma~\ref{lem:mainthm:factioncrs}.  Therefore $h^3$ is trivial by Lemma~\ref{lem:centralizer:fixtwocurves}.  In particular, $f^3(p) = p$, as desired.

\p{Case 2: $\left|\extremalarg{s_{i+1}}\right| = 1$}  In this case $\Phi$ is simple as in Section~\ref{section:onemaximal}.  If there is only one puncture of type $\{i,i+1\}$, then $\Phi(s_i^2)(p) = p$ by Theorem~\ref{thm:alpha:welltypedpuncture}.  Therefore, we may assume that there are least two punctures of type $\{i,i+1\}$.  Therefore $\gamma_i$ and $\gamma_{i+1}$ bound a puncture $q$ with $\calT_{\Phi}(q) \cap \{i,i+1\} = \emptyset$ by Lemma~\ref{lem:onemaximal:twopunc}.  Then $\Phi(s_i)$ and $\Phi(s_{i+1})$ both fix $q$ by Lemma~\ref{lem:extremal:structural}. Since $f=s_is_{i+1}$,  $\Phi(f)$ also fixes $q$. Then $h^3$ fixes $\gamma_i$ and fixes $q$ which is exterior to $\gamma_i$ so again $h^3$ is central by Lemma~\ref{lem:centralizer:fixtwocurves}.  As above this implies $\Phi(f)^3(p) = p$.
\end{proof}

We are now ready to prove that Theorem~\ref{mainthm:natleast5} holds if we assume additionally that $\Phi$ is externally central.  

\begin{theorem}\label{thm:extcent}
Let $n \geq 5$ and $m \geq 3$.  Let $\Phi:\braid{n} \rightarrow \braid{m}$ be an irreducible, non-cyclic, externally central homomorphism.  Then $m = n$ and $\Phi$ is centrally equivalent to the identity.
\end{theorem}

\begin{proof}
    By Theorem~\ref{thm:alpha:welltypedpuncture} and Proposition~\ref{prop:curves:valid} we may assume $\Phi$ is a minimally typed homomorphism with $\Delta(\Phi)=\emptyset$. Also, note Lemma~\ref{lem:action:on:puncs} yields that 
    $\Phi(s_1^4)(p)=p$ for every puncture $p$ of type $\calT_{\Phi}(p)=\{n, 1\}$.  Recall that $a_2=s_1^2s_2\ldots s_{n-1}$. Lemma~\ref{lem:extremal:structural} implies that $\Phi(a_2^2)(p)=\Phi(s_1^4)(p)=p$ for every puncture of type $\calT_{\Phi}(p)=\{n,1\}$.  Note that if there are at least two punctures of type $\{1,n\}$, then $\Phi(a_2^2)$ is central by Lemma~\ref{lem:crs:fixtwopunc}.  But this contradicts Lemma~\ref{lem:crs:order}.  Therefore, there is at most one puncture of type $\{1,n\}$.  Assume there is a puncture of empty type $p\in \disk{m}$. Lemma~\ref{lem:extremal:structural} implies $\Phi(s_i)(p)=p$ for every $s_i\in \extgen{n}$. In particular, $\Phi(a_2^2)$ fixes at least two punctures: $p$ and a puncture of type $\{1,n\}$. Again $\Phi(a_2^2)=1$ and the argument above produces a contradiction. Therefore for every $s_i \in \extgen{n}$ there is exactly one puncture $p \in \disk{m}$ with $\calT_\Phi(p) = \{i,i+1\}$ and there is no $p \in \disk{m}$ with $\calT_{\Phi}(p) = \emptyset$.  In particular, we have $m=n$ and the statement follows from Theorem~\ref{thm:case:m:leq:n}.
\end{proof}

\section{The proof of Theorem~\ref{mainthm:natleast5}}\label{section:mainthm}

Our goal is to prove that if $n \geq 5$ and $m \geq 3$ and $\Phi:\braid{n} \rightarrow \braid{m}$ is an irreducible and non-cyclic homomorphism, then $m = n$ and $\Phi$ is centrally equivalent to the identity.  In light of Theorem~\ref{thm:extcent} it is enough to reduce to the case that $\Phi$ is externally central.  Let $\Phi:\braid{n} \rightarrow G$ be a homomorphism and let $f \in \cent{\Phi(\braid{n})}{G}$.  The \defn{transvection} (see Section~\ref{section:external}) of $\Phi$ along $f$ is defined on $s_i \in \extgen{n}$ by
\[
  \Phi_f(s_i) = \Phi(s_i) f.
\]

A first step is to show that if $\Phi$ is a homomorphism as in the statement of Theorem~\ref{mainthm:natleast5} then there is some $f \in \cent{\Phi(\braid{n})}{\braid{m}}$ where $\Phi_f$ is externally central while still being irreducible and non-cyclic.  We begin with the following.

\begin{lemma}\label{lem:mainthm:extremalfilling}
Let $n \geq 5$ and $m \geq 3$.  Let $\Phi:\braid{n} \rightarrow \braid{m}$ be an irreducible and non-cyclic homomorphism.  The set $\extremal$ is filling.
\end{lemma}

\begin{proof}
  Suppose by way of contradiction that $\extremal$ is not filling.  Let
  \[
    \calN = \{\delta \subset \disk{m}: \delta \text{ disjoint from } \extremal\}.
  \]
  Note that any $\delta \in \calN$ is exterior to $\extremal$.  Indeed, if $\delta < \gamma$ for some $\gamma \in \extremal$, then by Theorem~\ref{thm:alpha:welltypedpuncture} there is some $p \in \Int(\delta)$ with $\calT_{\Phi}(p) = \{i,i+1\}$ for some $s_i \in \extgen{n}$.  But then since $\delta$ is disjoint from $\extremal$ by assumption, there is some $\gamma' \in \extremal$ with $\delta \leq \gamma
$.  Since $\Delta(\Phi) = \emptyset$ by Proposition~\ref{prop:curves:valid} the existence of this $\delta$ contradicts Lemma~\ref{lem:curves:trivialnestingthing}.

Now, let $M$ denote the minimal curves in $\calN$. The set $\calN$ is nonempty by assumption and no curve in $\calN$ is interior to any $\gamma \in \extremal$ by the above, so $M$ is nonempty.  Since $n \geq 5$, the graph $\Comm_{\braid{n}}(\extgen{n})$ is connected.  Therefore by Lemma~\ref{lem:alpha:exterioragreement} we have $\Phi(s_i)(\calN) = \calN$ for all $s_i \in \extgen{n}$ and $\Phi(s_i)(\delta) = \Phi(s_j)(\delta)$ for all $s_i,s_j \in \extgen{n}$ and $\delta \in \calN$.  Therefore $M$ is a $\Phi(\braid{n})$-invariant set of curves.  Furthermore, if $\delta, \delta' \in M$ were to intersect, then some $\nu \subset \disk{m}$ homotopy equivalent to a boundary component of $\disk{m} \cut (\delta \cup \delta')$ would satisfy $\nu \in \calN$ and $\nu < \delta$ and $\nu < \delta'$, contradicting the minimality of $\delta$ and $\delta'$.  Therefore $M$ is a reducing system for $\Phi$.
\end{proof}

We obtain the following fact about centralizers of irreducible and non-cyclic homomorphisms.

\begin{lemma}\label{lem:mainthm:centralizerperiodic}
Let $n \geq 5$ and $m \geq 3$.  Let $\Phi:\braid{n} \rightarrow \braid{m}$ be an irreducible and non-cyclic homomorphism.  Let  $f \in \cent{\Phi(\braid{n})}{\braid{m}}$.  Then $f$ is periodic.
\end{lemma}

\begin{proof}
    Let $f \in \cent{\Phi(\braid{n})}{\braid{m}}$.  Note that this implies $\Phi(\braid{n}) \subseteq \cent{f}{\braid{m}}$.  Suppose by way of contradiction that $f$ is not periodic.  By Theorem~\ref{cor:crs:NTbraid}, we have two cases to consider.

\p{Case 1: $f$ is pseudo-Anosov} By~\cite{GMW} the centralizer $\cent{f}{\braid{m}} \cong \ZZ^2$.  In particular since $\braid{n}^{\ab} \cong \ZZ$, this implies that $\Phi$ is cyclic, a contradiction.

\p{Case 2: $f$ is aperiodic reducible}  In this case $\crs{f}$ is a reducing system for $\Phi$ by Lemma~\ref{lem:crs:commdisj}, so $\Phi$ is reducible, a contradiction.
\end{proof}

\begin{lemma}\label{lem:mainthm:irreducibletransvection}
   Let $n \geq 5$ and $m \geq 3$.  Let $\Phi:\braid{n} \rightarrow \braid{m}$ be an irreducible and non-cyclic homomorphism.  Let $f \in \cent{\Phi(\braid{n})}{\braid{m}}$. Then the transvection $\Phi_f$ is also irreducible and non-cyclic.
\end{lemma}

\begin{proof}
    If $\Phi_f$ were cyclic, then Lemma~\ref{lem:alpha:transvectioncyclic} would imply that $(\Phi_f)_{f^{-1}} = \Phi$ would be cyclic, a contradiction.

    Suppose now by way of contradiction that $\Phi_f$ has a reducing system $M$.  The element $f$ is periodic by Lemma~\ref{lem:mainthm:centralizerperiodic}. Note that since canonical reduction systems are preserved under powers by Lemma~\ref{lem:crs:power}, we may choose some $r \in \ZZ_{> 0}$ with $f^r \in Z(\braid{m}). $  By the definition of transvection we have
    \[
      \Phi_f(s_i^r)= \Phi(s_i)^rf^r
    \]
    for all $s_i \in \extgen{n}$, so we conclude
    \[
      \crs{\Phi(s_i)} = \crs{\Phi_f(s_i)}
    \]
    for all $s_i \in \extgen{n}$. Now, let $\delta\in M$. By Lemma~\ref{lem:mainthm:extremalfilling} there is an $s_i\in \extgen{n}$ such that $\delta$ intersects some curve $\gamma\in \crs{\Phi(s_i)}$. Therefore Lemma~\ref{lem:crs:infdisj} implies that $\delta$ has an infinite $\Phi_f(s_i)$--orbit, but this contradicts that $M$ is a reducing system of $\Phi_f$.
  \end{proof}  Let $M\subset \disk{m}$ be a non-nested multicurve. Let $T_M$ denote the subgroup of $\braid{m}$ generated by Dehn twists along curves in $M$.  A multitwist $t \in \braid{m}$ is \defn{supported} on $M$ if $t \in T_M$.  Denote $S_{M}\subseteq \disk{m}$  the union of punctured spheres interior to the sub-disks bounded by curves of $M$. We say that two curves $\gamma, \gamma' \subset \disk{m}$ are \defn{isotopic away from $M$} if there are representatives $c$ of $\gamma$ and $c'$ of $\gamma'$ such that
  \[
    c\cap (\disk{m}\setminus S_{M})=c'\cap (\disk{m}\setminus S_{M}),
  \]
and $c\cap S_M$ is isotopic to $c'\cap S_M$ in $S_M$.  We record the following elementary fact about curves isotopic away from $M$.  

\begin{lemma}\label{lem:mainthm:gluingaction}
Let $M\subset \disk{m}$ be a non-nested multicurve. Let $\gamma, \gamma' \subset \disk{m}$ be two curves isotopic away from $M$.  There is a multitwist $t\in \braid{m}$  supported on $M$ such that $\gamma'=t(\gamma)$.
\end{lemma}
\begin{proof}
    Let $c, c'$ be representatives of $\gamma, \gamma'$ such that $c\cap (\disk{m}\setminus S_{M})=c'\cap (\disk{m}\setminus S_{M})$ and $c\cap S_M$ is isotopic to $c'\cap S_{M}$ in $S_M$. Let $h:S_M\to S_M$ be a homeomorphism isotopic to the identity with $h(c\cap S_M)=c'\cap S_M$. Since $c\cap (\disk{m}\setminus S_{M})=c'\cap (\disk{m}\setminus S_{M})$, we may choose an $h$ that extends to a homeomorphism $\widehat{h}:\disk{m}\to \disk{m}$ acting as the identity on $\disk{m}\setminus S_{M}$. As a consequence, $\widehat{h}(c)=c'$.  Note that the mapping class $[\widehat{h}]$ fixes every curve disjoint from $M$. In particular, $t=[\widehat{h}]\in \braid{m}$ is multitwist supported on $M$ and $\gamma'=t(\gamma)$ as a consequence of Farb--Margalit~\cite[Corollary 13.3]{FM}.
\end{proof}

We now prove the following.

\begin{lemma}\label{lem:mainthm:externallytrivial}
Let $n \geq 5$ and $m \geq 3$.  Let $\Phi:\braid{n} \rightarrow \braid{m}$ be an irreducible and non-cyclic homomorphism.  Then there is an $f \in \cent{\Phi(\braid{n})}{\braid{m}}$ such that the transvection $\Phi_{f^{-1}}$ is externally central.
\end{lemma}

\begin{proof}
  For all $s_i \in \extgen{n}$ let $M_i = \crs{\Phi(s_i)}^{\outmulti}$.  As a consequence of Lemma~\ref{lem:curves:nonesting} we have $M_i = \extremalarg{s_i}$.  Therefore $M_i \cup M_j$ is a non-nested multicurve by Lemma~\ref{lem:curves:nonesting} for any $[s_i,s_j] = 1$.  By Lemma~\ref{lem:alpha:exterioragreement} we have
  \[
    \Ext_{M_i \cup M_j}(\Phi(s_i)) = \Ext_{M_i \cup M_j}(\Phi(s_j)).
  \]
  In particular, Lemma~\ref{lem:centralizer:sheaf} says that there is a unique element $f_{i,j} \in \braid{M_i \cap M_j}$ satisfying:
    \begin{itemize}
        \item $\Ext_{M_i}(f_{i,j}) = \Ext(\Phi(s_i))$; and
        \item $\Ext_{M_j}(f_{i,j}) = \Ext(\Phi(s_j))$.
        \end{itemize} Since $\Phi$ is irreducible, Lemma~\ref{lem:alpha:extremalcollision} implies that $M_i \cap M_j = \emptyset$, so in fact $f_{i,j} \in \braid{m}$.  Now, by construction the element $f_{i,j}$ behaves as $\Phi(s_j)$ interior to $M_i$ and vice versa.  Therefore we have $[f_{i,j}, \Phi(s_i)] = [f_{i,j}, \Phi(s_j)] = 1$.  Furthermore, $\Phi(s_i)f_{i,j}^{-1}(\gamma) = \gamma$ for all $\gamma \in M_i$.  In particular, to show that $f_{i,j}$ is the element $f$ as in the statement of the lemma, it is enough to show that $f_{i,j} = f_{k,\ell}$ for all $s_k, s_\ell \in \extgen{n}$ with $[s_k, s_\ell] = 1$.  Lemma~\ref{lem:mainthm:extremalfilling} says that $\extremal$, which is by definition
        \[M_1 \cup \ldots \cup M_n,\]
        is filling, so it suffices to show that $f_{i,j}$ and $f_{k,\ell}$ agree on all $\gamma \in \extremal$.  In particular, it is enough to prove that for all $s_h \in \extgen{n}$ and all $\gamma \in M_h$, we have
        \[
          f_{i,j}(\gamma) = \Phi(s_h)(\gamma).
        \]
        If $[s_h,s_i] = 1$ then this follows from Lemma~\ref{lem:curves:nonesting} which says that $M_h \cup M_i$ is a non-nested multicurve, and Lemma~\ref{lem:alpha:exterioragreement}.  The same argument holds if $[s_h, s_j]= 1$.  It therefore suffices to prove the result in the case that $[s_h,s_i] \neq 1$ and $[s_h, s_j] \neq 1$.  Since $n \geq 5$, this occurs precisely for $j = i+2$ and $h = i+1$.  
    
    Now, pick $\gamma \in M_h$.  Since $n \geq 5$ the graph $\Comm_{\braid{n}}(\extgen{n})$ has diameter 2.  In particular, there is some $s_k \in \extgen{n}$ with $[s_k,s_i] = [s_k,s_{i+1}] = 1$.  By Lemma~\ref{lem:alpha:exterioragreement} and the definition of $f_{i, i+2}$, we have that $\Phi(s_i), \Phi(s_k)$  and $f_{i. i+2}$ agree on the common exterior to $M_i$ and $M_k$. In particular, there are representatives $c$ of $f_{i,i+2}(\gamma)$ and $c'$ of $\Phi(s_k)(\gamma)=\Phi(s_{i+1})(\gamma)$ such that $c\cap (\disk{m}\setminus S_{M_i})=c'\cap (\disk{m}\setminus S_{M_i})$.  A similar line of reasoning yields that  $f_{i,i+2}(\gamma)$ and $\Phi(s_{i+1})(\gamma)$ are isotopic on $\disk{m}\setminus S_{M_{i+2}}$. Observe that $S_{M_{i}}\subseteq\disk{m}\setminus S_{M_{i+2}}$ so we may conclude $c\cap S_{M_{i}}$ is isotopic to $c'\cap S_{M_{i}}$ in $S_{M_i}$. That is, we proved that $f_{i, i+2}(\gamma)$ is isotopic to  $\Phi(s_{i+1})(\gamma)$ away from $M_i$. 
    
    By Lemma~\ref{lem:mainthm:gluingaction} we have that $f_{i, i+2}(\gamma)=t_i( \Phi(s_{i+1})(\gamma))$ for some multitwist $t_i\in T_{M_i}$. However, a symmetric argument yields that $f_{i, i+2}(\gamma)=t_{i+2}( \Phi(s_{i+1})(\gamma))$ for some multitwist $t_{i+2}\in T_{M_{i+2}}$. Since the curves in $M_i$ and $M_{i+2}$ are pairwise distinct by Lemma~\ref{lem:alpha:extremalcollision}, disjoint by Lemma~\ref{lem:crs:commdisj} and non-nested by Lemma~\ref{lem:curves:nonesting}, we conclude that $t_i=t_{i+2}=1$. Thus $f_{i, i+2}(\gamma)= \Phi(s_{i+1})(\gamma)$, as desired.
\end{proof}  We are now ready to conclude the proof of Theorem~\ref{mainthm:natleast5}.  Recall that this said that if $\Phi:\braid{n} \rightarrow \braid{m}$ is an irreducible and non-cyclic homomorphism with $n \geq 5$ and $m \geq 3$, then $\Phi$ is centrally equivalent to the identity. 

\begin{proof}[Proof of Theorem~\ref{mainthm:natleast5}]
    By Lemma~\ref{lem:mainthm:externallytrivial} there is an element $f \in \cent{\Phi(\braid{n})}{\braid{m}}$ such that $\Phi_f$ is externally central.  By Lemma~\ref{lem:mainthm:irreducibletransvection} the homomorphism $\Phi_f$ is still irreducible and non-cyclic.  By Theorem~\ref{thm:extcent} we have $m = n$.  The theorem then follows from Theorem~\ref{thm:case:m:leq:n}.
  \end{proof}  We can also reformulate Theorem~\ref{mainthm:natleast5} constructively.  Recall Dyer--Grossman's~\cite{DyerGrossman} description of $\Aut(\braid{n})$ as a split exact sequence
  \[
    1 \rightarrow \braid{n}/Z(\braid{n}) \rightarrow \Aut(\braid{n}) \rightarrow \Z/2\Z \rightarrow 1.
  \]

\setcounter{altthm}{1}

\begin{altthm}\label{mainthmalt:natleast5}
  Let $n \geq 5$ and $m \geq 3$.  Let $\Phi:\braid{n} \rightarrow \braid{m}$ be an irreducible and non-cyclic homomorphism.  Then $m = n$.  Furthermore, let $\nu:\braid{n} \rightarrow \ZZ$ denote the abelianization map.  Then there is some $k \in \ZZ$ and automorphism $\rho \in \Aut(\braid{n})$ such that, for all $f\in\braid{n}$,
  \[
    \Phi(f) = \rho(f)\cdot T_{\partial\disk{n}}^{k \nu(f)}.
  \]
\end{altthm}

\begin{proof}
  By Theorem~\ref{mainthm:natleast5} we have $m = n$ and $\Phi$ centrally equivalent to the identity map. Therefore, there is some $\rho \in \Aut(\braid{n})$ and a function $\zeta:\braid{n} \rightarrow Z(\braid{n})$ such that
  \[
    \Phi(f) = \rho(f) \cdot \zeta(f)
  \]
  for all $f \in \braid{n}$.  Since $\braid{n}^{\ab}\cong\Z$ and $Z(\braid{n})=\langle T_{\partial\disk{n}}\rangle$ is cyclic, it suffices to show that $\zeta$ is a homomorphism. And indeed, for all $f,g\in\braid{n}$ we have that
    \[
    \zeta(fg) = \Phi(f) \Phi(g) \rho(g)^{-1} \rho(f)^{-1} = \Phi(f) \zeta(g) \rho(f)^{-1} = \Phi(f)\rho(f)^{-1} \zeta(g) = \zeta(f) \zeta(g).\qedhere
    \]
\end{proof}

\section{Holomorphic maps}\label{section:holomorphic}

The purpose of this section is to prove Theorem~\ref{mainthm:holo} and Theorem~\ref{mainthm:hyp-to-hyp} along with several other holomorphic rigidity consequences of Theorem~\ref{mainthm:natleast5}.  Let $g,r\geq0$ and $2g+r\geq3$. Recall that $\calM_{g,r}$ refers to the moduli space of genus $g$ Riemann surfaces with $r$ marked points. The mapping class group $\Mod_{g,r}$ is the group $\Homeo^+(\Sigma)/\sim$, where $\sim$ denotes isotopy and $\Sigma$ is a topological surface of genus $g$ with $r$ punctures. The \defn{pure mapping class group} $\PMod_{g,r}$ is the subgroup of $\Mod_{g,r}$ that additionally fixes each puncture.

\p{Moduli space and orbifolds} The moduli space $\calM_{g,r}$ can be regarded as the orbifold quotient
\[
    \calM_{g,r} = \calT_{g,r} / \PMod_{g,r},
\]
where $\calT_{g,r}$ denotes Teichm\"{u}ller space, which is biholomorphic to a bounded domain in $\mathbb{C}^{3g-3+r}$. While $\Mod_{g,r}$ acts discretely on $\calT_{g,r}$, the action is not free,  hence  why we regard $\calM_{g,r}$ as an orbifold. When $\calM_{g,r}$ is equipped with its orbifold structure, we have $\pi_1^{\rm orb}(\calM_{g,r})\cong\PMod_{g,r}$. Similarly, we have $\pi_1^{\rm orb}(\calM_{g,r}/\mathfrak{S}_r)\cong\Mod_{g,r}$. The coarse moduli space of $\calM_{g,r}$ refers to the underlying space obtained by forgetting the orbifold structure, which is a quasi-projective variety. Note that $\calM_{g,r}$ admits a finite orbifold cover that is in fact a smooth quasi-projective variety.  A holomorphic map $X\to\calM_{g,r}$ is \defn{isotrivial} if it is constant when viewed as a map to the coarse moduli space. 

\p{Homotopic holomorphic maps} Imayoshi--Shiga proved a rigidity result for the monodromy of families of Riemann surfaces over a Riemann surface~\cite{IS88}. The second and fourth authors have recorded a generalization of their result based on combining a generalization due to Antonakoudis--Aramayona--Souto~\cite[Proposition 3.2]{AAS18} (allowing the base to be quasi-projective) and another independent generalization due to Chen and Salter~\cite[Theorem 2.5]{ChenSalter2026} (also allowing anti-holomorphic maps). 

\begin{theorem}[{\cite[Theorem 7.5]{HuxfordSchillewaert2025}}]\label{thm:moduli-rigidity}
    Let $g,r\geq0$ with $2g+r\geq3$, and let $X$ be a smooth quasi-projective variety. If $\Psi_1,\Psi_2\colon X\to\calM_{g,r}$ are non-isotrivial holomorphic or anti-holomorphic maps that are homotopic in the category of orbifolds, then $\Psi_1=\Psi_2$.
\end{theorem}

We note that two continuous maps $\Psi_1,\Psi_2\colon X\to\calM_{g,r}$ are homotopic in the category of orbifolds if and only if the induced homomorphisms $(\Psi_1)_*,(\Psi_2)_*\colon\pi_1(X)\to\PMod_{g,r}$ differ by at most an inner automorphism of $\PMod_{g,r}$, since $\calT_{g,r}$ is a classifying space for proper actions of $\PMod_{g,r}$~\cite{JW10}.

\p{Irreducibility} The following theorem of the first author and Souto has been shown in the special case where $X$ is a compact Riemann surface and $r=0$ by McMullen~\cite{McM00}. Daskalopoulos--Wentworth also reach a stronger conclusion~\cite[Theorem 5.7]{DW07}, but their results are also only stated explicitly in the special case where $X$ is a compact Riemann surface and $r=0$.  Recall that a homomorphism to $\Mod_{g,r}$ is irreducible if its image does not preserve a multicurve. Also, note that all smooth connected varieties are irreducible.

\begin{theorem}[{\cite[Theorem 1.2]{DPS24}}]\label{thm:hol-implies-irred}
    Let $g,r\geq0$ with $2g+r\geq3$, $X$ be an irreducible quasi-projective variety, and $\calM$ be a finite cover of $\calM_{g,r}/\mathfrak{S}_r$. If $\Psi\colon X\to\calM$ is a non-isotrivial holomorphic map, then $\Psi_*\colon\pi_1(X)\to\Mod_{g,r}$ is irreducible.
\end{theorem}

\begin{proof}
    The first author and Souto prove this result when $\calM=\calM_{g,r}$, however from their result we can deduce the general case. Let $\tilde{X}$ be a finite connected cover of $X$ such that $\Psi$ lifts to a holomorphic map $\tilde{\Psi}\colon\tilde{X}\to\calM_{g,r}$. Note that $\tilde{\Psi}$ is not isotrivial since $\Psi$ is not isotrivial. By a result of Grothendieck, $\tilde{X}$ is also a quasi-projective variety~\cite[Expos{\'e} XII, Th{\'e}or{\`e}me 5.1]{Gro63}. By~\cite[Theorem 1.2]{DPS24} we have $\tilde{\Psi}_*\colon\pi_1(\tilde{X})\to\PMod_{g,r}$ is irreducible, i.e.\ the restriction of $\Psi_*$ to the subgroup $\pi_1(\tilde{X})$ of $\pi_1(X)$ is irreducible. Hence $\Psi_*$ is also irreducible.
\end{proof}

We can now show that holomorphic maps that induce cyclic homomorphisms on the level of orbifold fundamental groups are in fact isotrivial.

\begin{theorem}\label{thm:hol-implies-noncyclic}
    Let $n\geq3$, $g,r\geq0$ with $2g+r\geq3$, and $\calM$ a finite cover of $\calM_{g,r}/\mathfrak{S}_r$. If $\Psi\colon\conf{n}\to\calM$ is a holomorphic map and $\Psi_*\colon\braid{n}\to\Mod_{g,r}$ has cyclic image, then $\Psi$ is isotrivial.
\end{theorem}

\begin{proof}
    By~\cite[Proposition 7.2]{HuxfordSchillewaert2025} for each $i=1,\ldots,n-1$ we have that $\Psi_*(s_i)$ is a root of a possibly empty multitwist in $\Mod_{g,r}$, i.e.\ it is a root of a non-identity multitwist, or it is finite order. Since $\braid{n}=\langle s_1,\ldots,s_{n-1}\rangle$, this implies that $\Psi_*(\braid{n})$ is generated by a root of a possibly empty multitwist.

    If $\Psi_*(\braid{n})$ is generated by a root of a non-identity multitwist, then $\Psi_*$ is reducible, so by Theorem~\ref{thm:hol-implies-irred} $\Psi$ is isotrivial. Suppose then that $\Psi_*(\braid{n})$ is finite cyclic. Then there is a finite connected cover $X$ of $\conf{n}$ so that $\Psi$ lifts to a homotopically trivial map $\tilde{\Psi}\colon X\to\calM_{g,r}$. Note that $X$, being a finite cover of a smooth quasi-projective variety, is itself quasi-projective by a result of Grauert--Remmert~\cite{GR58}. By~\cite[Lemma 3.1]{AAS18}, $\tilde{\Psi}$ is constant, and hence $\Psi$ is isotrivial.
\end{proof}

\p{Chen--Salter's affine twisting} We now explain the distinctions between our notion of affine equivalence and the notion of affine twisting in Chen--Salter~\cite{ChenSalter2026}.  Given holomorphic maps $f\colon\conf{n}\to\conf{m}$ and $A\colon\conf{n}\to\Aff$, Chen--Salter~\cite{ChenSalter2026} define the \defn{affine twist} $f^A\colon\conf{n}\to\conf{m}$ of $f$ by $A$ to be the holomorphic map defined by
\[
  f^A(S) = A(S)\cdot f(S), \quad \text{for all } S\in\conf{n}.
\]
The operation of affine twisting defines an equivalence relation on the set of holomorphic maps $\conf{n}\to\conf{m}$ that is strictly finer than affine equivalence. However, our proof of Theorem~\ref{mainthm:holo} will show, when $n\geq5$ and $m\geq3$, that two holomorphic maps $\conf{n}\to\conf{m}$ that are not affine equivalent to a constant map, are affine equivalent to one another if and only if they are related by an affine twist.

\p{Central equivalence implies affine equivalence} The following result has been shown by the second and fourth authors. However, in this paper we only need to apply it in the case of $m=n$ and where $\Psi_1$ is the identity map, where it was established by Chen--Salter~\cite[Section 3.1]{ChenSalter2026}.

\begin{theorem}[{\cite[Proposition 7.6]{HuxfordSchillewaert2025}}]\label{thm:conf-rigidity}
    Let $n,m\geq3$, and $\Psi_1,\Psi_2\colon\conf{n}\to\conf{m}$ be holomorphic maps that are not affine equivalent to a constant map. If the induced homomorphisms $(\Psi_1)_*,(\Psi_2)_*\colon\braid{n}\to\braid{m}$ are centrally equivalent, then $\Psi_1$ is an affine twist of $\Psi_2$.
\end{theorem}

\p{Orbifold structure of $\conf{m}/\Aff$} When regarded as a complex analytic orbifold, the quotient $\conf{m}/\Aff$ is identical to $\calM_{0,m+1}/\mathfrak{S}_m$, an $(m+1)$-sheeted cover of $\calM_{0,m+1}/\mathfrak{S}_{m+1}$, by e.g.~\cite[Lemmas 2.6 and 2.7]{ChenSalter2026}. Furthermore, the map $\conf{m}\to\conf{m}/\Aff$ induces the quotient homomorphism $\braid{m}\to\braid{m}/Z(\braid{m})$ on orbifold fundamental groups.

\p{Holomorphic rigidity} The fact that Theorem~\ref{mainthm:natleast5} implies Theorem~\ref{mainthm:holo} has already been shown by Chen--Salter~\cite[Theorem 1.5]{ChenSalter2026}. In fact, they show that Theorem~\ref{mainthm:natleast5} implies the following strengthening of Theorem~\ref{mainthm:holo}.

\setcounter{altthm}{0}

\begin{altthm}\label{mainthm:holo-improved}
    Let $n \geq 5$ and $m\geq 3$.  Let  $\Psi:\conf{n} \rightarrow \conf{m}$ be a holomorphic map that is not affine equivalent to a constant map. Then $m=n$ and $\Psi$ is an affine twist of the identity.
\end{altthm}

We include a proof here for the sake of completeness.

\begin{proof}[{Proof of Theorem~\ref{mainthm:holo}} and \altref{mainthm:holo-improved}]
    Since $\Psi$ is not affine equivalent to a constant map, the induced map $\overline{\Psi}\colon\conf{n}\to\conf{m}/\Aff$ is not isotrivial. By Theorems~\ref{thm:hol-implies-irred} and~\ref{thm:hol-implies-noncyclic} the induced homomorphism $\overline{\Psi}_*\colon\braid{n}\to\braid{m}/Z(\braid{m})<\Mod_{0,m+1}$ is irreducible and has non-cyclic image. Therefore $\Psi_*\colon\braid{n}\to\braid{m}$ is irreducible and has non-cyclic image. By Theorem~\ref{mainthm:natleast5}, $m=n$ and $\Psi_*$ is centrally equivalent to the identity. By Theorem~\ref{thm:conf-rigidity}, $\Psi$ is an affine twist of, and hence also affine equivalent to, the identity map.
\end{proof}

Given $n\geq3$ distinct points $x_1,\ldots,x_n$ in $\mathbb{C}$, the affine equation $y^2=(x-x_1)\cdots(x-x_n)$ in $\mathbb{C}^2$ defines a hyperelliptic curve. If $n=2g+1$ then this curve has genus $g$ and has a distinguished point at infinity that is fixed by the hyperelliptic involution. We obtain a holomorphic map
\[
H_g\colon\conf{2g+1}\to\calH_{g,1}.
\]

The above map $H_g$ is not isotrivial. We prove the following.

\begin{theorem}\label{thm:conf-to-hyp}
Let $n\geq5$, $g\geq1$, and $\Psi\colon\conf{n}\to\calH_{g,1}$ be a non-isotrivial holomorphic map. Then $n=2g+1$ and $\Psi$ agrees with $H_g$ when viewed as a map to the coarse moduli space.
\end{theorem}

If $g\geq1$, then there are exactly two holomorphic maps $\conf{2g+1}\to\calH_{g,1}$ that agree with $H_{g}$ when viewed as maps to the coarse moduli space, see~\cite[Section 3.5]{ChenSalter2026} for a discussion of this phenomenon.  The lower bound on $n$ in Theorem~\ref{thm:conf-to-hyp} is once again sharp, since there is a non-isotrivial map $H_1\circ R\colon\conf{4}\to\calH_{1,1}=\calM_{1,1}$. The second and fourth authors have shown an analog of Theorem~\ref{thm:conf-to-hyp} for $n\geq3$ and $g=1$~\cite[Theorem 1.2]{HuxfordSchillewaert2025}.

\p{Orbifold structure of $\calH_{g,1}$} Let $\phi\in\Mod_{g,1}$ be a hyperelliptic involution. By work of Harvey--Maclachlan,~\cite[Proposition 8]{MH75}, the fix set $\Fix(\phi)$ of $\phi$ in Teichm{\"u}ller space $\calT_{g,1}$ is biholomorphic to the Teichm{\"u}ller space $\calT_{0,2g+2}$ of the quotient surface, which is a genus 0 surface with $2g+2$ branch points, one of which is distinguished. The isomorphism
\[
  \Fix(\phi) \cong \calT_{0, 2g+2}
\]
is given by lifting complex structures.  Since hyperelliptic involutions on Riemann surfaces in $\calM_{g,1}$ are unique when they exist, it follows that $\Fix(\phi)$ is a connected component of the preimage of $\calH_{g,1}$ under the orbifold universal covering map $\calT_{g,1}\to\calM_{g,1}$. Since $\Fix(\phi)\cong\calT_{0,2g+2}$ is simply connected, this is the orbifold universal cover of $\calH_{g,1}$. It follows that the inclusion map $\calH_{g,1}\hookrightarrow\calM_{g,1}$ is $\pi_1^{\orb}$-injective, and $\pi_1^{\orb}(\calH_{g,1})$ is precisely the centralizer $C_{\Mod_{g,1}}(\phi)$ of $\phi$ in $\Mod_{g,1}$.

Quotienting by the hyperelliptic involution defines a holomorphic map
\[
  \calH_{g,1}\to\calM_{0,2g+2}/\mathfrak{S}_{2g+1}.
\]
Precomposing this by $H_g$ gives the natural quotient map
\[
  \conf{2g+1}\to\conf{2g+1}/\Aff.
\]
The homomorphism $\pi_1^{\orb}(\calH_{g,1})\to\pi_1^{\orb}(\calM_{0,2g+2}/\mathfrak{S}_{2g+1})\cong\braid{2g+1}/Z(\braid{2g+1})$ has kernel $\langle\phi\rangle\cong\Z/2\Z$. Furthermore, it follows from Birman--Hilden that the homomorphism $(H_g)_*\colon\braid{2g+1}\to\pi_1^{\orb}(\calH_{g,1})$ is surjective~\cite{BirmanHilden}. Therefore,
\[
  \pi_1^{\orb}(\calH_{g,1})\cong\braid{2g+1}/Z(\braid{2g+1})^2
\]
and $(H_g)_*$ is the obvious quotient map.  A well known special case of this occurs when $g=1$, where $\calH_{1,1}=\calM_{1,1}$. Famously, $\pi_1^{\rm orb}(\calM_{1,1})\cong\SL_2(\ZZ)\cong\braid{3}/Z(\braid{3})^2$, whereas $\pi_1^{\rm orb}(\calM_{0,4}/\mathfrak{S}_3)\cong\PSL_2(\ZZ)\cong\braid{3}/Z(\braid{3})$.

We have the following auxiliary lemma that we use in the proof of Theorem~\ref{thm:conf-to-hyp}.

\begin{lemma}\label{lem:holo:lift}
Let $n,m\geq3$ and $\overline{\Phi}:\braid{n} \rightarrow \braid{m}/Z(\braid{m})$ be a homomorphism.  There is a lift of $\overline{\Phi}$ to a map $\Phi:\braid{n} \rightarrow \braid{m}$.
\end{lemma}

\begin{proof}
  For each $s_i \in \extgen{n} \setminus \{s_n\}$ choose $\sigma_i$ a lift of $\overline{\Phi}(s_i)$ to $\braid{m}$.  We define elements $\tau_i \in \braid{m}$ as follows.  For $i = 1$, set $\tau_i = \sigma_i$.  For $2 \leq i \leq n -1$, we recursively define
  \[
    \tau_i = \tau_{i-1}\sigma_i\tau_{i-1} \sigma_i^{-1}\tau_{i-1}^{-1}.
  \]
  Note that $\tau_i$ and $\sigma_i$ have the same image in $\braid{m}/Z(\braid{m})$.  By construction, $\tau_i$ and $\tau_{i-1}$ satisfy the braid relation for $2 \leq i \leq n-1$.  Likewise, we see that $[\tau_i, \tau_j] \in Z(\braid{m})$ for $[s_i,s_j] = 1$. Note also that $[\braid{m}, \braid{m}] \cap Z(\braid{m}) = \{1\}$ since the composition $Z(\braid{m}) \hookrightarrow \braid{m} \rightarrow \braid{m}^{\ab} \cong \ZZ$ is injective, so $[\tau_i, \tau_j] = 1$.  In particular, the map $\Phi(s_i) = \tau_i$ for $s_i \in \extgen{n} \setminus \{s_n\}$ defines our desired homomorphism.
\end{proof} We are now ready to prove Theorem~\ref{thm:conf-to-hyp}.

\begin{proof}[{Proof of Theorem~\ref{thm:conf-to-hyp}}]
    Let $g\geq1$, $n\geq5$, and $\Psi\colon\conf{n}\to\calH_{g,1}$ be a non-isotrivial holomorphic map. By further quotienting by the hyperelliptic involution, we obtain a holomorphic map $\bar{\Psi}\colon\conf{n}\to\calM_{0,2g+2}/\mathfrak{S}_{2g+1}$. By Theorems~\ref{thm:hol-implies-irred} and~\ref{thm:hol-implies-noncyclic}, the induced homomorphism $\overline{\Psi}_*\colon\braid{n}\to\braid{2g+1}/Z(\braid{2g+1})<\Mod_{0,2g+2}$ is irreducible and has non-cyclic image.
    
    Let $\Phi\colon\braid{n} \rightarrow \braid{2g+1}$ denote a lift of $\overline{\Psi}_*$ from Lemma~\ref{lem:holo:lift}. Recall that the quotient map $\braid{2g+1}\to\braid{2g+1}/Z(\braid{2g+1})<\Mod_{0,2g+2}$ is induced by capping. Therefore, the fact that $\overline{\Psi}_*$ is irreducible implies that $\Phi$ is also irreducible. By Theorem~\ref{mainthm:natleast5} we have $n=2g+1$ and $\Phi$ is centrally equivalent to the identity. Therefore, $\overline{\Psi}_*$ is, up to post-composition by inner automorphisms, either the quotient map $\braid{n}\to\braid{n}/Z(\braid{n})$ or the pre-composition of this with an outer automorphism of $\braid{n}$. The former of these homomorphisms is realized by the holomorphic quotient map $\conf{n}\to\conf{n}/\Aff$ discussed above, while the latter is realized by the anti-holomorphic map obtained by pre-composing this with complex conjugation. Since $\overline{\Psi}$ is holomorphic, by Theorem~\ref{thm:moduli-rigidity} it must agree with the holomorphic quotient map. Hence $\Psi$ agrees with $H_g$ on the level of coarse spaces.
\end{proof}

\p{Maps between hyperelliptic moduli spaces} Note that there are exactly two holomorphic maps $\calH_{g,1}\to\calH_{g,1}$ that agree with the identity as maps of coarse moduli spaces. Precomposing them with $H_g$ gives the two holomorphic maps $\conf{2g+1}\to\calH_{g,1}$ that agrees with $H_g$ as a map to the coarse moduli space.

We now prove Theorem~\ref{mainthm:hyp-to-hyp}, which we restate here for the reader's convenience.

\begin{mainthm*}[\textbf{\ref{mainthm:hyp-to-hyp}}]
    Let $g\geq2$ and $h\geq1$. Let $\Psi\colon\calH_{g,1}\to\calH_{h,1}$ be a non-isotrivial holomorphic map. Then $g=h$ and $\Psi$ agrees with the identity when viewed as a map between the coarse moduli spaces.
\end{mainthm*}

\begin{proof}
    The composition $\Psi\circ H_g\colon\conf{2g+1}\to\calH_{h,1}$ is also holomorphic, and is non-isotrivial since $H_g$ is surjective. By Theorem~\ref{thm:conf-to-hyp} we must have $g=h$ and $\Psi\circ H_g$ agrees with $H_g$ as a map to the coarse moduli space. Since $H_g$ is surjective, $\Psi$ agrees with the identity as a map between coarse spaces.
\end{proof}

We leverage Theorem~\ref{mainthm:hyp-to-hyp} to prove the Theorem~\ref{thm:mod-to-mod}.  This resolves Question~\ref{question:moduli} for $r=s=1$ and $g\geq2$ under the additional assumption that $\Psi$ sends the hyperelliptic locus to the hyperellptic locus.  Note that the conclusion does not mention coarse moduli spaces. 

\begin{theorem}\label{thm:mod-to-mod}
Let $g\geq3$ and $h\geq1$. Let $\Psi\colon\calM_{g,1}\to\calM_{h,1}$ be a non-constant holomorphic map. Suppose that $\Psi(\calH_{g,1})\subset \calH_{h,1}$. Then $g=h$ and $\Psi$ is the identity.
\end{theorem}

\begin{proof}
    Let $n=2g+1$. Assume first for a contradiction that the restriction of $\Psi$ to $\calH_{g,1}$ is isotrivial, i.e.\ that $\Psi\circ H_g$ is isotrivial. This means that $(\Psi\circ H_g)_*(\braid{n})$ is a finite subgroup of $\Mod_{h,1}$, as it acts via monodromy by automorphisms on some Riemann surface in $\calM_{h,1}$. However, any such automorphism fixes the marked point, hence $(\Psi\circ H_g)_*$ is cyclic. Therefore, $(\Psi\circ H_g)_*(s_1)=(\Psi\circ H_g)_*(s_2)$, since $s_1$ and $s_2$ have the same image in $B_n^{\ab}\cong\Z$.

    Birman--Hilden proved that for each $i\in\{1,\ldots,n-1\}$ the mapping class $(H_g)_*(s_i)=T_{\alpha_i}$ is a Dehn twist about a non-separating curve $\alpha_i$. Moreover, the geometric intersection numbers of the $\alpha_i$ satisfy $\iota(\alpha_i,\alpha_{i+1})=1$ for $1\leq i<n-1$, and $\iota(\alpha_i,\alpha_j)=0$ if $[s_i,s_j]=1$~\cite{BH71}. Humphries showed that there is a non-separating curve $\beta$ in $S_{g,1}$ satisfying $\iota(\beta,\alpha_4)=1$ and $\iota(\beta,\alpha_i)=0$ for $i\in\{1,\ldots,n-1\}\setminus\{4\}$, such that $T_{\alpha_1},\ldots,T_{\alpha_{n-1}},T_\beta$ generate $\Mod_{g,1}$~\cite{Hum79}, see also~\cite[Section 4.4.4]{FM}. We have shown that $\Psi_*(T_{\alpha_1})=\Psi_*(T_{\alpha_2})$. Hence $T_{\alpha_1}T_{\alpha_2}^{-1}$ lies in the kernel of $\Psi_*$. By the change of coordinates principle, $T_{\beta_1}T_{\beta_2}^{-1}$ lies in the kernel of $\Psi_*$ for any non-separating simple closed curves $\beta_1$ and $\beta_2$ satisfying $\iota(\beta_1,\beta_2)=1$. Hence $\Psi_*(T_{\alpha_1})=\cdots=\Psi_*(T_{\alpha_{n-1}})=\Psi_*(T_\beta)$. Therefore, $\Psi_*$ has cyclic image. However, the abelianization of $\Mod_{g,1}$ is trivial if $g\geq3$. Hence $\Psi_*$ is trivial, and so $\Psi$ is constant by~\cite[Lemma 3.1]{AAS18}, which contradicts our hypothesis.

    Therefore the restriction of $\Psi$ to $\calH_{g,1}$, and the composition $\Psi\circ H_g$, are both non-isotrivial. By Theorem~\ref{mainthm:hyp-to-hyp} this implies that $g=h$. The first author and Souto show that if $g\geq4$ then any non-constant holomorphic map $\Psi\colon\calM_{g,1}\to\calM_{g,1}$ is the identity~\cite[Theorem 1.1]{DPS24}. Here we provide an argument that reaches this conclusion for all $g\geq3$ under the additional hypothesis that $\Psi(\calH_{g,1})\subset\calH_{g,1}$.
    
    The conclusion of Theorem~\ref{thm:conf-to-hyp} tells us that $\Psi\circ H_g$, viewed as a map $\conf{n}\to\calH_{g,1}$, agrees with $H_g$ on the level of coarse spaces. In fact, the proof tells us that by post-composing this with the map that quotients by the hyperelliptic involution, one obtains the quotient map $\conf{n}\to\conf{n}/\Aff$. This implies that there are one of two possibilities, up to precomposing with an inner automorphism, for the induced homomorphism
    \[
      (\Psi\circ H_g)_*\colon\braid{n}\to\braid{n}/Z(\braid{n})^2<\Mod_{g,1}.
    \]
    Namely, either $(\Psi\circ H_g)_*=(H_g)_*$, or $(\Psi\circ H_g)_*(s_i)=(H_g)_*(s_iT_{\partial\disk{n}})$ for all $s_i\in\extgen{n}$. We seek to rule out the latter possibility. Assume for a contradiction that $(\Psi\circ H_g)_*(s_i)=(H_g)_*(s_iT_{\partial\disk{n}})$ for all $s_i\in\extgen{n}$.

    Let $\alpha_1,\ldots,\alpha_{n-1},\beta$ be curves as above. Note that $(H_g)_*(T_{\partial\disk{n}})=\phi$ is the hyperelliptic involution. Our assumption implies $\Psi_*(T_{\alpha_1})=T_{\alpha_1}\phi$. Note that $\Psi_*(\phi)=\phi$, since $T_{\partial\disk{n}}=(s_1\cdots s_{n-1})^n$ is a word of even length in the standard generators of $\braid{n}$. In particular, the image of $\Psi_*$ contains $\cent{\phi}{\Mod_{g,1}}=\braid{n}/Z(\braid{n})^2<\Mod_{g,1}$. We now prove that there is nothing else in the image of $\Psi_*$.
    
    Suppose that $h\in\Mod_{g,1}$ commutes with $T_{\alpha_1}$. Then $\Psi_*(h)$ commutes with both $\Psi_*(T_{\alpha_1})=T_{\alpha_1}\phi$ and its square $T_{\alpha_1}^2$. By Lemmas~\ref{lem:crs:commdisj} and \ref{lem:crs:power} we have $\Psi_*(h)(\alpha_1)=\alpha_1$. Hence $\Psi_*(h)T_{\alpha_1}\Psi_*(h)^{-1}=T_{\Psi_*(h)(\alpha_1)}=T_{\alpha_1}$, i.e. $[\Psi_*(h),T_{\alpha_1}]=1$. Hence $[\Psi_*(h),\phi]=1$. If we apply this reasoning to $h=T_{\beta}$, then we conclude that $\Psi_*(T_\beta)\in\cent{\phi}{\Mod_{g,1}}$. Hence $\Psi_*(\Mod_{g,1})=\cent{\phi}{\Mod_{g,1}}=\braid{n}/Z(\braid{n})^2$. However, $\Mod_{g,1}$  has trivial abelianization if $g\geq 3$ and $\braid{n}/Z(\braid{n})^2$ has non-trivial abelianization $\Z/2\Z$, which is a contradiction.  Therefore $(\Psi\circ H_g)_*$ agrees with $(H_g)_*$. Hence $\Psi_*$ maps non-separating Dehn twists to Dehn twists. By~\cite[Theorem 3.2]{DP25}, $\Psi_*$ is an automorphism, so by Theorem~\ref{thm:moduli-rigidity}, $\Psi$ is the identity map.
\end{proof}

\printbibliography{}
\end{document}